\documentclass[10pt]{amsart}
      \usepackage[mathscr]{eucal}
      \usepackage{amsmath,amsfonts}
      \parskip\smallskipamount
          \newtheorem{theorem}{Theorem}[section]
      
      \newtheorem{proposition}[theorem]{Proposition}
      \newtheorem{corollary}[theorem]{Corollary}
      \newtheorem{lemma}[theorem]{Lemma}

      \makeatletter
      \@addtoreset{equation}{section}
      \makeatother

      \newcommand{\BB}{{\mathbb B}}
      \newcommand{\CC}{{\mathbb C}}
      \newcommand{\NN}{{\mathbb N}}

      \newcommand{\DD}{{\mathbb D}}
      \newcommand{\RR}{{\mathbb R}}
      \newcommand{\FF}{{\mathbb F}}
      \newcommand{\TT}{{\mathbb T}}

      \newcommand{\cA}{{\mathcal A}}

      \newcommand{\cG}{{\mathcal G}}
      \newcommand{\cH}{{\mathcal H}}
      \newcommand{\cK}{{\mathcal K}}
      
      \newcommand{\cM}{{\mathcal M}}
      \newcommand{\cN}{{\mathcal N}}
      \newcommand{\cP}{{\mathcal P}}
      \newcommand{\cR}{{\mathcal R}}

      \newcommand{\cU}{{\mathcal U}}
      \newcommand{\cV}{{\mathcal V}}

      \newdimen\expt
      \def\boxit#1{\setbox0\hbox{$\displaystyle{#1}$}
            \hbox{\lower.4\expt
       \hbox{\lower3\expt\hbox{\lower\dp0
            \hbox{\vbox{\hrule height.4\expt
       \hbox{\vrule width.4\expt\hskip3\expt
            \vbox{\vskip3\expt\box0\vskip2\expt}%
       \hskip3\expt\vrule width.4\expt}\hrule height.4\expt}}}}}}
      
\begin{document}
       \pagestyle{myheadings}
      \markboth{ Gelu Popescu}{   Composition operators on   noncommutative  Hardy spaces  }

      \title [   Composition operators on   noncommutative  Hardy spaces  ]
      {     Composition operators on   noncommutative  Hardy spaces
      }
        \author{Gelu Popescu}
\date{April 14, 2010}
      \thanks{Research supported in part by an NSF grant}

      \address{Department of Mathematics, The University of Texas
      at San Antonio \\ San Antonio, TX 78249, USA}
      \email{\tt gelu.popescu@utsa.edu}

\bigskip

\begin{abstract}
 In this paper we initiate the study of  composition operators  on the noncommutative Hardy space
  $H^2_{\bf ball}$, which is  the Hilbert space
of all free holomorphic functions
 of the
form
$$f(X_1,\ldots, X_n)=\sum_{k=0}^\infty \sum_{|\alpha|=k}
a_\alpha  X_\alpha, \qquad   \sum_{\alpha\in
\FF_n^+}|a_\alpha|^2<\infty,
$$
where the convergence is in the operator norm  topology for all
$(X_1,\ldots, X_n)$ in the noncommutative operatorial  ball
$[B(\cH)^n]_1$ and  $B(\cH)$ is the algebra of all bounded linear
operators on  a Hilbert space $\cH$.
  When the symbol
$\varphi$ is a free holomorphic self-map of $[B(\cH)^n]_1$, we show
that the composition operator
$$C_\varphi f:=f\circ\varphi,\qquad f\in H^2_{\bf ball},$$
is
bounded on $H^2_{\bf ball}$. Several classical results about
 composition operators (boundedness, norm estimates, spectral properties, compactness, similarity)
  have free analogues in our noncommutative multivariable setting.
The most prominent  feature of this paper is the interaction between
 the
noncommutative  analytic function theory in the unit ball of
$B(\cH)^n$, the  operator  algebras generated by the left creation
operators   on the full Fock space with $n$ generators, and the
classical complex function theory in the unit ball of $\CC^n$.

In a more general setting, we establish basic properties concerning the
  composition
operators acting on Fock spaces  associated with noncommutative varieties
$\cV_{\cP_0}(\cH)\subseteq [B(\cH)^n]_1$ generated by sets $\cP_0$
of  noncommutative polynomials in $n$ indeterminates  such  that
$p(0)=0$, $p\in \cP_0$.
In particular, when
$\cP_0$ consists of the commutators $X_iX_j-X_jX_i$ for $i,j=1,\ldots,n$, we show that
 many of our  results have
commutative counterparts  for composition operators on the symmetric
Fock space  and, consequently,  on  spaces of analytic functions in
the unit ball of $\CC^n$.
\end{abstract}

      \maketitle

\bigskip
\bigskip

\section*{Contents}
{\it

\quad Introduction

\begin{enumerate}
   \item[1.]     Noncommutative  Littlewood subordination principle
   \item[2.]   Composition operators on  the noncommutative Hardy
   space $H^2_{\bf ball}$
   \item[3.]  Noncommutative Wolff theorem for free holomorphic self-maps of $[B(\cH)^n]_1$
 \item[4.]    Composition operators and their adjoints
\item[5.]   Compact  composition operators on $H^2_{\bf ball}$
\item[6.]  Schr\" oder equation  for noncommutative  power series and spectra  of composition operators
 \item[7.]    Composition operators on Fock spaces associated to noncommutative
 varieties
 \end{enumerate}

\quad References

}

\newpage

\bigskip

\section*{Introduction}

An important consequence of Littlewood's subordination principle (\cite{L}, \cite{Du})
is the boundedness of the composition operator $C_\varphi$ on the Hardy space $H^2(\DD)$,
when $\varphi:\DD\to \DD$ is an analytic self-map of the open unit disc $\DD:=\{z\in \CC:\ |z|<1\}$
 and $C_\varphi f:=f\circ\varphi$. This result was the starting point  of the modern  theory
 of composition operators on spaces of analytic functions, which has been developed since the
 $1960$'s through the fundamental work of
Ryff (\cite{Ry}), Nordgren (\cite{Nor1}, \cite{Nor2}), Schwartz
(\cite{Schw}), Shapiro (\cite{Sha}), Cowen (\cite{Cow}) and many
others (see \cite{Sha-book}, \cite{CoMac}, \cite{BS}, and the
references therein).  They  answered basic questions about
 composition operators such as boundedness, compactness, spectra, cyclicity,  revealing
a beautiful   interaction  between   operator theory and complex
function  theory. In the multivariable setting, when $\varphi$ is a
holomorphic self-map of the open unit ball $$\BB_n:=\{z=(z_1,\ldots,
z_n)\in \CC^n:\ \|z\|_2<1\},$$ the composition operator $C_\varphi$
is  no longer  a bounded operator on the Hardy space  $H^2(\BB_n)$.
However, significant work was done concerning the spectra of
automorphism-induced composition operators and compact  composition
operators on $H^2(\BB_n)$ by MacCluer (\cite{MC1}, \cite{MC2},
\cite{MC3}) and others (see \cite{CoMac} and its references). The
study of composition operators on the Hardy space $H^2(\BB_n)$ is
close connected to the several variable function theory in the
 unit ball of $\CC^n$ (\cite{Ru2}). There is an extensive  literature on composition
  operators on  other spaces of analytic functions in  several variables (see \cite{CoMac}).

For the interested reader, we mention two  very nice  books on composition operators:
 Shapiro's monograph \cite{Sha-book}, which is an excellent account of composition operators
  on $H^2(\DD)$  and    the monograph \cite{CoMac} by Cowen and MacCluer,  which
  is a comprehensive
treatment of
composition operators on spaces of analytic functions in one or several variables.

It is our hope that  the present paper will  open a new chapter  in the theory of composition
 operators. The goal is to   initiate the study of composition operators on the
noncommutative Hardy space $H^2_{\bf ball}$ (which will be
introduced shortly)  and, more generally,  on
   subspaces of the full Fock space with $n$ generators  associated to noncommutative
 varieties.
The most prominent  feature of this paper is the interplay between
 the
noncommutative  analytic function theory in the unit ball of
$B(\cH)^n$, the  operator  algebras generated by the left creation
operators $S_1,\ldots, S_n$  on the full Fock space with $n$
generators: the Cuntz-Toeplitz algebra $C^*(S_1,\ldots, S_n)$
(\cite{Cu}),  the noncommutative disk algebra $\cA_n$ and the
analytic Toeplitz algebra $F_n^\infty$  (\cite{Po-von},
\cite{Po-funct},  \cite{Po-analytic}, \cite{Po-disc}), as well as
the classical function theory in the unit ball of $\CC^n$
(\cite{Ru2}).
  To present our results we need some notation and preliminaries on free holomorphic functions.

Initiated in \cite{Po-holomorphic}, the theory of free holomorphic
(resp.~pluriharmonic) functions on the unit ball of $B(\cH)^n$,
where $B(\cH)$ is the algebra of all bounded linear operators on  a
Hilbert space $\cH$,
   has been developed   very
recently (see   \cite{Po-pluri-maj}, \cite{Po-unitary},
\cite{Po-hyperbolic}, \cite{Po-pluriharmonic},
  \cite{Po-automorphism}, \cite{Po-holomorphic2}) in the attempt  to
provide a framework for the study of arbitrary
 $n$-tuples of operators on a Hilbert space.
  Several classical
 results from complex analysis and  hyperbolic geometry   have
 free analogues in
 this  noncommutative multivariable setting. Related to our work, we
 mention  the papers \cite{HKMS},   \cite{MuSo2},
 \cite{MuSo3}, and  \cite{V}, where several aspects of the theory
 of noncommutative analytic functions are considered in various
 settings.
We recall that  the algebra $H_{{\bf ball}}$~ of free holomorphic
functions on the open operatorial  $n$-ball of radius one is defined
 as the set of all power series $\sum_{\alpha\in
\FF_n^+}a_\alpha Z_\alpha$ with radius of convergence $\geq 1$,
i.e.,
 $\{a_\alpha\}_{\alpha\in \FF_n^+}$ are complex numbers  with
$\limsup_{k\to\infty} \left(\sum_{|\alpha|=k}
|a_\alpha|^2\right)^{1/2k}\leq 1,$
 where $\FF_n^+$ is the
free semigroup with $n$ generators  $g_1,\ldots, g_n$ and the identity $g_0$.  The length
of $\alpha\in \FF_n^+$ is defined by $|\alpha|:=0$ if $\alpha=g_0$
and $|\alpha|:=k$ if
 $\alpha=g_{i_1}\cdots g_{i_k}$, where $i_1,\ldots, i_k\in \{1,\ldots, n\}$.
If $(X_1,\ldots, X_n)\in B(\cH)^n$, we denote $X_\alpha:=
X_{i_1}\cdots X_{i_k}$ and $X_{g_0}:=I_\cH$.  A free holomorphic
function on  the open ball
$$
[B(\cH)^n]_1:=\left\{ (X_1,\ldots, X_n)\in B(\cH)^n: \
\|X_1X_n^*+\cdots + X_nX_n^*\|^{1/2}<1\right\},
$$
 is
the representation of an element $f\in H_{{\bf ball}}$~ on
the Hilbert space $\cH$, that is,  the mapping
$$
[B(\cH)^n]_1\ni (X_1,\ldots, X_n)\mapsto f(X_1,\ldots,
X_n):=\sum_{k=0}^\infty \sum_{|\alpha|=k}
 a_\alpha X_\alpha\in
B(\cH),
$$
where  the convergence is in the operator norm topology.
   Due to the fact that a free holomorphic function is
uniquely determined by its representation on an infinite dimensional
Hilbert space,  throughout this paper, we  identify  a free
holomorphic function   with its representation on a   separable
infinite dimensional Hilbert space.

 A free holomorphic function $f$ on $[B(\cH)^n]_1$ is
bounded if $
 \|f\|_\infty:=\sup  \|f(X)\|<\infty,
  $
where the supremum is taken over all $X\in [B(\cH)^n]_1$ and $\cH$
is an infinite dimensional Hilbert space. Let $H^\infty_{\bf ball}$
be the set of all bounded free holomorphic functions and  let
$A_{\bf ball}$ be the set of all elements $f\in H^\infty_{\bf ball}$
such that the mapping
$$[B(\cH)^n]_1\ni (X_1,\ldots, X_n)\mapsto f(X_1,\ldots, X_n)\in B(\cH)$$
 has a continuous extension to the closed unit ball $[B(\cH)^n]^-_1$.
We  showed in \cite{Po-holomorphic} that $H^\infty_{\bf ball}$  and
$A_{\bf ball}$ are Banach algebras under pointwise multiplication
and the norm $\|\cdot \|_\infty$, which can be identified, via the
noncommutative Poisson transform (\cite{Po-poisson}),
 with the
noncommutative analytic Toeplitz algebra $F_n^\infty$ and the
noncommutative disc algebra $\cA_n$, respectively.

 If $f:[B(\cH)^n]_1\to B(\cH)$  and
$\varphi:[B(\cH)^n]_1\to [B(\cH)^n]_1$ are  free holomorphic
functions then $f\circ \varphi$ is a free holomorphic function on
$[B(\cH)^n]_1$ (see \cite{Po-automorphism}), defined by
$$
(f\circ \varphi)(X_1,\ldots, X_n)=\sum_{k=0}^\infty
\sum_{|\alpha|=k}  a_\alpha \varphi_\alpha (X_1,\ldots, X_n),\qquad
(X_1,\ldots, X_n)\in [B(\cH)^n]_1,
$$
where $\varphi=(\varphi_1,\ldots, \varphi_n)$ and the convergence is
in the operator norm topology. The noncommutative Hardy space
$H^2_{\bf ball}$ is the Hilbert space of all free holomorphic
functions on $ [B(\cH)^n]_1 $
 of the
form
$$f(X_1,\ldots, X_n)=\sum_{k=0}^\infty \sum_{|\alpha|=k}
a_\alpha  X_\alpha, \qquad   \sum_{\alpha\in
\FF_n^+}|a_\alpha|^2<\infty,
$$ with the inner product
$ \left< f,g\right>:=\sum_{k=0}^\infty \sum_{|\alpha|=k}a_\alpha
{\overline b}_\alpha,$ where  $g=\sum_{k=0}^\infty \sum_{|\alpha|=k}
b_\alpha  X_\alpha$ is
 another free holomorphic function  in $H^2_{\bf ball}$.
 The main question that we answer in this paper is whether $f\circ
 \varphi\in H^2_{\bf ball}$ for any $f\in H^2_{\bf ball}$ and whether the
 corresponding composition operator is bounded. This will be the
 starting point in   our attempt to develop a theory of compositions operators on
 noncommutative Hardy spaces. We are interested in extracting properties of the composition
operator $C_\varphi$ (boundedness, spectral properties, compactness)
from the operatorial or dynamical properties of the model boundary
function $\widetilde\varphi:=\text{SOT-}\lim_{r\to1}
\varphi(rS_1,\ldots, rS_n)\in F_n^\infty\otimes \CC^n$ or the scalar
representation of $\varphi$, i.e., the holomorphic  function
$\BB_n\ni \lambda\mapsto \varphi(\lambda)\in \BB_n$.

In Section 1, we characterize the free holomorphic self-maps of
$[B(\cH)^n]_1$ in terms of the model boundary functions with respect
to the left creation operators on the full Fock space $F^2(H_n)$.
This will be used, together with  the natural identification of
$H_{\bf ball}^2$ with $F^2(H_n)$, to  provide a  noncommutative
Littlewood subordination theorem for the Hardy space $H^2_{\bf
ball}$. More precisely, we show that if
 $\varphi$  is   a   free holomorphic self-map of  the ball
$[B(\cH)^n]_1$ such that $\varphi(0)=0$ and  $f\in H^2 _{\bf ball}$,
then  $ f\circ \varphi \in H^2_{\bf ball}$   and $\|f\circ
\varphi\|_2\leq \|f\|_2$.

Section 2 contains the core material on boundedness of compositions
operators on the noncommutative Hardy space $H^2_{\bf ball}$  and
estimates for their norms. An important role in our investigation
will be played by the characterization of  $H^2_{\bf ball}$ in terms
of  pluriharmonic majorants (\cite{Po-pluri-maj})  and the
Herglotz-Riesz type representation  for positive free pluriharmonic
functions (\cite{Po-pluriharmonic}). The key result of this section
asserts that if  $\varphi$ is a   free holomorphic automorphism of
the  noncommutative ball $[B(\cH)^n]_1$ (see
\cite{Po-automorphism}), then
$$
\left(\frac{1-\|\varphi(0)\|}{1+\|\varphi(0)\|}\right)^{1/2}\|f\|\leq
\|C_\varphi f\|\leq
\left(\frac{1+\|\varphi(0)\|}{1-\|\varphi(0)\|}\right)^{1/2}\|f\|
$$
for any $f\in H^2_{\bf ball}$. Moreover,  these inequalities are
best possible and we have a formula for the norm of $C_\varphi$.
Combining this result with the noncommutative Littlewood
subordination theorem from the previous section,  we obtain the main
result which asserts that, for any    free holomorphic self-map
$\varphi$ of $[B(\cH)^n]_1$, the composition $C_\varphi f:=f\circ
\varphi$ is a bounded operator on $ H^2_{\bf ball}$ and
$$
\frac{1}{(1-\|\varphi(0)\|^2)^{1/2}}\leq \|C_\varphi\|\leq
\left(\frac{1+\|\varphi(0)\|}{1-\|\varphi(0)\|}\right)^{1/2}.
$$
This leads to an extension of Cowen's (\cite{Cow}) one-variable
spectral radius formula for composition operators   to our
noncommutative multivariable setting. More precisely, we obtain
$$r(C_\varphi)=\lim_{k\to\infty}(1-\|\varphi^{[k]}(0)\|)^{-1/2k}, $$
where $\varphi^{[k]}$ is the $k$-iterate of $\varphi$. Another
consequence of the above-mentioned result is that $C_\varphi$ is
similar to a contraction if and only if
 there is $\xi\in \BB_n$ such that $\varphi(\xi)=\xi$. This will
 also show that similarity   of composition operators on $H^2_{\bf
 ball}$ to contractions is equivalent to power (resp.~ polynomial) boundedness.
 This is interesting in light of Pisier's (\cite{Pi}) famous example
 of a polynomially bounded operator which is not similar to a
 contraction, and Paulsen's  (\cite{Pa}) result that every
 completely polynomially bounded operator is  similar to a
 contraction. For more information on similarity problems we refer
 the reader
 to \cite{Pa-book} and \cite{Pi-book}.

In Section 3, extending  the classical result obtained by Wolff
(\cite{Wo1}, \cite{Wo2}) and  MacCluer's version for  $\BB_n$ (see
\cite{MC1}), we provide a noncommutative analogue of  Wolff's
theorem for free holomorphic self-maps of $[B(\cH)^n]_1$. We show
that if
  $\varphi:[B(\cH)^n]_1\to [B(\cH)^n]_1$ is a free
holomorphic function such that its scalar representation  has  no
fixed points in $\BB_n$, then there is a unique point $\zeta\in
\partial \BB_n$ (the Denjoy-Wolff point of $\varphi$)  such that  each  noncommutative ellipsoid ${\bf
E}_c(\zeta)$ (see Section 3 for the definition)  is mapped into
 itself by every iterate of the symbol  $\varphi$. We also show that
  the spectral radius of a composition  operator   on $H^2_{\bf ball}$  is $1$ when  the symbol
   is  elliptic or parabolic, which extends some of Cowen's results
   (\cite{Cow}) from the single variable case.

In Section 4, we obtain a formula for  the adjoint of  a composition
operator  on $H^2_{\bf ball}$. It is shown that if
$\varphi=(\varphi_1, \ldots, \varphi_n)$  is a   free holomorphic
self-map of  the noncommutative ball $[B(\cH)^n]_1$, then
$$
C_\varphi^*f =\sum_{\alpha\in \FF_n^+} \left<
f,\varphi_\alpha\right> e_\alpha,
$$
where  $f$ and $\varphi_1, \ldots, \varphi_n$ are seen as elements
of the Fock space $F^2(H_n)$. As a consequence we prove that
$C_\varphi$ is normal   if and only if
$$
\varphi(X_1,\ldots, X_n)=[X_1,\ldots, X_n] A
$$
for some normal scalar matrix   $A\in M_{n\times n}$  with
$\|A\|\leq 1$, where $[X_1,\ldots, X_n]$ is seen as a row
contraction. This
 leads to characterizations of self-adjoint or
 unitary composition operators  on $H^2_{\bf ball}$. A nice
 connection between Fredholm composition operators on $H^2_{\bf
 ball}$ and the automorphisms of the open unit ball $\BB_n$ is also
 presented.

 In Section 5, we study compact composition operators  on the noncommutative Hardy space
  $H^2_{\bf ball}$.  Using  some of Shapiro's arguments from the single
  variable case (see \cite{Sha}) in our setting  as well as  some  results  from Section 4,
   we obtain a formula for the
  essential norm of the
composition operator $C_\varphi$ on $H^2_{\bf ball}$. In particular,
this implies that $C_\varphi$  is a compact operator if and only if
$$
\lim_{k\to\infty} \sup_{f\in H^2_{\bf ball}, \|f\|_2\leq 1}
\sum_{|\alpha|\geq k}|\left<f,\varphi_\alpha\right>|^2 =0.
$$
Moreover, we show that if $C_\varphi$ is  a compact  operator on
$H^2_{\bf ball}$, then the scalar representation of $\varphi$ is a
holomorphic self-map of $\BB_n$ which
\begin{enumerate}
\item[(i)]
 cannot have
 finite angular derivative at any point of $\partial \BB_n$, and
 \item[(ii)]has exactly one fixed
point in the open ball $\BB_n$.
\end{enumerate}
As a consequence, we deduce that every  compact composition operator
on $H^2_{\bf ball}$ is similar to a contraction. In the end of this
section, we prove that the  set of compact composition operators on
$H^2_{\bf ball}$ is arcwise connected in  the set    of all
composition operators with respect to the operator norm topology.

In Section 6,   we consider  a noncommutative multivariable
extension of Schr\" oder  equation (\cite{Schr}) which is used to
obtain results concerning the spectrum of composition operators on
$H^2_{\bf ball}$ (see Theorem \ref{inclusions}). Combining these
results with those from Section 5,  we determine the spectra of
compact composition operators on $H^2_{\bf ball}$. More precisely,
if  $\varphi$ is a free holomorphic self-map of  the noncommutative
ball $[B(\cH)^n]_1$ and $C_\varphi$ is a compact composition
operator on $H^2_{\bf ball}$, then   the scalar representation of
$\varphi$ has a unique fix point $\xi\in \BB_n$  and   the spectrum
$ \sigma(C_\varphi)$ consists of \ $0$, $1$, and
 all
possible products of the eigenvalues of the matrix
$$
\left[\left<\psi_i,e_j\right>\right]_{n\times n},
$$
 where $\psi=(\psi_1,\ldots, \psi_n):=\Phi_\xi\circ \varphi \circ
\Phi_\xi$ and  $\Phi_\xi$ is the involutive free holomorphic
automorphism of $[B(\cH)^n]_1$ associated with $\xi$, the functions
$\psi_1, \ldots, \psi_n$ are seen as elements of the Fock space
$F^2(H_n)$, and   the Hilbert space $H_n$ has $e_1$, $e_2$, $\dots, e_n$ as   orthonormal
basis.

In Section 7, we consider composition operators on Fock spaces
associated to noncommutative varieties in unit ball $[B(\cH)^n]_1$.
 Given a  set $\cP_0$ of noncommutative polynomials in
$n$ indeterminates  such  that $p(0)=0$, $p\in \cP_0$,  we define a
noncommutative variety $\cV_{\cP_0}(\cH)\subseteq [B(\cH)^n]_1$  by
setting
$$
\cV_{\cP_0}(\cH):=\{(X_1,\ldots, X_n)\in [B(\cH)^n]_1:\
p(X_1,\ldots, X_n)=0 \text{ for all } p\in \cP_0\}.
$$
   According to \cite{Po-varieties}, there  is  a universal model $(B_1,\ldots,
B_n)$ associated with  the noncommutative variety
$\cV_{\cP_0}(\cH)$, where $B_i=P_{\cN_{\cP_0}} S_i|_{\cN_{\cP_0}}$
and $\cN_{\cP_0}$ is a subspace of the full Fock space $F^2(H_n)$.
 Let $F_n^\infty
(\cV_{\cP_0})$ be the $w^*$-closed algebra generated by $B_1,\ldots,
B_n$ and the identity. Using the results from Section 2 and the
noncommutative commutant lifting theorem \cite{Po-isometric} (see
\cite{SzF-book} for the classical case $n=1$), we show that given
any $\widetilde\psi\in F_n^\infty (\cV_{\cP_0})\otimes \CC^n$  with
$\|\widetilde\psi\|\leq 1$, one can define a composition operator
$C_{\widetilde\psi} : \cN_{\cP_0}\to \cN_{\cP_0}$, which turns out
to be bounded. Many results from  the previous sections  have
analogues in this more general setting. In particular, if
$\cP_c:=\{X_iX_j-X_jX_i:\ i,j=1,\ldots,n\}$, then $\cN_{\cP_c} $
coincides  with the symmetric Fock space.
 As a consequence, many of our results have commutative counterparts  for
composition operators on the symmetric Fock space  and on spaces of
analytic functions in the unit ball  $\BB_n$  of $\CC^n$.

\bigskip

\section{ Noncommutative  Littlewood subordination principle }

In this section,  we characterize the free holomorphic self-maps of
the unit ball  $[B(\cH)^n]_1$ in terms of the model boundary
functions with respect to the left creation operators on the full
Fock space $F^2(H_n)$. This will be used  to  provide a
noncommutative Littlewood subordination theorem for the Hardy space
$H^2_{\bf ball}$.

Let $H_n$ be an $n$-dimensional complex  Hilbert space with
orthonormal
      basis
      $e_1$, $e_2$, $\dots,e_n$, where $n\in\{1,2,\dots\}$.
       We consider the full Fock space  of $H_n$ defined by
      $$F^2(H_n):=\CC 1\oplus \bigoplus_{k\geq 1} H_n^{\otimes k},$$
      where   $H_n^{\otimes k}$ is the (Hilbert)
      tensor product of $k$ copies of $H_n$.
       We denote $e_\alpha:=
e_{i_1}\otimes\cdots \otimes  e_{i_k}$  if $\alpha=g_{i_1}\cdots
g_{i_k}$, where $i_1,\ldots, i_k\in \{1,\ldots,n\}$, and
$e_{g_0}:=1$. Note that $\{e_\alpha\}_{\alpha\in \FF_n^+}$ is an
orthonormal basis for $F^2(H_n)$.
Define the left  (resp.~right) creation
      operators  $S_i$ (resp.~$R_i$), $i=1,\ldots,n$, acting on $F^2(H_n)$  by
      setting
      $$
       S_i\varphi:=e_i\otimes\varphi, \qquad  \varphi\in F^2(H_n),
      $$
       (resp.~$
       R_i\varphi:=\varphi\otimes e_i$). Note that $S_iR_j=R_jS_i$ for $i,j\in \{1,\ldots, n\}$.
The noncommutative disc algebra $\cA_n$ (resp.~$\cR_n$) is the norm
closed algebra generated by the left (resp.~right) creation
operators and the identity. The   noncommutative analytic Toeplitz
algebra $F_n^\infty$ (resp.~$R_n^\infty$)
 is the the weakly
closed version of $\cA_n$ (resp.~$\cR_n$). These algebras were
introduced in \cite{Po-von} in connection with a noncommutative
version of the classical  von Neumann inequality (\cite{von}).

Let $C^*(S_1,\ldots, S_n)$ be the Cuntz-Toeplitz
$C^*$-algebra generated by the left creation operators (see
\cite{Cu}). The noncommutative Poisson transform at
 $X:=(X_1,\ldots, X_n)\in [B(\cH)^n]_1^-$ is the unital completely contractive  linear map
 ${\bf P}_X:C^*(S_1,\ldots, S_n)\to B(\cH)$ defined by
 \begin{equation*}
 {\bf P}_X[f]:=\lim_{r\to 1} K_{rX}^* (f \otimes I_\cH)K_{rX}, \qquad f\in C^*(S_1,\ldots,
 S_n),
\end{equation*}
 where the limit exists in the operator norm topology of $B(\cH)$.
Here,
$
K_{rX} :\cH\to    F^2(H_n)\otimes \cH$, \
 $0< r\leq 1$,
is the  noncommutative Poisson  kernel defined by
\begin{equation*}
K_{rX}h:= \sum_{k=0}^\infty \sum_{|\alpha|=k} e_\alpha\otimes r^{|\alpha|}
\Delta_{rX} X_\alpha^*h,\qquad h\in \cH,
\end{equation*}
where   $\Delta_{rX}:=(I_\cH-r^2X_1X_1^*-\cdots -r^2
X_nX_n^*)^{1/2}$. We recall that
 $$
 {\bf P}_X[S_\alpha S_\beta^*]=X_\alpha X_\beta^*, \qquad \alpha,\beta\in \FF_n^+.
 $$
 When $X:=(X_1,\ldots, X_n)$  is a pure row contraction, i.e.
 $ \text{\rm SOT-}\lim\limits_{k\to\infty} \sum_{|\alpha|=k}
X_\alpha X_\alpha^*=0$, then
   we have $${\bf P}_X[f]=K_X^*(f\otimes I_\cH)K_X, \qquad f\in C^*(S_1,\ldots,
 S_n)\ \text{ or } \ f\in F_n^\infty.
   $$
   Under an appropriate modification of the Poisson kernel ($e_\alpha$ becomes $e_{\widetilde \alpha}$ where $\widetilde \alpha= g_{i_k}\cdots g_{i_k}$ is the reverse of
   $\alpha=g_{i_1}\cdots g_{i_k}\in\FF_n^+$), similar results hold for $C^*(R_1,\ldots, R_n)$ of $R_n^\infty$. For simplicity, we use the same notation for the noncommutative Poisson transform.
We refer to \cite{Po-poisson}, \cite{Po-curvature},  and
\cite{Po-unitary} for more on noncommutative Poisson transforms on
$C^*$-algebras generated by isometries.

According to \cite{Po-holomorphic} and \cite{Po-pluriharmonic}, the
noncommutative Hardy space
  $H_{\text{\bf ball}}^\infty $ (see the introduction)  can be identified with    the noncommutative
analytic Toeplitz algebra $F_n^\infty$. More precisely, a bounded
free holomorphic function $\psi$ on $[B(\cH)^n]_1$ is uniquely determined by its {\it
(model) boundary function} $\widetilde \psi(S_1,\ldots, S_n)\in
F_n^\infty $ defined by
$$\widetilde \psi=\widetilde \psi(S_1,\ldots, S_n):=\text{\rm SOT-}\lim_{r\to 1}
\psi(rS_1,\ldots, rS_n). $$ Moreover, $\psi$ is  the noncommutative
Poisson transform of $\widetilde \psi(S_1,\ldots, S_n)$ at
$X:=(X_1,\ldots, X_n)\in [B(\cH)^n]_1$, i.e.,
$$
\psi(X_1,\ldots, X_n)={\bf P}_X [\widetilde \psi(S_1,\ldots, S_n)].
$$
Similar results hold for bounded free holomorphic functions on the
noncommutative ball  $[B(\cH)^n]_1$ with operator-valued
coefficients. There are also versions of these results when the boundary function is taken with respect to the right creation operators $R_1,\ldots, R_n$.

Throughout this paper, we deal with free holomorphic self-maps of
the unit ball  $[B(\cH)^n]_1$. The following results gives us, in
particular, a characterization  of these maps in terms of the model boundary
functions with respect to the left creation operators on the full
Fock space $F^2(H_n)$. For simplicity, $[X_1,\ldots, X_n]$ denotes
either the $n$-tuple $(X_1,\ldots, X_n)\in B(\cH)^n$ or the operator
row matrix $[ X_1\, \cdots \, X_n]$ acting from $\cH^{(n)}$, the
direct sum  of $n$ copies of  a Hilbert space $\cH$, to $\cH$.

\begin{theorem}\label{strict}
Let $\varphi:[B(\cH)^n]_1 \to [B(\cH)^m]_1^-$ be   a    free
holomorphic function.  Then  the following statements hold.
\begin{enumerate}
\item[(i)] Either
$\varphi\left([B(\cH)^n]_1\right)\subseteq  [B(\cH)^m]_1$
 or there exists $\zeta\in \partial\BB_m$ such that $\varphi(X)=\zeta$ for all $X\in [B(\cH)^n]_1$.
 \item[(ii)] $\varphi$ is constant if and only if
 $\|\varphi(0)\|=\|\varphi\|_\infty$.
\item[(iii)] If  $\varphi$ is non-constant and $\varphi_r(X):=\varphi(rX)$, $X\in [B(\cH)^n]_1$, then
 the map $[0,1)\ni r\mapsto \|\varphi_r\|_\infty$ is strictly
increasing.
\item[(iv)] If $\widetilde \varphi$ is the  boundary function  of $\varphi$
with respect to $S_1,\ldots, S_n$, then
$\varphi\left([B(\cH)^n]_1\right)\subseteq [B(\cH)^m]_1$ if and only
if either $\widetilde \varphi=\zeta I$ for some $\zeta\in \BB_n$ or
$\widetilde \varphi$ is  non-constant  with $\|\widetilde
\varphi\|\leq 1$.
\end{enumerate}
\end{theorem}
\begin{proof}
If $\|\varphi\|_\infty<1$, then (i) holds. Assume that
$\|\varphi\|_\infty=1$. In this case, if $\|\varphi(0)\|<1$ then,
according to the maximum principle for free holomorphic functions
(see Proposition 5.2 from \cite{Po-automorphism}), we have
$\|\varphi(X)\|<1$ for all $X\in [B(\cH)^n]_1$.  It remains  to
consider the case when $\|\varphi(0)\|=1$. Set
$\zeta=[\zeta_1,\ldots, \zeta_m]:=\varphi(0)\in \partial \BB_m$ and
let $U\in M_{m\times m}$ be a unitary matrix such that
$[\zeta_1,\ldots, \zeta_m]U=\xi_1:=[1,0,\ldots, 0]\in \partial
\BB_m$. Let $\varphi_U(X):=[X_1,\ldots, X_m]U$ and note that
$g:=\varphi_U\circ\varphi:[B(\cH)^n]_1 \to [B(\cH)^m]_1^-$ is    a
free holomorphic function with $g(0)=\xi_1$. Setting $g=(g_1,\ldots,
g_m)$, we deduce that  $g_i$ are free holomorphic functions  with
$g_1(0)=1$ and $g_i(0)=0$ if $i=2,\ldots,m$.  Applying Theorem 5.1
from \cite{Po-automorphism} to $g_1$, we deduce that $g_1(X)=1$ for
all $X\in [B(\cH)^n]_1$. Hence $g_2=\cdots=g_m=0$. This implies that
that $\varphi(X)=\zeta$ for all $X\in [B(\cH)^n]_1$, and completes the
proof of item (i). Since the direct implication in item (ii) is
obvious, we assume that $\|\varphi(0)\|=\|\varphi\|_\infty$ and
$\|\varphi\|_\infty=1$. The rest of the  proof  of (ii) is contained in
the proof of item (i).

  To prove  item (iii),  assume that $\varphi$ is non-constant. Due
  to part (ii), we must have $\|\varphi(0)\|<\|\varphi\|_\infty$.
  Using again Proposition
5.2 from \cite{Po-automorphism}), we have
$\|\varphi(X)\|<\|\varphi\|_\infty$ for all $X\in [B(\cH)^n]_1$. Let
$0\leq r_1<r_2<1$. We recall that,  if $r\in[0,1)$, then the
boundary function $\widetilde \varphi_r$ is in  $\cA_n\otimes
M_{1\times m}$, where   $\cA_n$ is the noncommutative disc algebra
and $\|\varphi_r\|_\infty=\|\widetilde
\varphi_r\|=\|\varphi_r(S_1,\ldots, rS_n)\|$.  Using the
noncommutative von Neumann inequality (see \cite{Po-von}) and
applying the above-mentioned result to $\varphi_{r_2}$ and
$(X_1,\ldots, X_n):=(\frac{r_1}{r_2}S_1,\ldots,
\frac{r_1}{r_2}S_n)$, we obtain
$$
\|\varphi_{r_1}\|_\infty=\|\varphi_{r_1}(S_1,\ldots,
S_n)\|=\left\|\varphi_{r_2}\left(\frac{r_1}{r_2}S_1,\ldots,
\frac{r_1}{r_2}S_n\right)\right\|<\|\varphi_{r_2}(S_1,\ldots,
S_n)\|=\|\varphi_{r_2}\|_\infty,
$$
which shows that (iii) holds.

Now we prove (iv). If $\varphi\left([B(\cH)^n]_1\right)\subseteq
[B(\cH)^m]_1$, then $\|\widetilde \varphi\|=\|\varphi\|_\infty\leq
1$ and the result  follows. Conversely, assume that $\|\widetilde
\varphi\|\leq 1$ and  $\widetilde \varphi$ is not of the form $\zeta
I$ for some $\zeta\in \BB_n$. Then $\varphi$ is not a constant and
due to (ii) we have $\|\varphi(0)\|<\|\varphi\|_\infty$. Using now
item  (iii), we deduce that the map $[0,1)\ni r\mapsto
\|\varphi_r\|_\infty$ is strictly increasing.  If $X:=(X_1,\ldots,
X_n)\in [B(\cH)^n]_1$, then there is $r\in [0,1)$ such that
$\|X\|<r$. Consequently, due to the noncommutative von Neumann
inequality, we have
$$
\|\varphi(X_1,\ldots, X_n)\|\leq \|\varphi(rS_1,\ldots,
rS_n)\|=\|\varphi_r\|_\infty<1.
$$
The proof is complete.
\end{proof}

Note that if $f\in H_{\bf ball}$, then $f\in H^2_{\bf ball}$ if and only
$\sup_{r\in [0,1)} \|f(rS_1,\ldots, rS_n)1\|<\infty$. Moreover, in this case, we have
$$
\|f\|_2=\lim_{r\to 1}\|f(rS_1,\ldots, rS_n)1\|
=\sup_{r\in [0,1)} \|f(rS_1,\ldots, rS_n)1\|.
$$
If  $f=\sum_{k=0}^\infty \sum_{|\alpha|=k} a_\alpha
X_\alpha$ and $g=\sum_{k=0}^\infty \sum_{|\alpha|=k} b_\alpha
X_\alpha$ are in $H^2_{\bf ball}$, then
\begin{equation*}
\begin{split}
\left< f,g\right>&=\lim_{r\to 1}\left<f(rS_1,\ldots,
rS_n)1,g(rS_1,\ldots, rS_n)1\right>_{F^2(H_n)} =\left<
\sum_{\alpha\in \FF_n^+} a_\alpha e_\alpha, \sum_{\alpha\in \FF_n^+}
b_\alpha e_\alpha\right>_{F^2(H_n)}.
\end{split}
\end{equation*}
Consequently, the noncommutative Hardy space $H^2_{\bf ball}$   can
be identified with the full Fock space $F^2(H_n)$, via the unitary
operator $\cU:H^2_{\bf ball}\to F^2(H_n)$ defined by the mapping
$$ H^2_{\bf ball}\ni
\sum_{k=0}^\infty \sum_{|\alpha|=k} a_\alpha  X_\alpha  \mapsto
\sum_{k=0}^\infty \sum_{|\alpha|=k} a_\alpha e_\alpha\in F^2(H_n).
$$
This identification will be used throughout the paper whenever
necessary. We recall  from \cite{Po-automorphism}  that
 if $f:[B(\cH)^n]_1\to B(\cH)$  and
$\varphi:[B(\cH)^n]_1\to [B(\cH)^n]_1$ are  free holomorphic
functions then $f\circ \varphi$ is a free holomorphic function on
$[B(\cH)^n]_1$ defined by
$$
(f\circ \varphi)(X_1,\ldots, X_n)=\sum_{k=0}^\infty
\sum_{|\alpha|=k}  a_\alpha \varphi_\alpha (X_1,\ldots, X_n),\qquad
(X_1,\ldots, X_n)\in [B(\cH)^n]_1,
$$
where the convergence is in the operator norm topology and
$\varphi=(\varphi_1,\ldots, \varphi_n)$.

We can prove now the following noncommutative Littlewood
subordination theorem for the Hardy space $H^2_{\bf ball}$, which will play an important role in this paper.

\begin{theorem}\label{comp1}
Let  $\varphi$ be  a   free holomorphic self-map of  the ball
$[B(\cH)^n]_1$ such that $\varphi(0)=0$, and let $f\in H^2 _{\bf
ball}$. Then  $ f\circ \varphi \in H^2_{\bf ball}$   and $\|f\circ
\varphi\|_2\leq \|f\|_2$.
\end{theorem}
 \begin{proof}
 Let $\varphi:=(\varphi_1,\ldots, \varphi_n)$ be  a    free holomorphic self-map of  the ball
$[B(\cH)^n]_1 $ such that $\varphi(0)=0$, and let $\widetilde \varphi=(\widetilde \varphi_1,\ldots, \widetilde \varphi_n)\in F_n^\infty\otimes \CC^n$ be the model boundary  function with respect to the left creation operators $S_1,\ldots, S_n$.
Thus $\widetilde \varphi_i:=\text{\rm SOT-}\lim_{r\to 1} \varphi_i(rS_1,\ldots, rS_n)$ for $i=1,\ldots,n$.
Let $\cP_n$ be the set of all polynomials in $F^2(H_n)$ and define $C_{\widetilde \varphi}:\cP_n\to F^2(H_n)$ by setting
$$
C_{\widetilde \varphi}\left(\sum_{|\alpha|\leq
m} a_\alpha e_\alpha\right)
:=\sum_{|\alpha|\leq
m} a_\alpha \widetilde \varphi_\alpha(1).
$$
If  $q:=\sum_{|\alpha|\leq
m} a_\alpha X_\alpha$ is a polynomial in $H^2_{\bf ball}$, then
 $p:=\cU q=\sum_{|\alpha|\leq
m} a_\alpha e_\alpha$ is a polynomial in $F^2(H_n)$. Note that
$p=p(0)+\sum_{i=1}^n S_i(S_i^*p)$, where $p(0)=P_\CC p=a_0:=a_{g_0}$. Hence,
we deduce that
 $$C_{\widetilde \varphi} p=a_0 +\sum_{i=1}^n \widetilde\varphi_i C_{\widetilde \varphi}(S_i^*p).
  $$
  Since $\varphi (0)=0$,  the vector
$\sum_{i=1}^n \widetilde\varphi_i C_{\widetilde \varphi}(S_i^*p)$ is
orthogonal to the constants in $F^2(H_n)$. Consequently, using the fact that
$[\widetilde \varphi_1,\ldots, \widetilde \varphi_n]$ is
a row contraction, we have
\begin{equation*}
\begin{split}
\| C_{\widetilde \varphi} p\|_2^2&=|a_0|^2+\left\|\sum_{i=1}^n \widetilde \varphi_i C_{\widetilde \varphi}(S_i^*p)\right\|^2 \\
&\leq  |a_0|^2+ \left\|\oplus_{i=1}^n C_{\widetilde \varphi}(S_i^*p)
\right\|^2.
\end{split}
\end{equation*}
 Note that, for each $i=1,\ldots, n$, we have
 $$
 C_{\widetilde\varphi} (S_i^*p)=(S_i^*p)(0)+\sum_{j=1}^n \widetilde\varphi_j C_{\widetilde\varphi}(S_j^*S_i^*p).
 $$
 Hence,  using again that $\varphi(0)=0$ and that  $[\widetilde \varphi_1,\ldots, \widetilde \varphi_n]$ is a row contraction,  we deduce that

 \begin{equation*}
\begin{split}
\left\|\bigoplus_{i=1}^n C_{\widetilde \varphi}(S_i^*p)
\right\|^2 &= \left\|\bigoplus_{i=1}^n (S_i^*p)(0)
\right\|^2
 + \left\|\bigoplus_{i=1}^n \left(\sum_{j=1}^n
\widetilde\varphi_j C_{\widetilde\varphi}(S_j^*S_i^*p)\right)
\right\|^2 \\
& \leq
\sum_{|\alpha|=1} |a_\alpha|^2 +\sum_{i=1}^n \left\|\sum_{j=1}^n \widetilde\varphi_j C_{\widetilde\varphi}(S_j^*S_i^*p)\right\|^2\\
&\leq \sum_{|\alpha|=1} |a_\alpha|^2 + \left\|
\bigoplus_{|\beta|=2}C_{\widetilde\varphi}(S_\beta^*p)
\right\|^2.
\end{split}
\end{equation*}
 Similarly,  for any $k\in \{1,\ldots, m+1\}$, we obtain
 $$
 \left\|\bigoplus_{|\beta|=k-1}
C_{\widetilde\varphi}(S_\beta^*p)
\right\|^2 \leq \sum_{|\alpha|=k-1} |a_\alpha|^2 +
\left\|\bigoplus_{|\beta|=k} C_{\widetilde\varphi}(S_\beta^*p)
\right\|^2.
$$
 Using these relations and the fact that $S_\gamma^*p=0$ for $|\gamma|\geq m+1$, we obtain
 \begin{equation*}
 \| C_{\widetilde\varphi} p\|_2^2\leq \sum_{|\alpha|\leq m} |a_\alpha|^2=\|p\|^2_2.
 \end{equation*}
 Since $\cU C_\varphi \cU^{-1}p=C_{\widetilde\varphi}p$, we deduce that
 \begin{equation}
  \label{Cp}
  \|C_\varphi q\|_2\leq \|q\|_2\quad \text{ for any polynomial }\
  q\in H^2_{\bf ball}.
  \end{equation}
Now, we prove  that $f\circ \varphi$ is in
$H^2_{\bf ball}$  for any $f\in H^2_{\bf ball}$ and
 $\|f\circ \varphi\|_2\leq \|f\|_2$.
Let $f(X_1,\ldots, X_n)=\sum_{k=0}^\infty \sum_{|\alpha|=k} c_\alpha
X_\alpha$ be  a free holomorphic function in $H^2_{\bf ball}$. Then
$f\circ \varphi$ is a free holomorphic function on $[B(\cH)^n]_1$,
defined by
$$
(f\circ \varphi)(X_1,\ldots, X_n)=\sum_{k=0}^\infty
\sum_{|\alpha|=k}  c_\alpha \varphi_\alpha (X_1,\ldots, X_n),\qquad
(X_1,\ldots, X_n)\in [B(\cH)^n]_1,
$$
where the convergence is in the operator norm topology.  In
particular, we have
\begin{equation}
\label{circ} (f\circ \varphi)(rS_1,\ldots, rS_n)1=\sum_{k=0}^\infty
\sum_{|\alpha|=k}  c_\alpha \varphi_\alpha (rS_1,\ldots, rS_n)1,
\end{equation}
where the convergence is in $F^2(H_n)$. On the other hand, setting
$p_m(X_1,\ldots, X_n):=\sum_{k=0}^m \sum_{|\alpha|=k}  c_\alpha
X_\alpha$, we have $p_m\to f$ in $H^2_{\bf ball}$ as $m\to\infty$.
Therefore, $\{p_m\}$ is a Cauchy sequence in $H^2_{\bf ball}$. Due
to relation \eqref{Cp}, we have
$$
\|p_m\circ\varphi- p_k\circ \varphi\|_2\leq \|p_m-p_k\|_2,\qquad
m,k\in \NN.
$$
Hence, $\{p_m\circ\varphi\}$ is a Cauchy sequence  sequence in
$H^2_{\bf ball}$ and, consequently, there is $g\in H^2_{\bf ball}$
such that $p_m\circ \varphi\to g$ in  $H^2_{\bf ball}$. Hence, for
each $r\in [0,1)$, we have
$$\lim_{m\to\infty} (p_m\circ \varphi)(rS_1,\ldots, rS_n)1=g(rS_1,\ldots, rS_n)1.
$$
Combining this relation with \eqref{circ}, we get
$$
g(rS_1,\ldots, rS_n)1=(f\circ \varphi)(rS_1,\ldots, rS_n)1,\qquad
r\in[0,1).
$$
Since $f\circ \varphi$ and $g$ are free holomorphic functions, we
deduce that
 $f\circ \varphi=g\in H^2_{\bf ball}$. Now, since
 $p_m\circ \varphi\to f\circ \varphi$ in  $H^2_{\bf ball}$, relation \eqref{Cp} implies
$\|f\circ \varphi\|_2\leq \|f\|_2 $ for any  $f\in H^2_{\bf ball}$.
The proof is complete.
\end{proof}

 If in addition to the hypothesis of Theorem \ref{comp1}, we assume that $\varphi$ is  inner,
  i.e. the boundary function $\widetilde\varphi$  is an isometry, then we can prove the following result.

\begin{theorem}\label{comp1-iso}
Let  $\varphi$ be  an inner    free holomorphic self-map of  the ball
$[B(\cH)^n]_1$ such that $\varphi(0)=0$. Then the composition operator $C_\varphi$ is an isometry on $ H^2 _{\bf ball}$.
\end{theorem}
 \begin{proof} Let $\widetilde\varphi:=[\widetilde \varphi_1,\ldots, \widetilde \varphi_n]$ be the boundary function of $\varphi$ with respect to the left creation opeartors.
  Note that due to the fact that $\varphi(0)=0$, we have  $\left< 1, \widetilde \varphi_\alpha 1\right>=0$ for any $\alpha\in \FF_n^+$ with $|\alpha|\geq 1$. On the other hand, since $[\widetilde \varphi_1,\ldots, \widetilde \varphi_n]$  is an isometry, we have $\widetilde \varphi_i^* \widetilde \varphi_j=\delta_{ij}I_{F^2(H_n)}$ for $i,j\in \{1,\ldots, n\}$.
Consequently,
\begin{equation*}
\begin{split}
\left< \varphi_\alpha, \varphi_\beta\right>_{H^2_{\bf ball}}&=
\left<\widetilde \varphi_\alpha 1, \widetilde \varphi_\beta 1\right>
=
\begin{cases}
\left<\widetilde \varphi_\gamma 1,  1\right>\quad &\text{ if } \alpha=\beta \gamma\\
1 \quad &\text{ if } \alpha=\beta \\
\left< 1, \widetilde \varphi_\gamma 1\right>\quad &\text{ if } \beta =\alpha \gamma
\end{cases}
\\
&=\begin{cases} 1 &\text{ if } \alpha=\beta\\
0&\text{ if } \alpha\neq\beta.
\end{cases}
\end{split}
\end{equation*}
This shows that $\{\varphi_\alpha\}_{\alpha\in \FF_n^+}$ is an orthonormal set in $H^2_{\bf ball}$.
 If $f=\sum_{k=0}^\infty \sum_{|\alpha|=k} c_\alpha  X_\alpha$ is in $H^2_{\bf ball}$, then
 setting $p_m(X_1,\ldots,
X_n):=\sum_{k=0}^m \sum_{|\alpha|=k}  c_\alpha X_\alpha$, we have
$p_m\to f$ in $H^2_{\bf ball}$, as $m\to\infty$. Note that
\begin{equation}
\label{isom}
\|p_m\circ \varphi\|_2^2=\left<\sum_{k=0}^m
\sum_{|\alpha|=k} c_\alpha  \varphi_\alpha, \sum_{k=0}^m
\sum_{|\beta|=k} c_\beta \varphi_\beta\right>= \sum_{k=0}^m
\sum_{|\alpha|=k} |c_\alpha|^2=\|p_m\|_2^2.
\end{equation}
Consequently, $\{p_m\circ\varphi\}$ is a Cauchy sequence sequence in
$H^2_{\bf ball}$ and  there is $g\in H^2_{\bf ball}$ such that
$p_m\circ \varphi\to g$ in  $H^2_{\bf ball}$. Hence,
  we deduce that
$$g(rS_1,\ldots, rS_n)1=\lim_{m\to\infty} (p_m\circ \varphi)(rS_1,\ldots, rS_n)1
=(f\circ \varphi)(rS_1,\ldots, rS_n)1,\qquad r\in[0,1).
$$
Since $f\circ \varphi$ and $g$ are free holomorphic functions, the
identity theorem for free holomorphic functions implies
 $f\circ \varphi=g$.
Therefore,  relation \eqref{isom} implies that $C_\varphi$ is an
isometry and the proof is complete.
\end{proof}

\bigskip

\section{Composition operators on the noncommutative Hardy
   space $H^2_{\bf ball}$}

This section contains the core material on the  boundedness of
compositions operators on the noncommutative Hardy space $H^2_{\bf
ball}$  and the estimates of  their norms. We also characterize the
similarity of composition operators on $H^2_{\bf ball}$ to
contractions.

Let $\theta$ be an analytic function on the open disc $\DD$. It is
well-known that  the map $\varphi:\DD\to \RR^+$ defined by
$\varphi(\lambda):=|\theta(\lambda)|^2$ is subharmonic. A  classical
result on harmonic majorants (see Section 2.6 in \cite{Du}) states
that $\theta$ is in the Hardy space $H^2(\DD)$ if and only if
$\varphi$ has a harmonic majorant. Moreover, the least harmonic
majorant of $\varphi$ is given by the Herglotz-Riesz (\cite{Her},
\cite{Ri}) formula
$$
h(\lambda)=\frac{1}{2\pi}\int_0^{2\pi}
\frac{e^{it}+\lambda}{e^{it}-\lambda}|\theta(e^{it})|^2 dt,\quad
\lambda\in \DD.
$$
In \cite{Po-pluri-maj}, we obtained free analogues of these results.
Since these results play an important role in our investigation we
shall recall them.

We say that a map $h:[B(\cH)^n]_1\to B(\cH)$ is a self-adjoint {\it free
pluriharmonic function} on $[B(\cH)^n]_1$ if $h=\Re f:=\frac{1}{2}(f^*+f)$
for some free holomorphic function $f$ on $[B(\cH)^n]_1$. An
arbitrary free pluriharmonic function is a linear combination of
self-adjoint free pluriharmonic functions.
  A  {\it
pluriharmonic   curve} in
 $ C^*(S_1,\ldots, S_n)$ is a  map $\varphi:[0,1)\to
 \overline{ \cA_n+
 \cA_n}^{\|\cdot\|}$
satisfying the Poisson mean value property, i.e.,
\begin{equation*}
  \varphi(r)= {\bf P}_{\frac{r}{t} S}[\varphi(t)]\quad \text{
for } \ 0\leq r<t<1,
\end{equation*}
where $S:=(S_1,\ldots, S_n)$ and  ${\bf P}_X[u]$ is the noncommutative Poisson transform of $u$
at $X$.
 According to
\cite{Po-pluriharmonic}, there exists a one-to-one correspondence
$u\mapsto \varphi$ between the set of all free  pluriharmonic
functions on the noncommutative ball $[B(\cH)^n]_1$, and the set of
all pluriharmonic curves $\varphi:[0,1)\to \overline{ \cA_n^*+
\cA_n}^{\|\cdot\|}$.
   Moreover, we have
$$u(X)={\bf P}_{\frac{1}{r} X}[\varphi(r)]\quad \text{
for } \   X\in [B(\cH)^n]_r \ \text{ and } \ r\in (0,1),
$$
and $\varphi(r)=u(rS_1,\ldots, rS_n)$ if $r\in [0,1)$. We say that a
map $\psi:[0,1)\to \overline{ \cA_n+
 \cA_n}^{\|\cdot\|}$  is self-adjoint  if  $\psi(r)=\psi(r)^*$ for  $r\in
 [0,1)$. We call $\psi$
  a {\it  sub-pluriharmonic} curve   provided that  for each $\gamma\in (0,1)$
 and each self-adjoint  pluriharmonic  curve
  $\varphi: [0,\gamma]\to  \overline{ \cA_n+
 \cA_n}^{\|\cdot\|}$,
  if
 $\psi(\gamma)\leq \varphi(\gamma)$,   then
 $
\psi(r)\leq \varphi(r)\ \text{ for any } \ r\in [0,\gamma].
 $
We proved  that a self-adjoint  map $g:[0,1)\to
\overline{ \cA_n^*+ \cA_n}^{\|\cdot\|}$   is a sub-pluriharmonic
curve in   $C^*(S_1,\ldots, S_n)$  if and only if
$$
g(r)\leq {\bf P}_{\frac{r}{\gamma}S}[g(\gamma)]\qquad \text{ for  }\
0\leq r<\gamma<1.
$$
 We obtained a characterization  for  the class of all
sub-pluriharmonic curves
  that admit free
pluriharmonic majorants, and   proved the existence of the least
pluriharmonic majorant. We mention that all these results can be written for sub-pluriharmonic curves in $C^*(R_1,\ldots, R_n)$, where $R_1,\ldots, R_n$ are the right creation operators on the full Fock space.

 In \cite{Po-pluri-maj}, we showed
that, for any
 free holomorphic function   $\Theta$ on the noncommutative ball $[B(\cH)^n]_1$, the mapping
 $$
 \varphi:[0,1)\to C^*(R_1,\ldots, R_n),\quad
 \varphi(r)=\Theta(rR_1,\ldots, rR_n)^*\Theta(rR_1,\ldots, rR_n),
 $$
is a sub-pluriharmonic curve in the Cuntz-Toeplitz algebra
  generated by the right creation operators $R_1,\ldots, R_n$. We proved that a free holomorphic
function $\Theta$   is in the noncommutative Hardy space $H^2_{\bf
ball}$ if and only if $\varphi$  has a pluriharmonic majorant. In
this case, the least pluriharmonic majorant $\psi$  for $\varphi$ is
given by $ \psi(r):=\Re W(rR_1, \ldots rR_n)$,
$r\in[0,1), $
  where $W$ is the free holomorphic function  having the
  Herglotz-Riesz
  type representation
  \begin{equation*}
  W(X_1,\ldots,
  X_n)=({\mu}_\theta\otimes \text{\rm id})\left[\left(I+\sum_{i=1}^n R_i^*\otimes
  X_i\right)\left(I-\sum_{i=1}^n R_i^*\otimes
  X_i\right)^{-1}\right]
  \end{equation*}
for $(X_1,\ldots, X_n)\in [B(\cH)^n]_1$, where
$\mu_\theta:\cR_n^*+\cR_n\to  \CC$ is a   positive linear map
uniquely determined
 by   the function $\Theta$.

Now, we need to recall from \cite{Po-automorphism} some basic facts
concerning the free holomorphic automorphisms of the noncommutative
ball $[B(\cH)^n]_1$. A map $\varphi:[B(\cH)^n]_{1}\to [B(\cH)^n]_{1}$ is
called free biholomorphic if $\varphi$  is  free homolorphic, one-to-one
and onto, and has  free holomorphic inverse. The automorphism group
of $[B(\cH)^n]_{1}$, denoted by $Aut([B(\cH)^n]_1)$, consists of all
free biholomorphic functions of $[B(\cH)^n]_{1}$. It is clear that
$Aut([B(\cH)^n]_1)$ is  a group with respect to the composition of
free holomorphic functions.
We used  the theory of noncommutative
characteristic functions for row contractions (\cite{Po-charact}) to
find all the involutive free holomorphic automorphisms of
$[B(\cH)^n]_1$, which turned out to be of the form
\begin{equation*}
 \Phi_\lambda(X_1,\ldots, X_n)=- \Theta_\lambda(X_1,\ldots, X_n), \qquad (X_1,\ldots,
 X_n)\in [B(\cH)^n]_1,
\end{equation*}
for some $\lambda=[\lambda_1,\ldots, \lambda_n]\in \BB_n$, where
$\Theta_\lambda$ is the characteristic function  of the row
contraction $\lambda$, acting as an operator from $\CC^n$  to $\CC$.
We recall that the characteristic function of the row contraction
$\lambda$ is the boundary
function (with respect to $R_1,\ldots, R_n$)
$$\widetilde{\Theta}_\lambda:=\text{SOT-}\lim_{r\to1}
{\Theta}_\lambda(rR_1,\ldots, rR_n)$$
 of the free holomorphic
function $\Theta_\lambda:[B(\cH)^n]_1\to [B(\cH)^n]_1$ given by
\begin{equation*}
  \Theta_\lambda(X_1,\ldots, X_n):=-{
\lambda}+\Delta_{ \lambda}\left(I_\cH-\sum_{i=1}^n \bar{{
\lambda}}_i X_i\right)^{-1} [X_1,\ldots, X_n] \Delta_{{\lambda}^*}
\end{equation*}
for $(X_1,\ldots, X_n)\in [B(\cH)^n]_1$, where
$\Delta_\lambda=(1-\|\lambda\|_2^2)^{1/2} I_\CC$ and
$\Delta_{\lambda^*}=(I_\cK-\lambda^*\lambda)^{1/2}$.  For
simplicity, we  used the notation $ {\lambda}:=[ {\lambda}_1
I_\cG,\ldots, {\lambda}_n I_\cG]$ for the row contraction acting
from $\cG^{(n)}$ to $\cG$, where $\cG$ is a Hilbert space.

In \cite{Po-automorphism}, we proved that if
  $\lambda:=(\lambda_1,\ldots, \lambda_n)\in
\BB_n\backslash \{0\}$ and  $\gamma:=\frac{1}{\|\lambda\|_2}$, then
$\Phi_\lambda:=-\Theta_\lambda$ is a free holomorphic function on
$[B(\cH)^n]_\gamma$ which has the following properties:
\begin{enumerate}
\item[(i)]
$\Phi_\lambda (0)=\lambda$ and $\Phi_\lambda(\lambda)=0$;
\item[(ii)] the identities
\begin{equation}\label{E1}
\begin{split}
I_{\cH}-\Phi_\lambda(X)\Phi_\lambda(Y)^*&= \Delta_{\lambda}(I- X
\lambda^*)^{-1}(I-  X   Y^*)(I-  \lambda  Y^*)^{-1}
\Delta_{\lambda},\\
I_{\cH\otimes \CC^n}-\Phi_\lambda(X)^*\Phi_\lambda(Y)&=
\Delta_{\lambda^*}(I-{X}^*\lambda)^{-1}(I-{X}^* Y)(I-{\lambda}^*
Y)^{-1} \Delta_{\lambda^*},
\end{split}
\end{equation}
hold  for all  $X$ and $Y$ in  $[B(\cH)^n]_\gamma$;

\item[(iii)] $\Phi_\lambda$ is an involution, i.e., $\Phi_\lambda(\Phi_\lambda(X))=X$
for any $X\in [B(\cH)^n]_\gamma$;
\item[(iv)] $\Phi_\lambda$ is a free holomorphic automorphism of the
noncommutative unit ball $[B(\cH)^n]_1$;
\item[(v)] $\Phi_\lambda$ is a homeomorphism of $[B(\cH)^n]_1^-$
onto $[B(\cH)^n]_1^-$.
\item[(vi)] ${\Phi}_\lambda$ is  inner,    i.e.,  the boundary function $\widetilde{\Phi}_\lambda$ is  an
isometry.
\end{enumerate}
Moreover, we determined all the free holomorphic automorphisms of
the noncommutative ball $[B(\cH)^n]_1$ by  showing  that if $\Phi\in
Aut([B(\cH)^n]_1)$ and $\lambda:=\Phi (0)$, then there is a unitary
operator $U$ on $\CC^n$ such that
$$
\Phi= \Phi_\lambda\circ \Phi_U,
$$
where $$  \Phi_U(X_1,\ldots X_n):=[X_1,\ldots, X_n]U , \qquad
(X_1,\ldots, X_n)\in [B(\cH)^n]_1.
$$

We have now all the ingredients  to prove the  key result of this section.

\begin{theorem}\label{comp2}
If $\varphi$ is a   free holomorphic  automorphism of  the  noncommutative ball
$[B(\cH)^n]_1$, then $C_\varphi f\in H^2_{\bf ball}$ for all $f\in H^2_{\bf ball}$, and
$$
\left(\frac{1-\|\varphi(0)\|}{1+\|\varphi(0)\|}\right)^{1/2}\|f\|\leq \|C_\varphi f\|\leq
\left(\frac{1+\|\varphi(0)\|}{1-\|\varphi(0)\|}\right)^{1/2}\|f\|
$$
for all $f\in H^2_{\bf ball}$. Moreover,  these inequalities are
best possible and
$$
\|C_\varphi\|=\left(\frac{1+\|\varphi(0)\|}{1-\|\varphi(0)\|}\right)^{1/2}.
$$
\end{theorem}

\begin{proof}
Let $\varphi:=(\varphi_1,\ldots, \varphi_n)$ be an inner
free holomorphic self-map of  the  noncommutative ball
$[B(\cH)^n]_1$. Then the boundary function with respect to the right
creation operators $R_1,\ldots, R_n$, i.e.,
$$\widetilde \varphi:=(\widetilde\varphi_1
,\ldots, \widetilde\varphi_n ),\quad \text{  where }\quad \widetilde
\varphi_i:=\text{\rm SOT-}\lim_{r\to 1}\varphi_i(rR_1,\ldots, rR_n),
$$
 is an isometry. Consequently,
$\widetilde\varphi_i ^*\widetilde\varphi_j =\delta_{ij}
I_{F^2(H_n)}$ for $i,j\in \{1,\ldots,n\}$. Recall that $R_1,\ldots,
R_n$ are isometries with orthogonal ranges, so $R_i^*R_j=
\delta_{ij} I_{F^2(H_n)}$ for $i,j\in \{1,\ldots,n\}$. Consequently,
we have
$$
R_\alpha^* R_\beta=\begin{cases} R_\gamma\quad &\text{ if } \beta=\alpha \gamma\\
I \quad &\text{ if } \alpha=\beta \\
R_\gamma^* \quad &\text{ if } \alpha=\beta\gamma
\end{cases}
\quad \text{ and } \quad
\widetilde\varphi_\alpha^* \widetilde\varphi_\beta=
\begin{cases} \widetilde\varphi_\gamma\quad &\text{ if } \beta=\alpha \gamma\\
I \quad &\text{ if } \alpha=\beta \\
\widetilde\varphi_\gamma^* \quad &\text{ if } \alpha=\beta\gamma.
\end{cases}
$$
 Fix a noncommutative
polynomial $p(X_1,\ldots, X_n):=\sum_{|\alpha|\leq m}a_\alpha
r^{|\alpha|} X_\alpha$. Note that, using the above-mentioned relations and
applying the noncommutative Poisson transform (with respect to $R_1,\ldots, R_n$) at
$[\widetilde\varphi_1 ,\ldots, \widetilde\varphi_n]$, we obtain
\begin{equation} \label{P-tra}
{\bf P}_{[\widetilde\varphi_1 ,\ldots,
\widetilde\varphi_n]}\left[p(rR_1,\ldots, rR_n)^*p(rR_1,\ldots,
rR_n)\right]=p(r\widetilde\varphi_1,\ldots, r\widetilde\varphi_n)^*
p(r\widetilde\varphi_1,\ldots, r\widetilde\varphi_n)
\end{equation}
for any $r\in[0,1)$. Since
 $ p \in H^2_{\bf ball}$, Theorem 2.3 from \cite{Po-pluri-maj} shows that    the map
$$
[0,1)\ni r\mapsto p(rR_1,\ldots, rR_n)^*p(rR_1,\ldots, rR_n)\in
C^*(R_1,\ldots, R_n)
$$
has a pluriharmonic majorant.   In this case, the least
pluriharmonic majorant   is given by
\begin{equation*}
\label{==} [0,1)\ni r\mapsto \Re W(rR_1, \ldots rR_n)\in
C^*(R_1,\ldots, R_n),
\end{equation*}
  where $W$ is the free holomorphic function on $[B(\cH)^n]_1$ having the
  Herglotz-Riesz
  type representation
  \begin{equation}\label{Herg}
  W(X_1,\ldots,
  X_n)=({\mu}_p\otimes \text{\rm id})\left[\left(I+\sum_{i=1}^n R_i^*\otimes
  X_i\right)\left(I-\sum_{i=1}^n R_i^*\otimes
  X_i\right)^{-1}\right]
  \end{equation}
for $(X_1,\ldots, X_n)\in [B(\cH)^n]_1$, where
$\mu_p:\cR_n^*+\cR_n\to  \CC$ is the completely positive linear
map uniquely determined by   the equation
\begin{equation}
\label{def-mu}  \mu_p(R_{\widetilde \alpha}^*) := \lim_{r\to
1}\left<p(rR_1,\ldots, rR_n)^*  S_{\widetilde \alpha}^*
p(rR_1,\ldots, rR_n) 1,   1\right>
\end{equation}
for $\alpha\in \FF_n^+$, where $\tilde\alpha$ is  the
reverse of $\alpha\in \FF_n^+$, i.e.,
  $\tilde \alpha= g_{i_k}\cdots g_{i_k}$ if
   $\alpha=g_{i_1}\cdots g_{i_k}\in\FF_n^+$.
 Therefore, we have
$$
p(rR_1,\ldots, rR_n)^*p(rR_1,\ldots, rR_n)\leq \Re
W(rR_1, \ldots, rR_n)
$$
for any $r\in [0,1)$. Hence,  using relation  \eqref{P-tra} and  the
fact that  the noncommutative Poisson transform is a completely
positive map, we deduce that
$$
p(r\widetilde\varphi_1,\ldots, r\widetilde\varphi_n)^*
p(r\widetilde\varphi_1,\ldots, r\widetilde\varphi_n) \leq \Re  W(r\widetilde\varphi_1, \ldots, r\widetilde\varphi_n)
$$
for any $r\in [0,1)$. The latter relation implies
\begin{equation*}
\begin{split}
\|p(r\widetilde\varphi_1,\ldots, r\widetilde\varphi_n)1\|^2 &\leq
\left<\text{\rm Re}\, W(r\widetilde\varphi_1, \ldots,
\widetilde\varphi_n) 1,1\right>=\Re W (r\varphi_1(0),\ldots,
\varphi_n(0)).
\end{split}
\end{equation*}
On the other hand, according to  the Harnak type theorem for
positive free pluriharmonic functions (see \cite{Po-hyperbolic}), we
have
$$
\text{\rm Re}\, W (r\varphi_1(0),\ldots, \varphi_n(0))\leq \Re W(0) \frac{1+r \|\varphi(0)\|}{1-r \|\varphi(0)\|}.
$$
Combining the latter two inequalities and taking $r\to 1$, we deduce
that
\begin{equation} \label{ine1}
\|p\circ \varphi\|_2^2= \|p(\widetilde\varphi_1,\ldots,
\widetilde\varphi_n)1\|^2\leq \Re W(0) \frac{1+
\|\varphi(0)\|}{1- \|\varphi(0)\|}.
\end{equation}
Using the Herglotz-Riesz
   representation \eqref{Herg}  and  relation \eqref{def-mu},
   we obtain
   $$
   W(0)=\mu_p (I) =\lim_{r\to 1}\|p(rR_1,\ldots,
   rR_n)1\|^2=\|p\|_2^2.
   $$
Hence, and using relation \eqref{ine1}, we have
\begin{equation} \label{ine2}
\|p\circ \varphi\|_2\leq \|p\|_2\,\left(\frac{1+ \|\varphi(0)\|}{1-
\|\varphi(0)\|}\right)^{1/2}
\end{equation}
for any noncommutative polynomial $p\in H^2_{\bf ball}$. Let $f(X_1,\ldots,
X_n)=\sum_{k=0}^\infty \sum_{|\alpha|=k}  c_\alpha X_\alpha$ be  a
free holomorphic function in $H^2_{\bf ball}$. Then $f\circ \varphi$
is a free holomorphic function on $[B(\cH)^n]_1$  and
\begin{equation}
\label{circ2} (f\circ \varphi)(rS_1,\ldots, rS_n)1=\sum_{k=0}^\infty
\sum_{|\alpha|=k}
 c_\alpha \varphi_\alpha (rS_1,\ldots, rS_n)1,
\end{equation}
where the convergence is in $F^2(H_n)$. Setting $p_m(X_1,\ldots,
X_n):=\sum_{k=0}^m \sum_{|\alpha|=k}  c_\alpha X_\alpha$, we have
$p_m\to f$ in $H^2_{\bf ball}$ as $m\to\infty$. Therefore,
$\{p_m\}$ is
 a Cauchy sequence in $H^2_{\bf ball}$. Due to relation \eqref{ine2}, we have
$$
\|p_m\circ\varphi- p_k\circ \varphi\|_2\leq \left(\frac{1+
\|\varphi(0)\|}{1- \|\varphi(0)\|}\right)^{1/2}\|p_m-p_k\|_2,\qquad
m,k\in \NN.
$$
Consequently, $\{p_m\circ\varphi\}$ is a Cauchy sequence
sequence in $H^2_{\bf ball}$ and  there is $g\in H^2_{\bf ball}$
such that $p_m\circ \varphi\to g$ in  $H^2_{\bf ball}$ as $m\to\infty$. Hence, and
using relation \eqref{circ2}, we deduce that
$$g(rS_1,\ldots, rS_n)1=\lim_{m\to\infty} (p_m\circ \varphi)(rS_1,\ldots, rS_n)1
=(f\circ \varphi)(rS_1,\ldots, rS_n)1,\qquad r\in[0,1).
$$
Since $f\circ \varphi$ and $g$ are free holomorphic functions, the
identity theorem for free holomorphic functions implies
 $f\circ \varphi=g$. Using that fact that  $p_m\circ \varphi\to f\circ \varphi$ in  $H^2_{\bf ball}$ and
 relation \eqref{ine2}, we obtain
 \begin{equation} \label{ine3}
\|f\circ \varphi\|_2\leq \left(\frac{1+ \|\varphi(0)\|}{1-
\|\varphi(0)\|}\right)^{1/2} \|f\|_2,\qquad f\in H^2_{\bf ball}.
\end{equation}
Since any free holomorphic automorphism of $[B(\cH)^n]_1$ is  inner,
i.e., its boundary function  with respect to $R_1,\ldots, R_n$ is an
isometry, the result above implies the right-hand inequality of the
theorem.

Now, we prove the left-hand inequality.  For each
$\mu:=(\mu_1,\ldots, \mu_n)\in \BB_n$, we define  the vector
$z_\mu:=\sum_{k=0}\sum_{|\alpha|=k} \overline{\mu}_\alpha e_\alpha$,
where $\mu_\alpha:=\mu_{i_1}\cdots \mu_{i_p}$ if
$\alpha=g_{i_1}\cdots g_{i_p}\in \FF_n^+$ and $i_1,\ldots, i_p\in
\{1,\ldots, n\}$, and $\mu_{g_0}=1$. Note that $z_\mu\in F^2(H_n)$
and $Z_\mu(X):=\sum_{k=0}\sum_{|\alpha|=k} \overline{\mu}_\alpha
X_\alpha$ is in $H^2_{\bf ball}$. Since $C_\varphi$ is a bounded
operator on $H^2_{\bf ball}$, we have
$$(C_\varphi^*Z_\mu)(X)=\sum_{k=0}\sum_{|\alpha|=k} b_\alpha X_\alpha, \qquad X\in [B(\cH)^n]_1,
 $$
 for some coefficients $b_\alpha\in \CC$ with $\sum_{\alpha\in \FF_n^+} |b_\alpha|^2<\infty$.
Since the monomials $\{X_\alpha\}_{\alpha\in \FF_n^+}$ form an
othonormal basis for $H^2_{\bf ball}$, for each $\alpha\in \FF_n^+$, we have
\begin{equation*}
\begin{split}
b_\alpha&=\left< C_\varphi^*Z_\mu, X_\alpha\right>=\left< Z_\mu, C_\varphi (X_\alpha)\right>\\
&=\left< z_\mu, \varphi_\alpha(S_1,\ldots, S_n)1\right>=\left<\varphi_\alpha(S_1,\ldots, S_n)^* z_\mu,  1\right>.
\end{split}
\end{equation*}
Since $S_i^*z_\mu=\overline{\mu}_i z_\mu$, one can see that
$\varphi_\alpha(S_1,\ldots, S_n)^* z_\mu=\overline{\varphi_\alpha(\mu)} z_\mu$.
Consequently, we deduce that $b_\alpha=\overline{\varphi_\alpha(\mu)}$, $\alpha\in \FF_n^+$, and
\begin{equation}\label{CZ}
C_\varphi^*Z_\mu=\sum_{k=0}\sum_{|\alpha|=k} \overline{\varphi_\alpha(\mu)}X_\alpha=Z_{\varphi(\mu)},\qquad  \mu:=(\mu_1,\ldots, \mu_n)\in \BB_n.
\end{equation}
A straightforward computation shows that
$$
\|C_\varphi^*Z_\mu\|=\|z_{\varphi(\mu)}\|
=\left(\frac{1}{1-\|\varphi(\mu)\|^2}\right)^{1/2}.
$$
Now, we assume that $\varphi=\Phi_\lambda\in Aut([B(\cH)^n]_1)$.
Then, using relation \eqref{E1},  we deduce that
\begin{equation*}
\begin{split}
\|C_{\Phi_\lambda} \| &=\|C_{\Phi_\lambda}^*\|\geq \frac{\|C_{\Phi_\lambda}^* Z_\mu\|}{\|Z_\mu\|}\\
&=\left(\frac{1-\|\mu\|^2}{1-\|\Phi_\lambda(\mu)\|^2}\right)^{1/2}
=\left(\frac{|1-\left<\mu,\lambda\right>|^2}{1-\|\lambda\|^2}\right)^{1/2}
\end{split}
\end{equation*}
for any $\mu\in \BB_n$.  Taking $\mu\to
-\frac{\overline{\lambda}}{\|\lambda\|}$ and using the fact that
$\Phi_\lambda(0)=\lambda$, we obtain
$$
\|C_{\Phi_\lambda} \|\geq
\left(\frac{1+\|\Phi_\lambda(0)\|}{1-\|\Phi_\lambda(0)\|}\right)^{1/2}.
$$
Combining this inequality with  relation \eqref{ine3}, we obtain
\begin{equation}
\label{CPH} \|C_{\Phi_\lambda}
\|=\left(\frac{1+\|\Phi_\lambda(0)\|}{1-\|\Phi_\lambda(0)\|}\right)^{1/2},
\end{equation}
which also shows that the right-hand inequality in the theorem is
sharp.

Now, we assume that $\varphi \in Aut([B(\cH)^n]_1)$ with $\varphi
(0)=\lambda$. Then, due to \cite{Po-automorphism}, we have $\varphi=
\Phi_\lambda\circ \Phi_U$, where $U\in B(\CC^n)$ is a unitary
operator.  Since $\Phi_U$ is inner and $\Phi_U(0)=0$, Theorem
\ref{comp1-iso} shows that $C_{\Phi_U}$ is an isometry.
Consequently, using relation \eqref{CPH}  and the fact that
$C_\varphi=C_{\Phi_U} C_{\Phi_\lambda}$, we deduce that
$$
\|C_\varphi\|=\left(\frac{1+\|\varphi(0)\|}{1-\|\varphi(0)\|}\right)^{1/2}.
$$

 Taking into account that $\Phi_\lambda\circ \Phi_\lambda=id$, we deduce that
\begin{equation*}
\begin{split}
\|f\|\leq \|C_{\Phi_\lambda}\|\|C_{\Phi_\lambda}f\|
\leq \left(\frac{1+\|\Phi_\lambda(0)\|}{1-\|\Phi_\lambda(0)\|}\right)^{1/2}
\|C_{\Phi_\lambda}f\|
\end{split}
\end{equation*}
for any $f\in H^2_{\bf ball}$. Now, we assume that $\varphi \in
Aut([B(\cH)^n]_1)$ with $\varphi (0)=\lambda$. As above, $\varphi=
\Phi_\lambda\circ \Phi_U$ and $C_\varphi=C_{\Phi_U}
C_{\Phi_\lambda}$. Since $C_{\Phi_U}$ is an isometry, the latter
inequality implies
$$
\|C_\varphi f\|=\|C_{\Phi_\lambda}
 C_{\Phi_U}f\|\geq
 \left(\frac{1-\|\varphi(0)\|}{1+\|\varphi(0)\|}\right)^{1/2}\|f\|,
$$
which shows that the left-hand inequality  of the theorem holds. To
prove that this inequality is sharp, let $g_k\in H^2_{\bf ball}$
with $\|g_k\|_2=1$ and $\|C_{\Phi_\lambda}\|=\lim_{k\to\infty}
\|C_{\Phi_\lambda} g_k\|$. Set $f_k:=C_{\Phi_\lambda} g_k$ and note
that  the inequality $
\left(\frac{1-\|\Phi_\lambda(0)\|}{1+\|\Phi_\lambda(0)\|}\right)^{1/2}
\|f_k\|\leq  \|C_{\Phi_\lambda}f_k\|$ is equivalent to $
\|C_{\Phi_\lambda} g_k\|\leq
\left(\frac{1+\|\Phi_\lambda(0)\|}{1-\|\Phi_\lambda(0)\|}\right)^{1/2},
$ which is sharp due to \eqref{CPH}, and proves our assertion. The
proof is complete.
\end{proof}

\begin{theorem}\label{comp3}
If $\varphi$ is an  inner  free holomorphic self-map of the
noncommutative ball $[B(\cH)^n]_1$, then  $C_\varphi f\in H^2_{\bf
ball}$ for all $f\in H^2_{\bf ball}$, and
$$
\left(\frac{1-\|\varphi(0)\|}{1+\|\varphi(0)\|}\right)^{1/2}\|f\|\leq
\|C_\varphi f\|\leq
\left(\frac{1+\|\varphi(0)\|}{1-\|\varphi(0)\|}\right)^{1/2}\|f\|
$$
for any $f\in H^2_{\bf ball}$. Moreover, these inequalities are best
possible and
$$
\|C_\varphi\|=\left(\frac{1+\|\varphi(0)\|}{1-\|\varphi(0)\|}\right)^{1/2}.
$$
\end{theorem}
\begin{proof}
First,   we consider the case when $\varphi$ is an   inner  free
holomorphic self-map of the  noncommutative ball $[B(\cH)^n]_1$ with
$\varphi(0)=0$. Then Theorem \ref{comp1-iso} shows that  the
composition operator $C_\varphi$ is an isometry on $ H^2 _{\bf
ball}$ and, therefore, the  theorem holds.

Now, we consider the case when $\lambda:=\varphi(0)\neq 0$. Since
$\varphi$ is a free holomorphic self-map of the noncommutative ball
$[B(\cH)^n]_1$, we must have $\|\lambda\|_2<1$. Let $\Phi_\lambda$
be the corresponding involutive free holomorphic automorphism of
$[B(\cH)^n]_1$ and let $\Psi:=\Phi_\lambda \circ \varphi$. Since
$\Phi_\lambda$ is inner and the composition of inner free
holomorphic functions is inner (see  Theorem 1.2 from
\cite{Po-holomorphic2}), we deduce that $\Psi$ is also inner. Since
$\Psi(0)=0$, the first part of the proof implies
$$\|C_\Psi f\|=\|f\|,\qquad f\in H^2_{\bf ball}.
$$
Consequently, using Theorem \ref{comp2} and the fact that
$\Phi_\lambda\circ \Phi_\lambda=id$, we get
\begin{equation}\label{vaca}
\begin{split}
\|C_\varphi f\|&=\|C_\Psi C_{\Phi_\lambda}f\|=\|C_{\Phi_\lambda}f\|
\leq \left(\frac{1+\|\Phi_\lambda(0)\|}{1-\|\Phi_\lambda(0)\|}\right)^{1/2}\|f\|\\
&=\left(\frac{1+\|\varphi(0)\|}{1-\|\varphi(0)\|}\right)^{1/2}\|f\|
\end{split}
\end{equation}
for any $f\in H^2_{\bf ball}$. Similarly, one can show that
$$
\|C_\varphi f\|=\|C_{\Phi_\lambda}f\|\geq \left(\frac{1-\|\Phi_\lambda(0)\|}{1+\|\Phi_\lambda(0)\|}\right)^{1/2}\|f\|
=\left(\frac{1-\|\varphi(0)\|}{1+\|\varphi(0)\|}\right)^{1/2}\|f\|
$$
for any $f\in H^2_{\bf ball}$. Therefore, the  inequalities in the
theorem hold. Now, we show that they are sharp. According to Theorem
\ref{comp2}, we can find $f_k\in H^2_{\bf ball}$ with $\|f_k\|_2=1$
such that
$$
\lim_{k\to \infty} \|C_{\Phi_\lambda} f_k\|=\left(\frac{1+\|\Phi_\lambda(0)\|}{1-\|\Phi_\lambda(0)\|}\right)^{1/2}.
$$
Hence,  using relation \eqref{vaca} and the fact that
$\Phi_\lambda(0)=\varphi(0)$, we obtain
$$
\lim_{k\to \infty}\|C_\varphi f_k\|=
\lim_{k\to \infty} \|C_{\Phi_\lambda} f_k\|=
\left(\frac{1-\|\varphi(0)\|}{1+\|\varphi(0)\|}\right)^{1/2},
$$
which shows that the right-hand inequality in the theorem is sharp.
Similarly, one can  show that the left-hand inequality is also
sharp. The proof is complete.
\end{proof}

Now, we can prove the main result of this section.

\begin{theorem}\label{comp4}
If $\varphi$ is a    free holomorphic self-map of  the ball
$[B(\cH)^n]_1$, then the composition operator  $C_\varphi f:=f\circ
\varphi$ is a bounded  on $ H^2_{\bf ball}$. Moreover,
$$
\frac{1}{(1-\|\varphi(0)\|^2)^{1/2}}\leq \sup_{\lambda\in
\BB_n}\left(\frac{1-\|\lambda\|^2}{1-\|\varphi(\lambda)\|^2}\right)^{1/2}\leq
\|C_\varphi\|\leq
\left(\frac{1+\|\varphi(0)\|}{1-\|\varphi(0)\|}\right)^{1/2}.
$$
\end{theorem}
 \begin{proof}
 If $\varphi(0)=0$, then the right-hand inequality follows from
 the noncommutative Littlewood subordination  principle of  Theorem \ref{comp1}.
Now, we consider the case when $\lambda:=\varphi(0)\neq 0$.
 Since  $\|\lambda\|_2<1$, let  $\Phi_\lambda$ be the corresponding  involutive free holomorphic automorphism of $[B(\cH)^n]_1$ and let $\Psi:=\Phi_\lambda \circ \varphi$.
Since $\Psi$ is   a free holomorphic self-map of  the ball
$[B(\cH)^n]_1$ with $\Psi(0)=0$, Theorem  \ref{comp1} implies
$\|C_\Psi\|\leq 1$. Using Theorem \ref{comp2} and the fact that
$\Phi_\lambda\circ \Phi_\lambda=id$, we deduce that
\begin{equation*}
\begin{split}
\|C_\varphi \|&=\|C_\Psi C_{\Phi_\lambda}\|\leq \|C_\Psi\|\|C_{\Phi_\lambda}\|\\
&\leq \left(\frac{1+\|\varphi(0)\|}{1-\|\varphi(0)\|}\right)^{1/2}.
\end{split}
\end{equation*}
 On the other hand, as in the proof of Theorem \ref{comp2}, we have
 \begin{equation*}
\begin{split}
\|C_{\varphi} \| &=\|C_{\varphi}^*\|\geq \frac{\|C_{\varphi}^*
Z_\mu\|}{\|Z_\mu\|}
=\left(\frac{1-\|\mu\|^2}{1-\|{\varphi}(\mu)\|^2}\right)^{1/2}
\end{split}
\end{equation*}
for any $\mu\in \BB_n$.  Hence, we deduce the left-hand inequality.
The proof is complete.
 \end{proof}

 Under the identification of the noncommutative Hardy  space  $H^2_{\bf ball}$ with the full Fock space  $F^2(H_n)$,
via the unitary operator $\cU:H^2_{\bf ball}\to F^2(H_n)$ defined by
$$
H^2_{\bf ball}\ni F\mapsto f:=\lim_{r\to 1} F(rS_1,\ldots,rS_n)1\in
F^2(H_n),
$$
the composition operator $C_\varphi:H^2_{\bf ball}\to H^2_{\bf
ball}$ associated with $\varphi$,  a free holomorphic self-map  of  $[B(\cH)^n]_1$,
  can be identified with the composition operator
   $C_{ \widetilde\varphi}:F^2(H_n)\to F^2(H_n)$
    defined by
 \begin{equation}
 \label{box}
 C_{ \widetilde \varphi}\left(\sum_{k=0}^\infty \sum_{|\alpha|=k}a_\alpha
 e_\alpha\right):=\lim_{r\to 1}
 \sum_{k=0}^\infty \sum_{|\alpha|=k} a_\alpha
 \varphi_\alpha(rS_1,\ldots, rS_n)1
 \end{equation}
for any $\sum_{k=0}^\infty \sum_{|\alpha|=k}a_\alpha
 e_\alpha\in F^2(H_n)$. Indeed, note that $C_{\widetilde\varphi}=\cU C_\varphi \cU^{-1}$.

A consequence of  Theorem \ref{comp4} is the following result.
\begin{corollary} \label{comp5} If $\varphi$ is a  free holomorphic self-map of  the
ball $[B(\cH)^n]_1$, then the composition operator
$C_{\widetilde\varphi}:F^2(H_n)\to F^2(H_n)$ satisfies the equation
$$
 C_{  \widetilde\varphi}\left(\sum_{k=0}^\infty \sum_{|\alpha|=k}a_\alpha e_\alpha\right)=
 \sum_{k=0}^\infty \sum_{|\alpha|=k} a_\alpha (\widetilde\varphi_\alpha
 1),
 $$
where  the convergence of the series is in $F^2(H_n)$ and
$\widetilde \varphi:=\text{\rm SOT-}\lim_{r\to 1}
 \varphi(rS_1,\ldots,rS_n)$ is the boundary function  of $\varphi$ with respect to
 the left creation operators $S_1,\ldots, S_n$.
 \end{corollary}
 \begin{proof}
 Let $\widetilde\varphi:=(\widetilde\varphi_1,\ldots, \widetilde\varphi_n)$ be the boundary of $\widetilde\varphi$ and
 let $f=\sum_{k=0}^\infty \sum_{|\alpha|=k} a_\alpha  X_\alpha$ be in $H^2_{\bf
 ball}$.
 Due to Theorem \ref{comp4} and the identification of $H^2_{\bf
 ball}$ with $F^2(H_n)$, we have
 \begin{equation}
 \label{til}
 \left\|\sum_{p\leq |\alpha|\leq m} a_\alpha \widetilde
 \varphi_\alpha
 1\right\|\leq
 \left(\frac{1+\|\varphi(0)\|}{1-\|\varphi(0)\|}\right)^{1/2}
 \left(\sum_{p\leq |\alpha|\leq m} |a_\alpha|^2\right)^{1/2}
 \end{equation}
for any $p,m\in \NN$, $p\leq m$. Consequently, since $f\in H^2_{\bf
ball}$, the sequence
 $\left\{\sum_{k=0}^m \sum_{|\alpha|=k}a_\alpha
\widetilde \varphi_\alpha 1\right\}_{m=1}^\infty$
 is Cauchy in
$F^2(H_n)$ and therefore
 convergent to an element in $F^2(H_n)$. Hence, and using relation
 \eqref{til}, we deduce that
 \begin{equation*}
\left\|\sum_{k=0}^\infty \sum_{|\alpha|=k}a_\alpha \widetilde
 \varphi_\alpha
 1\right\|\leq
 \left(\frac{1+\|\varphi(0)\|}{1-\|\varphi(0)\|}\right)^{1/2}\|f\|.
\end{equation*}
Similarly, one can show that $\sum_{k=0}^\infty \sum_{|\alpha|=k}
a_\alpha
 \varphi_\alpha(rS_1,\ldots, rS_n)1$ is in $F^2(H_n)$  and
\begin{equation*}
 \left\|\sum_{k=0}^\infty \sum_{|\alpha|=k}a_\alpha
\varphi_\alpha(rS_1,\ldots, rS_n)1\right\|\leq
 \left(\frac{1+\|\varphi(0)\|}{1-\|\varphi(0)\|}\right)^{1/2}\|f\|
\end{equation*}
 for each $r\in [0,1)$.
Consequently, taking into account   that $\widetilde
\varphi:=\text{\rm SOT-}\lim_{r\to 1}
 \varphi(rS_1,\ldots,rS_n)$, a simple approximation argument shows
 that
 $$
 \lim_{r\to 1}\sum_{k=0}^\infty \sum_{|\alpha|=k}a_\alpha
\varphi_\alpha(rS_1,\ldots, rS_n)1=\sum_{k=0}^\infty
\sum_{|\alpha|=k}a_\alpha \widetilde
 \varphi_\alpha
 1
 $$
 in $F^2(H_n)$, which  together with relation \eqref{box} completes the proof.
\end{proof}

 In this paper,  we will use  either one of the representations $C_\varphi$ or $C_{\widetilde\varphi}$
 for the composition operator with symbol $\varphi$.

\begin{corollary} \label{contr} Let $\varphi=(\varphi_1,\ldots, \varphi_n)$ be a   free holomorphic self-map of  the  noncommutative ball
$[B(\cH)^n]_1$ and let $C_\varphi$ be the composition operator on $H^2_{\bf ball}$.  Then  the following statements hold.
\begin{enumerate}
\item[(i)] $\|C_\varphi\|\geq 1$.
\item[(ii)] $C_\varphi$ is a contraction  if and only if $\varphi(0)=0$.
\item[(iii)]
$C_\varphi$ is an isometry if and only if
$\{\varphi_\alpha\}_{\alpha\in \FF_n}$ is an orthonormal   set in
$H^2_{\bf ball}$.
\end{enumerate}
\end{corollary}
\begin{proof} Since $C_\varphi 1=1$, we have $\|C_\varphi\|\geq 1$.
To prove part (ii), note that if $\|C_\varphi\|=1$, then according
to Theorem \ref{comp4}, we have
$$
\frac{1}{(1-\|\varphi(0)\|^2)^{1/2}} \leq \|C_\varphi\|=1,
$$
which implies $\varphi(0)=0$. Conversely, if $\varphi(0)=0$, the
same theorem implies $\|C_\varphi\|=1$. Now, assume that $C_\varphi$
is an isometry. Then
$$
\delta_{\alpha,\beta}=\left< C_\varphi (X_\alpha),C_\varphi( X_\beta)
\right>=\left<\varphi_\alpha, \varphi_\beta\right>,\qquad \alpha,
\beta\in \FF_n^+.
$$
 Conversely, assume that  $\{\varphi_\alpha\}_{\alpha\in \FF_n}$ is an orthonormal   set in
$H^2_{\bf ball}$. Then, for any $f=\sum_{k=0}^\infty\sum_{|\alpha|=k }
a_\alpha X_\alpha$ in the Hardy space $H^2_{\bf ball}$, we have
$$
\|C_\varphi f\|^2=\|\sum_{k=0}^\infty\sum_{|\alpha|=k } a_\alpha
\varphi_\alpha\|^2=\sum_{k=0}^\infty\sum_{|\alpha|=k }
|a_\alpha|^2=\|f\|^2.
$$
The proof is complete.
\end{proof}

Halmos' famous similarity problem \cite{H2}
   asked  whether any polynomially bounded operator is similar to a contraction.
   This long standing problem was answered  by Pisier
 \cite{Pi} in  a remarkable paper where  he shows that there are  polynomially
  bounded operator which are  not similar to  contractions.
  In what follows we show that, for compositions operators on $H^2_{\bf ball}$, similarity to contractions is equivalent  polynomial boundedness.

\begin{theorem}
\label{similarity}
Let $\varphi$ be a   free holomorphic self-map of  the  noncommutative ball
$[B(\cH)^n]_1$ and let $C_\varphi$ be the composition operator on $H^2_{\bf ball}$. Then the  following statements are equivalent:
\begin{enumerate}
\item[(i)] $C_\varphi$ is similar to a contraction;
\item[(ii)]  $C_\varphi$ is  polynomially bounded;
\item[(iii)] $C_\varphi$ is power bounded;
\item[(iv)] there is $\xi\in \BB_n$ such that $\varphi(\xi)=\xi$.
\end{enumerate}

\end{theorem}
\begin{proof} The fact that an operator similar to a contraction is
power bounded and   polynomially bounded  is  a consequence of the well-known von-Neumann inequality (\cite{von}). We
prove that $(iii)\implies  (iv)$. Assume that $C_\varphi$ is power
bounded, i.e., there is a constant $M>0$ such that
$\|C_\varphi^k\|\leq M$ for any $k\in \NN$. Note that the scalar representation of $\varphi$, i.e. $\BB_n\ni \lambda\mapsto \varphi(\lambda)\in \BB_n$, is  a holomorphic self-map of $\BB_n$.
Suppose there is no
$\xi\in \BB_n$ such that $\varphi(\xi)=\xi$. Then, due to MacCluer's
result \cite{MC1}, there is $\gamma\in \partial \BB_n$, called the
Denjoy-Wolff point of  the   map $\BB_n \ni \lambda\mapsto
\varphi(\lambda)\in \BB_n$, such that  the sequence of iterates
$\varphi^{[k]}:=\varphi\circ\cdots \circ \varphi$ converges to
$\gamma$ uniformly  on any compact subset of $\BB_n$. In particular,
we have $\|\varphi^{[k]}(0)\|\to 1$ as $k\to\infty$. On the other
hand, Theorem \ref{comp4} implies
$$
\|C_\varphi^k\|=\|C_{\varphi^{[k]}}\|\geq
\frac{1}{(1-\|\varphi^{[k]}(0)\|^2)^{1/2}}.
$$
Consequently, $\|C_\varphi^k\|\to \infty$ as $k\to\infty$, which
contradicts the fact that $C_\varphi$ is a power bounded operator.
Therefore, item (iv) holds. Finally, to prove that $(iv)\implies
(i)$, assume that there is $\xi\in \BB_n$ such that
$\varphi(\xi)=\xi$. Set $\Psi:=\Phi_\xi\circ \varphi\circ \Phi_\xi$,
where $\Phi_\xi$ is the involutive free holomorphic automorphism of
$[B(\cH)^n]_1$ associated with $\xi$.  Note that  $\Psi$ is a
bounded free holomorphic function on $[B(\cH)^n]_1$ and $\Psi(0)=0$.
Due to Theorem \ref{comp1}, we have $\|C_\Psi\|\leq 1$. On the other
hand, since $\Phi_\xi\circ \Phi_\xi=id$ and $C_\varphi=
C_{\Phi_\xi}^{-1} C_\Psi C_{\Phi_\xi}$, the result follows. The
proof is complete.
\end{proof}

\begin{corollary} \label{sp1} Let $\varphi$ be a   free holomorphic self-map of  the  noncommutative ball
$[B(\cH)^n]_1$ and let $C_\varphi$ be the composition operator on $H^2_{\bf ball}$. If there is $\xi\in \BB_n$ such that
$\varphi(\xi)=\xi$, then  the spectral radius of $C_\varphi$ is $
1$.
\end{corollary}

\begin{proof} According to the proof of Theorem \ref{similarity},
$C_\varphi$ is similar to a composition operator $C_\Psi$ with
$\Psi(0)=0$. Since $\Psi^{[k]}(0)=0$, Theorem \ref{comp1} implies
$\|C_{\Psi^{[k]}}\|=1$ for any $k\in \NN$. Consequently, we have
$$
r(C_\varphi)=r(C_\Psi)=\lim_{k\to\infty}\|C_{\Psi^{[k]}}\|^{1/k}=1.
$$
The proof is complete.
\end{proof}

\begin{corollary}\label{isometry} Let $\varphi$ be an inner  free holomorphic self-map of  the  noncommutative ball
$[B(\cH)^n]_1$ and let $C_\varphi$ be the composition operator on $H^2_{\bf ball}$.  Then the following statements hold.
 \begin{enumerate}
 \item[(i)]
 $C_\varphi$ is an isometry if and only if $\varphi(0)=0$.
 \item[(ii)] $C_\varphi$ is  similar to an isometry  if and only if there is $\xi\in \BB_n$ such that $\varphi(\xi)=\xi$.
     \end{enumerate}
\end{corollary}

\begin{proof}  Assume that $C_\varphi$ is an isometry.  Due to  Theorem \ref{comp3}, we have
$$
1=\|C_\varphi\|=\left(\frac{1+\|\varphi(0)\|}{1-\|\varphi(0)\|}\right)^{1/2}.
$$
Consequently, $\varphi(0)=0$. The converse follows also from Theorem
\ref{comp3}. Therefore, item (i) holds. The direct implication in
item (ii) follows from Theorem \ref{similarity}. To prove the
converse, assume that there is $\xi\in \BB_n$ such that
$\varphi(\xi)=\xi$  and set $\Psi:=\Phi_\xi\circ \varphi\circ \Phi_\xi$, where
$\Phi_\xi$ is the involutive free holomorphic automorphism of
$[B(\cH)^n]_1$ associated with $\xi$.

According to \cite{Po-holomorphic2}, the composition of inner free
holomorphic functions on $[B(\cH)^n]_1$ is inner. Consequently,
$\Psi$ is  an inner free holomorphic function
  and $\Psi(0)=0$. Due to part (i), the composition operator
  $C_\Psi$ is an isometry.
 Since   $C_\varphi=
C_{\Phi_\xi}^{-1} C_\Psi C_{\Phi_\xi}$, the result follows.
\end{proof}

The following result is  an extension  to our
noncommutative multivariable setting  of Cowen's (\cite{Cow}) one-variable
spectral radius formula for composition operators.

\begin{theorem} \label{spectral radius}
Let $\varphi $  be a   free holomorphic self-map of  the
noncommutative ball $[B(\cH)^n]_1$ and let $C_\varphi$ be the
composition operator on $H^2_{\bf ball}$. Then
 the spectral radius of $C_\varphi$ satisfies the relation
$$r(C_\varphi)=\lim_{k\to\infty}(1-\|\varphi^{[k]}(0)\|)^{-1/2k}. $$
Moreover,
$$
r(C_\varphi)=\lim_{k\to\infty}\left(\frac{1-\|\varphi^{[k]}(0)\|}{1-\|\varphi^{[k+1]}(0)\|}\right)^{1/2}
$$
if the latter limit exists.
\end{theorem}
\begin{proof}
 Note that  Theorem \ref{comp4} implies
\begin{equation*}
\begin{split}
\left(\frac{1}{1-\|\varphi^{[k]}(0)\|^2}\right)^{1/2k} &\leq
\|C_\varphi^k\|^{1/k} \leq
\left(\frac{1+\|\varphi^{[k]}(0)\|}{1-\|\varphi^{[k]}(0)\|}\right)^{1/2k}
\leq \left(\frac{2}{1-\|\varphi^{[k]}(0)\|}\right)^{1/2k}
\end{split}
\end{equation*}
Taking $k\to \infty$, we obtain the  first formula for the spectral radius
of $C_\varphi$.  To prove the second formula, note that
\begin{equation*}
\begin{split}
r(C_\varphi)&=\lim_{k\to\infty}(1-\|\varphi^{[k]}(0)\|)^{-1/2k}\\
&=\lim_{k\to\infty}\left(\prod_{p=0}^{k-1}\frac{1-\|\varphi^{[p]}(0)\|}
{1-\|\varphi^{[p+1]}(0)\|}\right)^{1/2k}\\
&=\lim_{k\to\infty}\left(\frac{1-\|\varphi^{[k]}(0)\|}
{1-\|\varphi^{[k+1]}(0)\|}\right)^{1/2}
\end{split}
\end{equation*}
if the latter limit exists.
The proof is complete.
\end{proof}

\section{Noncommutative Wolff theorem for free holomorphic self-maps of $[B(\cH)^n]_1$}

In this section, we use Julia type lemma for free holomorphic functions (\cite{Po-holomorphic2}) and the ideas from  the classical result obtained
by Wolff (\cite{Wo1}, \cite{Wo2}) and  MacCluer's extension to
$\BB_n$ (see \cite{MC1}), to provide a noncommutative analogue of
Wolff's theorem for free holomorphic self-maps of $[B(\cH)^n]_1$. We
also show that
  the spectral radius of a composition  operator   on $H^2_{\bf ball}$  is $1$ when  the symbol
   is  elliptic or parabolic, which extends some of Cowen's results
   (\cite{Cow}) from the single variable case.

 Julia's lemma  (\cite{J})
says that if $f:\DD\to \DD$ is an
 analytic function and there is a sequence $\{z_k\}\subset \DD$
 with $z_k\to 1$,  $f(z_k)\to 1$, and such that
 $\frac{1-|f(z_k)|}{1-|z_k|}$ is bounded, then $f$ maps each disc in
 $\DD$ tangent to $\partial \DD$ at   $1$  into a disc of the same
 kind.
 Wolff (\cite{Wo1}, \cite{Wo2}) used this result to show that if $f$
 has no fixed points in $\DD$, then there is a unique point $\xi\in \partial \DD$ such that
 any closed disc in   $\DD$ which is tangent to $\partial \DD$ at $\xi$ is mapped into
 itself by every iterate of $f$.
 The  Denjoy-Wolff theorem (\cite{Wo1}, \cite{De}) asserts that, under the above-mentioned
 conditions, the sequence of iterates of $f$  converges uniformly
  on compact subsets of $\DD$ to the constant map $g(z)=\xi$, $z\in \DD$. The point $\xi$
   is called the Denjoy-Wolff point of $f$.
 This result was extended to the unit ball of $\CC^n$ by MacCluer \cite{MC1}.

If $A,B\in B(\cK)$ are selfadjoint  operators, we say that $A<B$ if
$B-A$ is positive and invertible, i.e., there exists a constant
$\gamma>0$ such that $\left<(B-A)h,h\right>\geq \gamma\|h\|^2$ for
any $h\in \cK$. Note that  $T\in B(\cK)$ is a strict contraction
($\|T\|<1$) if and only if $TT^*<I$. For $0<c<1$ and
$\xi_1=(1,0,\ldots, 0)$, we define the noncommutative  ellipsoid
$$
{\bf E}_c(\xi_1):=\left\{ (X_1,\ldots, X_n)\in B(\cH)^n:\
\frac{[X_1-(1-c)I][X_1^*-(1-c)]I]}{c^2}+ \frac{X_2X_2^*}{c}+\cdots
+\frac{X_nX_n^*}{c}<I\right\}
$$
with center at $(1-c)\xi_1$ and containing $\xi_1$ in its boundary. If $\xi\in \BB_n$ we define the noncommutative ellipsoid ${\bf E}_c(\xi)$ centered at $(1-c)\xi$ and containing $\xi$ in its boundary in a similar manner.

 In \cite{Po-holomorphic2}, we obtained a Julia type  lemma for free holomorphic
functions.
 Let $F:[B(\cH)^n]_1\to [B(\cH)^m]_1$ be a free
holomorphic function   and $F=(F_1,\ldots, F_m)$. Let $\{z_k\}$ be a
sequence of points in $ \BB_n$  such that $\lim_{k\to \infty} z_k
=(1,0,\ldots, 0)\in
\partial\BB_n$, $\lim_{k\to \infty} F(z_k) =(1,0,\ldots, 0)\in
\partial\BB_m$, and
$$
\lim_{k\to \infty} \frac{1-\|F(z_k)\|^2}{1-\|z_k\|^2}= L <\infty. $$
  Then $L>0$ and
$$
(I-F_1(X)^*)(I-F(X)F(X)^*)^{-1} (I-F_1(X))\leq L
(I-X_1^*)(I-XX^*)^{-1}(I-X_1)
$$
for any $X=(X_1,\ldots, X_n)\in [B(\cH)^n]_1$.
   Moreover, if $0<c<1$, then
$$ F({\bf E}_c(\xi_1)) \subset {\bf E}_\gamma(\xi_1),
\quad \text{where }\ \gamma :=\frac{Lc}{1+Lc-c}.
$$

In what follows
we provide a noncommutative analogue of  Wolff's  theorem for free holomorphic self-maps of $[B(\cH)^n]_1$.

\begin{theorem} \label{Wolff}  Let $\varphi:[B(\cH)^n]_1\to [B(\cH)^n]_1$ be a free
holomorphic function such that its scalar representation  has  no  fixed points in $\BB_n$. Then there is a unique point $\zeta\in \partial \BB_n$ such that  each  noncommutative ellipsoid ${\bf E}_c(\zeta)$, $c\in (0,1)$,  is mapped into
 itself by every iterate of $\varphi$.
 \end{theorem}
 \begin{proof}  Let $r_k\in (0,1)$ be a convergent sequence to $1$. Define the map
 $\psi_k:[B(\cH)^n]_{r_k}^-\to [B(\cH)^n]_{r_k}^-$  by
 $\psi_k:=r_k \varphi(X)$, $X \in [B(\cH)^n]_{r_k}^-$, and note that $g_k$ is a
  free holomorphic  function in $[B(\cH)^n]_{r_k}^-$. Consequently,
  its scalar representation $\chi_k:[\CC^n]_{r_k}^-\to [\CC^n]_{r_k}^-$, defined by
  $\chi_k(\lambda):=\psi_k(\lambda)$, $\lambda\in [\CC^n]_{r_k}^-$, is
    holomorphic in $[\CC^n]_{r_k}^-$. According to  Brouwer  fixed point theorem
    there exists $\lambda_k\in [\CC^n]_{r_k}^-$ such that $\chi(\lambda_k)=\lambda_k$.
    Hence, $\varphi(\lambda_k)=\frac{\lambda_k}{r_k}$. Passing to a subsequence and
    taking into account that  the scalar representation of $\varphi$ has no fixed point
     in $\BB_n$, we may assume that $\lambda_k\to \zeta\in \partial \BB_n$.
 This implies that  $\varphi(\lambda_k)\to \zeta$ and
 $$
 \frac{1-\|\varphi(\lambda_k)\|^2}{1-\|\lambda_k\|^2}
 =\frac{1-\frac{1}{r_k^2}\|\lambda_k\|^2}
 {1-\|\lambda_k\|^2}<1.
 $$
 Consequently, we may assume that
$$
 \lim_{k\to\infty} \frac{1-\|\varphi(\lambda_k)\|^2}{1-\|\lambda_k\|^2}=L\leq 1.
 $$
 Without loss of generality, we may also assume that $\zeta=\xi_1:=(1,0,\ldots, 0)\in \partial\BB_n$.
  Using the above-mentioned Julia type  lemma for free holomorphic
functions, we deduce that $L>0$ and
\begin{equation}
\label{EcE}
 \varphi({\bf
E}_c(\xi_1)) \subset {\bf E}_\gamma(\xi_1), \quad \text{where }\
\gamma :=\frac{Lc}{1+Lc-c}.
\end{equation}
 Note that
$X\in {\bf E}_c(\xi_1)$ if and only if
 $$
(I-X_1)(I-X_1^*)<\frac{c}{1-c}(I-XX^*).
$$
Since $L\leq 1$, it is easy to see that  $\gamma\leq c$, which
implies $E_\gamma(\xi_1)\subseteq E_c(\xi_1)$. Combining this with
relation \eqref{EcE}, we obtain  $\varphi({\bf E}_c(\xi_1))\subseteq
{\bf E}_c(\xi_1)$ for any $c\in (0,1)$, which proves the first part
of the theorem.

To prove  the uniqueness, assume that there two distinct points $\zeta,\zeta'\in\partial \BB_n$ such that $\varphi({\bf E}_c(\zeta))\subseteq {\bf E}_c(\zeta)$  and $\varphi({\bf E}_c(\zeta'))\subseteq {\bf E}_c(\zeta')$for any $c\in (0,1)$. Let ${\bf E}^{\CC}_c(\zeta)$ be the scalar representation of  the noncommutative ellipsoid  ${\bf E}_c(\zeta)$ and let $\varphi^\CC$ be the scalar representation of $\varphi$. Choose $c,c'\in (0,1)$ such that ${\bf E}^{\CC}_c(\zeta)$ and ${\bf E}^{\CC}_{c'}(\zeta')$ are tangent to each other at some point $\xi\in \BB_n$. Note that $\varphi^{\CC}(\xi)\in \overline{{\bf E}^{\CC}_c(\zeta)}\cap
\overline{{\bf E}^{\CC}_{c'}(\zeta')}=\{\xi\}$, which contradicts the hypothesis.
The proof is complete.
\end{proof}

The point $\zeta$ of Theorem \ref{Wolff} is called the Denjoy-Wolff point of $\varphi$.
We remark that Theorem \ref{Wolff}  shows that
$$
0<\liminf_{z\to \zeta}
\frac{1-\|\varphi(z)\|^2}{1-\|z\|^2}=\alpha\leq 1.
$$
The number $\alpha$ is called the dilatation coefficient of $\varphi$. When $n=1$, $\alpha$ is the angular derivative of $\varphi$ at $\zeta$.

Combining Theorem \ref{Wolff} with Julia type  lemma for free holomorphic
functions \cite{Po-holomorphic2}, we obtain the following result.

\begin{theorem} Let $\varphi:[B(\cH)^n]_1\to [B(\cH)^n]_1$ be a free
holomorphic function  with Denjoy-Wolff point  $\zeta\in \partial
\BB_n$ and dilatation coefficient $\alpha$. Then, for any $X\in
[B(\cH)^n]_1$,
\begin{equation*}
 \begin{split}
[I-  \zeta  \varphi(X)^*][I-  \varphi(X)   \varphi(X)^*]^{-1}[I-
\varphi(X)
\zeta^*]
\leq  \alpha(I-\zeta X^*)(I-XX^*)^{-1} (I-X \zeta^*).
\end{split}
\end{equation*}

\end{theorem}

Let $\varphi:[B(\cH)^n]_1\to [B(\cH)^n]_1$ be a free holomorphic self-map.
Following the classical case,
 $\varphi$ will be called:
\begin{enumerate}
\item[(i)] {\it elliptic} if $\varphi$ fixes a point in $\BB_n$;
\item[(ii)] {\it parabolic} if $\varphi$ has no fixed points in $\BB_n$ and dilatation coefficient $\alpha=1$;
\item[(iii)] {\it hyperbolic} if $\varphi$ has no fixed points in $\BB_n$ and dilatation coefficient $\alpha<1$.
\end{enumerate}

In the single variable case, when $\varphi:\DD\to \DD$,  Cowen
(\cite{Cow})  proved that the spectral radius of the composition
operator $C_\varphi$ on $H^2(\DD)$ is $1$ if $\varphi$ is elliptic
or parabolic, and $\frac{1}{\sqrt{\alpha}}$ if $\varphi$ is
hyperbolic.
 We  can
  extend his result to  composition operators on $H^2_{\bf ball}$ when the symbol $\varphi$ is elliptic or parabolic.

\begin{theorem}\label{rad=1}  Let $\varphi $ be a   free
holomorphic self-map of  the  noncommutative ball $[B(\cH)^n]_1$. If $\varphi$ is elliptic or parabolic, then the spectral radius of the
composition operator $C_\varphi$ on $H^2_{\bf ball}$ is $1$.
\end{theorem}
\begin{proof} The case when $\varphi$ is elliptic was considered in Corollary \ref{sp1}. Now,
we assume that  $\varphi$ is parabolic and let  $\zeta\in \partial \BB_n$ be the corresponding
  Denjoy-Wolff point.
 According to MacCluer version \cite{MC1}  of Denjoy-Wolff theorem, the iterates of
  the scalar representation of $\varphi$  converge uniformly to $\zeta$ on compact subsets
  of $\BB_n$. In particular, we have $\varphi^{[k]}(0)\to \zeta$ as $k\to\infty$.
 Since the  dilatation coefficient of $\varphi$ is $1$,  we must have
 $\liminf_{k\to \infty}\left(\frac{1-\|\varphi^{[k+1]}(0)\|}{1-\|\varphi^{[k]}(0)\|}\right)^{1/2}\geq 1.
$
Consequently,
 as in  the proof of  Theorem \ref{spectral radius}, we deduce that
$$
r(C_\varphi)\leq \limsup_{k\to\infty}\left(\frac{1-\|\varphi^{[k]}(0)\|}
{1-\|\varphi^{[k+1]}(0)\|}\right)^{1/2}\leq 1.
$$
Taking into account that  $C_\varphi 1=1$, the result follows.
\end{proof}

To calculate the spectral radius of a composition  operator   on $H^2_{\bf ball}$ when
 the symbol is hyperbolic remains an open problem. Another open problem is  to
 find a Denjoy-Wolff type theorem (see \cite{De}, \cite{Wo1})  for free holomorphic self-maps of
  $[B(\cH)^n]_1$.

\bigskip
\section{Composition operators and their adjoints}

In this section, we obtain a formula for  the adjoint of  a
composition operator on $H^2_{\bf ball}$.   As a consequence we
characterize the normal composition operators   on $H^2_{\bf ball}$.
We also present a nice
 connection between Fredholm composition operators on $H^2_{\bf
 ball}$ and the automorphisms of the open unit ball $\BB_n$.

\begin{proposition} \label{adjoint}
Let $\varphi=(\varphi_1, \ldots, \varphi_n)$  be a   free
holomorphic self-map of  the  noncommutative ball $[B(\cH)^n]_1$.
Then the  adjoint of  the composition $C_\varphi$ on $H^2_{\bf
ball}$ satisfies the relation
$$
(C_\varphi^*f)(X_1,\ldots, X_n) =\sum_{\alpha\in \FF_n^+} \left<
f,\varphi_\alpha\right> X_\alpha,\qquad f\in H^2_{\bf ball}.
$$
\end{proposition}
\begin{proof} According   to Theorem \ref{comp4},  then composition operator $C_\varphi$ is  bounded
 on the Hardy space $H^2_{\bf ball}$.
If $f=\sum_{k=0}^\infty \sum_{|\alpha|=k} c_\alpha  X_\alpha$ is in $H^2_{\bf ball}$,
then,
$$
C_\varphi^*f=\sum_{k=0}\sum_{|\alpha|=k} b_\alpha X_\alpha, \qquad X\in [B(\cH)^n]_1,
 $$
 for some coefficients $b_\alpha\in \CC$ with $\sum_{\alpha\in \FF_n^+} |b_\alpha|^2<\infty$.
Since the monomials $\{X_\alpha\}_{\alpha\in \FF_n^+}$ form an othonormal basis for
$H^2_{\bf ball}$, we have
\begin{equation*}
\begin{split}
b_\alpha&=\left< C_\varphi^* f, X_\alpha\right>=\left<f, C_\varphi (X_\alpha)\right>=\left<f,
 \varphi_\alpha\right>, \qquad \alpha\in \FF_n^+.
\end{split}
\end{equation*}
The proof is complete.
\end{proof}

We remark that under the identification  of $H^2_{\bf ball}$ with
the Fock space $F^2(H_n)$, the operator $C_\varphi$ is unitarily
equivalent to $C_{\widetilde\varphi}$ (see Corollary \ref{comp5})
and
$$
C_{\widetilde\varphi}g=\sum_{\alpha\in \FF_n^+} \left< g,
\widetilde\varphi_\alpha(1)\right> e_\alpha,\qquad g\in F^2(H_n).
$$
By abuse of notation, we also write $ C_\varphi^*f=\sum_{\alpha\in
\FF_n^+}\left< f,\varphi_\alpha\right> e_\alpha,$ where $f,
\varphi_1,\ldots, \varphi_n$  are seen as elements in the Fock space
$F^2(H_n)$.

\begin{theorem}\label{normal}
Let $\varphi$ be a   free holomorphic self-map of  the
noncommutative ball $[B(\cH)^n]_1$.   Then  the composition operator
$C_\varphi$ on $H^2_{\bf ball}$ is normal   if and only if
$$
\varphi(X_1,\ldots, X_n)=[X_1,\ldots, X_n] A
$$
for some normal  scalar matrix  $A\in M_{n\times n}$  with
$\|A\|\leq 1$.
\end{theorem}
\begin{proof} Assume that  $A=[a_{ij}]_{n\times n}$  is a scalar matrix and  $\|A\|\leq 1$. Then it
is clear that the relation  $$ \varphi(X_1,\ldots, X_n)=[X_1,\ldots,
X_n] A,\qquad (X_1,\ldots, X_n)\in [B(\cH)^n]_1, $$
 defines a
bounded free holomorphic function $\varphi :[B(\cH)^n]_1\to
[B(\cH)^n]_1$. According to Theorem \ref{comp4}, the composition
operator $C_\varphi$ is bounded on $H^2_{\bf ball}$. Setting
$\varphi=(\varphi_1,\ldots, \varphi_n)$, we have  the Fock
representation  $\varphi_j=\sum_{p=1}^n a_{pj}e_{p}$ for each
$j=1,\ldots, n$. Fix $\beta=g_{i_1}\cdots g_{i_k}\in \FF_n^+$ and
let $\alpha=e_{j_1}\cdots e_{j_k}$. Note that $\left<e_\beta,
\varphi_\gamma\right>=0$ if $|\alpha|\neq |\gamma|$, $\gamma\in
\FF_n^+$, and
$$
\left<e_\beta, \varphi_\alpha\right>=\overline{a}_{i_1 j_1}\cdots
\overline{a}_{i_k j_k}.
$$
Consequently,  using  Proposition \ref{adjoint}, we deduce that
$$
C_\varphi^* e_\beta=\sum_{|\alpha|=k}\left<e_\beta,
\varphi_\alpha\right>e_\alpha=\sum_{\alpha=e_{j_1}\cdots e_{j_k},
i_1,\ldots i_k\in\{1,\ldots,n\}}\overline{a}_{i_1 j_1}\cdots
\overline{a}_{i_k j_k}e_\alpha.
$$
Now, define
 $$ \psi(X_1,\ldots, X_n)=[X_1,\ldots,
X_n] A^*,\qquad (X_1,\ldots, X_n)\in [B(\cH)^n]_1, $$
 and note that $\psi:[B(\cH)^n]_1\to
[B(\cH)^n]_1$ is a bounded free holomorphic function. Once again.
Theorem \ref{comp4} shows that  the composition operator $C_\psi$ is
bounded on   $H^2_{\bf ball}$. Setting $\psi=(\psi_1,\ldots,
\psi_n)$, we have  the Fock representation $\psi_i=\sum_{j=1}^n
\overline{a}_{ij}e_{j}$ for each $i=1,\ldots, n$. Hence, if
$\beta=g_{i_1}\cdots g_{i_k}\in \FF_n^+$, we  have
$$
C_\psi(e_\beta)=\psi_{i_1}\cdots
\psi_{i_k}=\sum_{\alpha=e_{j_1}\cdots e_{j_k}, i_1,\ldots
i_k\in\{1,\ldots,n\}}\overline{a}_{i_1 j_1}\cdots \overline{a}_{i_k
j_k}e_\alpha.
$$
This shows that $C_\varphi^*=C_\psi$. If we assume that $A$ is a
normal  matrix, then $\varphi\circ \psi=\psi\circ \varphi$. Indeed,
 for any $(X_1,\ldots, X_n)\in [B(\cH)^n]_1$, we have
 $$
(\varphi\circ \psi)(X_1,\ldots, X_n)=[X_1,\ldots, X_n]
A^*A=[X_1,\ldots, X_n]AA^*=(\psi\circ \varphi)(X_1,\ldots, X_n).
$$
Consequently, we deduce that
$$
C_\varphi C_\varphi^*= C_\varphi C_\psi=C_{\psi\circ
\varphi}=C_{\varphi\circ \psi} =C_\psi C_\varphi= C_\varphi^*
C_\varphi.
$$

Now we prove the direct implication. Assume that $\varphi$ is  a
free holomorphic self-map of  the noncommutative ball $[B(\cH)^n]_1$
and  the composition operator $C_\varphi$ is normal. Since
$C_\varphi 1=1$, the vector  $1\in F^2(H_n)$ is also an eigenvector
for  $C_\varphi^*$.  Since, due to Theorem \ref{adjoint},
$C_\varphi^*1 = \sum_{\alpha\in \FF_n^+}
\left<1,\varphi_\alpha\right> e_\alpha$, we deduce that
$\left<1,\varphi_\alpha\right>=0$ for all $\alpha\in \FF_n^+$ with
$|\alpha|\geq 1$. In particular, we have
$\left<1,\varphi_i\right>=0$
 which implies   $\varphi_i(0)=0$ for $i=1,\ldots, n$. Therefore
 $\varphi(0)=0$ and $C_\varphi^*1=1$. Consequently, we have
 $$
 \varphi(X_1,\ldots, X_n)=[X_1,\ldots, X_n]A+(\psi_1,\ldots, \psi_n)
 $$
 for some matrix  $A\in M_{n\times n}$ and  bounded free holomorphic
 functions $\psi_i=\sum_{|\alpha|\geq 2}
 c_\alpha^{(i)} e_\alpha$, $i=1,\ldots,n$. Consequently, using again the
 Fock space representation
 formula  for the adjoint of $C_\varphi$, we obtain
 $$
 C_\varphi^* (e_{g_i})=\sum_{\alpha\in \FF_n^+}
\left<e_{g_i},\varphi_\alpha\right> e_\alpha,
$$
which implies  that the subspace $\cM:=\text{\rm span} \{e_{g_i}:\
i=1,\ldots,n\}$ is invariant under $C_\varphi^*$. Since $\cM$ is
finite dimensional, it is  also invariant under $C_\varphi$ and
$C_\varphi|_{\cM}$ is a normal operator. This implies that, for each
$j=1,\ldots,n$,  $C_\varphi(e_j)$ is a linear combination of
$e_1,\ldots, e_n$ and, consequently, $\varphi(X_1,\ldots,
X_n)=[X_1,\ldots, X_n]A$ for $(X_1,\ldots, X_n)\in [B(\cH)^n]_1$.
Since $\varphi:[B(\cH)^n]_1\to [B(\cH)^n]_1$, we must have
$\|A\|\leq 1$. Setting $ \psi(X_1,\ldots, X_n)=[X_1,\ldots, X_n]
A^*$  for  $(X_1,\ldots, X_n)\in [B(\cH)^n]_1$, the first part of
the proof shows that $C_\psi$ is a bounded operator on $H^2_{\bf
ball}$ and $C_\varphi^*=C_\psi$. Since $C_\varphi$ is normal, we
have
$$
C_{\psi\circ \varphi}=C_\varphi C_\psi=C_\varphi C_\varphi^*
=C_\varphi^* C_\varphi=C_\psi C_\varphi=C_{\varphi\circ \psi},
$$
which implies $\psi\circ \varphi (X)=\varphi\circ \psi(X)$, $X\in
[B(\cH)^n]_1$. Hence, we deduce that $[X_1,\ldots,
X_n]A^*A=[X_1,\ldots, X_n]AA^*$ for any  $(X_1,\ldots, X_n)\in
[B(\cH)^n]_1$, which implies $A^*A=AA^*$. The proof is complete.
\end{proof}

 Due to Theorem \ref{normal}, characterizations of self-adjoint or
 unitary composition operators  on $H^2_{\bf ball}$ are now obvious.

\begin{lemma}\label{kernel}
Let $\varphi$ be a   free holomorphic self-map of  the
noncommutative ball $[B(\cH)^n]_1$ and let $C_\varphi$ be the
composition operator on $H^2_{\bf ball}$. If the kernel of
$C_\varphi^*$ is finite dimensional, then the scalar representation
of $\varphi$  is one-to-one.
\end{lemma}
\begin{proof}
Let $\lambda^{(j)}=(\lambda^{(j)}_1,\ldots, \lambda^{(j)}_n)$,
$j=1,\ldots, k$, be $k$ distinct points in  $\BB_n$ and fix $p\in
\{1,\ldots, k\}$. For each $ j\in \{1,\ldots, k\}$  with $j\neq p$,
there exists $q_j\in \{1,\ldots,n\}$ such that
$\lambda^{(p)}_{q_j}\neq \lambda^{(j)}_{q_j}$. Define the free
holomorphic function  $\varphi_p:[B(\cH)^n]_1\to B(\cH)$ by setting
$$
\varphi_p(X_1,\ldots, X_n)=\prod_{ j\in \{1,\ldots, k\}, j\neq
p}\frac{1}
{\lambda^{(p)}_{q_j}-\lambda^{(j)}_{q_j}}\left(X_{q_j}-\lambda^{(j)}_{q_j}I\right).
$$
Note that $\varphi_p(\lambda^{(p)})=1$ and
$\varphi_p(\lambda^{(j)})=0$  for any $j\in \{1,\ldots, k\}$ with $
j\neq p$.

For each $\mu:=(\mu_1,\ldots, \mu_n)\in \BB_n$, we define  the
vector $z_\mu:=\sum_{k=0}\sum_{|\alpha|=k} \overline{\mu}_\alpha
e_\alpha$, where $\mu_\alpha:=\mu_{i_1}\cdots \mu_{i_p}$ if
$\alpha=g_{i_1}\cdots g_{i_p}\in \FF_n^+$ and $i_1,\ldots, i_p\in
\{1,\ldots, n\}$, and $\mu_{g_0}=1$. Since  $z_\mu \in F^2(H_n)$ and
$S_i^*z_\mu=\overline{\mu}_i z_\mu$, one can see that $q(S_1,\ldots,
S_n)^* z_\mu=\overline{q(\mu)} z_\mu$ for any noncommutative
polynomial $q$.  Now we prove that  the vectors $z_{\lambda^{(1)}},
\ldots, z_{\lambda^{(k)}}$ are linearly independent. Let
$a_1,\ldots, a_k\in \CC$  be such that $a_1 z_{\lambda^{(1)}}+\cdots
+ a_kz_{\lambda^{(k)}}=0$. Due to the properties of the free
holomorphic function $\varphi_p$, $p\in \{1,\ldots, k\}$, we deduce
that
\begin{equation*}
\begin{split}
\varphi_p(S_1,\ldots, S_n)^*(a_1 z_{\lambda^{(1)}}+\cdots +
a_kz_{\lambda^{(k)}})&= a_1\overline{\varphi_p(\lambda^{(1)})}
z_{\lambda^{(1)}}+\cdots +
a_k \overline{\varphi_p(\lambda^{(k)})}z_{\lambda^{(k)}}\\
&=a_p
\overline{\varphi_p(\lambda^{(p)})}z_{\lambda^{(p)}}=a_pz_{\lambda^{(p)}}=0.
\end{split}
\end{equation*}
Hence, we deduce that $a_1=\cdots =a_k=0$, which proves  our
assertion.

Let $\psi:\BB_n\to \BB_n$ be the scalar representation of $\varphi$,
i.e., $\psi(\lambda)=\varphi(\lambda)$, $\lambda\in \BB_n$. Assume
that there is $\xi\in \BB_n$ such that $\psi^{-1}(\xi)$ is an
infinite set. Let $\{\lambda^{(j)}\}_{k\in \NN}\subset
\psi^{-1}(\xi)$ be a sequence of distinct points. Due to relation
\eqref{CZ}, we have
$C_\varphi^*(z_{\lambda^{(j)}})=C_\varphi^*(z_{\lambda^{(k)}})
=z_\xi$, which implies $z_{\lambda^{(j)}}-z_{\lambda^{(k)}}\in \ker
C_\varphi^*$. As shown above, $\{z_{\lambda^{(j)}}\}_j\in \NN$ is a
set of linearly independent vectors. Consequently, $\ker
C_\varphi^*$ is infinite dimensional, which contradicts the
hypothesis. Therefore, for each $\xi\in \BB_n$, the inverse image
$\psi^{-1}(\xi)$ is a finite set. According to Rudin's result
(Theorem 15.1.6 from \cite{Ru2}),  $\psi:\BB_n\to \BB_n$ is an open
map. Suppose that $\psi$ is not one-to-one. Let $u,v\in \BB_n$,
$u\neq v$,  be such that $\psi(u)=\psi(v)$, and let $U,V$ be open
sets in $\BB_n$ with the property that $u\in U$, $v\in V$, and
$U\cap V\neq \emptyset$. Since $\psi$ is an open map, we deduce that
$\psi(U)\cap \psi(V)$ is a nonempty open set. Consequently, we can
find sequences $\{\lambda^{(j)}\}_{j\in \NN}\subset U$ and
$\{\mu^{(j)}\}_{j\in \NN}\subset V$ of distinct points such that
$\psi(\lambda^{(j)})=\psi(\mu^{(j)})$ for all $j\in \NN$. As above,
we deduce that $z_{\lambda^{(j)}}-z_{\mu^{(j)}}\in \ker C_\varphi^*$
for $j\in \NN$.  Using the linear independence of the set
$\{z_{\lambda^{(j)}}\}_{j\in \NN}\cup \{z_{\mu^{(j)}}\}_{j\in \NN}$,
we  deduce that $\ker C_\varphi^*$ is infinite dimensional, which
contradicts the hypothesis. Therefore, $\psi$ is a one-to-one map.
The proof is complete.
\end{proof}

Note that, unlike the  single variable case, if $n\geq 2$, then the
composition operator  $C_\varphi$ is not  one-to-one on $H^2_{\bf
ball}$. For example, one can take $\varphi=(\varphi_1,
\varphi_1):[B(\cH)^2]_1\to [B(\cH)^2]_1$ and $f=e_1 e_2-e_2 e_1$,
and note that $C_\varphi f =0$.

We remark that if $\varphi\in Aut([B(\cH)^n]_1)$, then the
composition operator $C_\varphi$ is invertible on $H^2_{\bf ball}$
and therefore Fredholm. It will be interesting to see if the
converse is true. At the moment, we can prove the following result.

\begin{theorem} \label{Fredholm} Let $\varphi$ be a   free holomorphic self-map of  the
noncommutative ball $[B(\cH)^n]_1$.  If  $C_\varphi$ is a Fredholm
operator on $H^2_{\bf ball}$, then the  scalar representation of
$\varphi$ is a holomorphic automorphism of $\BB_n$.
\end{theorem}

\begin{proof} Let $\psi:\BB_n\to \BB_n$ be the scalar representation of $\varphi$, i.e., $\psi(\lambda):=\varphi(\lambda)$, $\lambda\in \BB_n$. Due to Lemma \ref{kernel},
$\psi$ is a one-to-one  holomorphic map.  We need to prove that
$\psi$ is surjective. To this end, assume that  $\psi$ is not
surjective. Then there is a sequence $\{\lambda^{(k)}\}\subset
\BB_n$ and $\zeta\in
\partial\BB_n$ such that $\lambda^{(k)}\to \zeta$ as $k\to \infty$
and $\psi(\lambda^{(k)})\to w$ for some $w\in \BB_n$.

As  we will see in the proof of Theorem \ref{angular} (see relation
\eqref{weakly}), $\frac{z_{\lambda^{(k)}}}{\|z_{\lambda^{(k)}}\|}\to
0$ weakly as $k\to \infty$.  On the other hand  taking into account
relation \eqref{CZ}, we have
\begin{equation*}
C_\varphi^*z_{\lambda^{(k)}}=\sum_{k=0}\sum_{|\alpha|=k}
\overline{\varphi_\alpha({\lambda^{(k)}})}e_\alpha=z_{\varphi({\lambda^{(k)}})},\qquad
k\in \NN.
\end{equation*}
Hence, we get
$$\left\|C_\varphi^*\left(\frac{z_{\lambda^{(k)}}}{\|z_{\lambda^{(k)}}\|}\right)\right\|
=\frac{\|z_{\varphi({\lambda^{(k)}})}\|} {\|z_{\lambda^{(k)}}\|}.
$$
Since $\|z_{\varphi({\lambda^{(k)}})}\|\to \| z_w\|<\infty$  and
$\|z_{\lambda^{(k)}}\|\to \infty$ as $k\to \infty$, we deduce that
$\left\|C_\varphi^*\left(\frac{z_{\lambda^{(k)}}}
{\|z_{\lambda^{(k)}}\|}\right)\right\|\to 0$ as $k\to\infty$.
 Denote $f_k:=\frac{z_{\lambda^{(k)}}}{\|z_{\lambda^{(k)}}\|}$.
 Since  $C_\varphi$ is a Fredholm
operator on $H^2_{\bf ball}$, there is an operator $\Lambda\in
B(F^2(H_n))$ such that  $\Lambda C_\varphi^*-I=K$ for some compact
operator $K\in B(F^2(H_n))$. Consequently, we have
\begin{equation}\label{Lambda}
\begin{split}
\|\Lambda C_\varphi^*f_k\|^2&= \|f_k+Kf_k\|^2 =\|f_k\|^2+\|Kf_k\|^2+
2 \Re \left< f_k,K f_k\right>.
\end{split}
\end{equation}
Since $K$ is a compact operator, $\|f_k\|=1$ and $f_k\to 0$ weakly
as $k\to \infty$, we must have $\|Kf_k\|\to 0$. Consequently, we
have $|\Re \left< f_k,K f_k\right>|\leq \|f_k\|\|Kf_k\|\to 0$ as
$k\to\infty$. On the other hand,  we have
 $\|C_\varphi^*f_k\|\to 0$. Now it is easy to see that relation
 \eqref{Lambda} leads to a contradiction. Therefore, $\psi$ is  surjective.
 In conclusion  $\psi$ is an automorphism of $\BB_n$.
 \end{proof}

 \bigskip

\section{Compact  composition operators on $H^2_{\bf ball}$}

 In this section
   we obtain a formula for the
  essential norm of the
composition operators $C_\varphi$ on $H^2_{\bf ball}$. In
particular, this implies  a characterization of  compact composition
operators. We show that if $C_\varphi$ is  a compact operator on
$H^2_{\bf ball}$, then the scalar representation of $\varphi$ is a
holomorphic self-map of $\BB_n$ which
  cannot have
 finite angular derivative at any point of $\partial \BB_n$ and
   has exactly one fixed
point in the open ball $\BB_n$. As a consequence, we deduce that
every  compact composition operator on $H^2_{\bf ball}$ is similar
to a contraction. In the end of this section, we prove that the  set
of compact composition operators on $H^2_{\bf ball}$ is arcwise
connected in  the set    of all composition operators.

We recall that the essential norm  of a bounded operator $T\in
B(\cH)$ is defined by
$$
\|T\|_e:=\inf \{ \|T-K\|:\ K\in B(\cH)  \text{ is compact } \}.
$$

\begin{theorem}\label{essential} Let $\varphi$ be a   free holomorphic self-map of  the
noncommutative ball $[B(\cH)^n]_1$. Then  the essential norm of the
composition operator $C_\varphi$ on $H^2_{\bf ball}$ satisfies the
equality
$$\|C_\varphi\|_e=\lim_{k\to\infty} \sup_{f\in H^2_{\bf ball}, \|f\|_2\leq 1}
\left(\sum_{ |\alpha|\geq k}|\left<f,\varphi_\alpha\right>|^2\right)^{1/2}.
$$
 Consequently,  $C_\varphi$  is a compact operator
if and only if
$$
\lim_{k\to\infty} \sup_{f\in H^2_{\bf ball}, \|f\|_2\leq 1}
 \sum_{|\alpha|\geq k}|\left<f,\varphi_\alpha\right>|^2 =0.
$$
\end{theorem}
\begin{proof}
Let $\varphi$ be a   free holomorphic self-map of  the
noncommutative ball $[B(\cH)^n]_1$. Since $C_\varphi$ is a bounded
composition operator on $H^2_{\bf ball}$
 (see Theorem \ref{comp4}), one can use standard arguments (see  Proposition 5.1 from \cite{Sha}) to show   that the essential norm of the
composition operator $C_\varphi$ on $H^2_{\bf ball}$ satisfies the
equality
\begin{equation}\label{ess-lim}
\|C_\varphi\|_e=\lim_{k\to\infty} \|C_\varphi P_{ k}\|,
 \end{equation}
 where $P_{k}$ is the orthogonal projection of $F^2(H_n)$ onto the closed linear span
  of  all $e_\alpha$ with  $\alpha\in \FF_n^+$ and $|\alpha|\geq k$.   Indeed, note that
   the sequence $\{\|C_\varphi P_{ k}\|\}_{k=1}^\infty$ is decreasing and,
 due to the fact that $I-P_k$ is a finite rank projection, we  have
$\|C_\varphi\|_e \leq \|C_\varphi P_{k}\|$ for any $k\in \NN$. Hence
$ \|C_\varphi\|_e \leq \lim_{k\to\infty} \|C_\varphi P_{ k}\|. $ On
the other hand, let $K$ be a compact operator and
 $a:=\lim_{k\to \infty} \|KP_k\|$. Assume that  $a>0$ and let $\epsilon>0$ with $0<a-\epsilon$. Then there   is a sequence   $h_{k}\in F^2(H_n)$ with $\|h_{k}\|\leq 1$, such that $\|P_{k}K^*h_{k}\|\geq a-\epsilon$ for any $k\geq N$ and some $N\in \NN$.
 Since $K^*$ is a compact operator, there is a subsequence
 $k_m\in \NN$ such that $K^*h_{k_m}\to v$ for some $v\in F^2(H_n)$. Consequently, taking into account that $P_{k_m} v\to 0$, $\|P_k\|\leq 1$, and
 $$
 \|P_{k_m}K^*h_{k_m}\|\leq \|P_{k_m} v\|+\|P_{k_m}\|\|v-K^*h_{k_m}\|,
 $$
 we deduce that $P_{k_m}K^*h_{k_m}\to 0$, which is a contradiction.
 Therefore, $\lim_{k\to \infty} \|KP_k\|=0$.  Note also that
 $$
 \|C_\varphi-K\|\geq \|(C_\varphi-K)P_k\|\geq \|C_\varphi P_k\|-\|P_kK^*\|.
 $$
 Now,  taking $k\to\infty$, we obtain $\|C_\varphi-K\|\geq \lim_{k\to\infty} \|C_\varphi P_{ k}\|$,
 which proves relation \eqref{ess-lim}.

According  to Proposition \ref{adjoint} and the remarks that follow,
we have
$$
P_k C_\varphi^*f=\sum_{|\alpha|\geq k } \left< f,\varphi_\alpha\right> e_\alpha,\qquad f\in F^2(H_n),
$$
where  $P_{k}$ is the orthogonal projection of the full Fock space
$F^2(H_n)$ onto the closed span of the vectors $ \{e_\alpha:\
\alpha\in \FF_n^+, |\alpha|\geq k\}$, and  $f$, $\varphi_1, \ldots,
\varphi_n$ are seen as elements of the Fock space $F^2(H_n)$.
 Hence, we deduce that
$$
\|P_k C_\varphi^*\|=\sup_{f\in H^2_{\bf ball}, \|f\|\leq 1}
\left( \sum_{|\alpha|\geq k } |\left< f,\varphi_\alpha\right>|^2\right)^{1/2}.
$$
Combining this  result with relation \eqref{ess-lim}, we  obtain the
formula for the essential norm of $C_\varphi$. The last part of the
theorem is now obvious.
\end{proof}

\begin{proposition} \label{comp-trace} Let $\varphi:=(\varphi_1,\ldots, \varphi_n)$ be a   free holomorphic self-map of  the
noncommutative ball $[B(\cH)^n]_1$ and let $C_\varphi$ be the
composition operator on $H^2_{\bf ball}$.  Then the following
statements hold.
\begin{enumerate}
\item[(ii)] If $\varphi$ is inner then $C_\varphi$   is  not compact.
\item[(ii)] If   \,$\|\varphi\|_\infty<1$
then  $C_\varphi$   is compact.
\item[(iii)] If $\|\varphi_1\|_\infty+\cdots +\|\varphi_n\|_\infty<1$, then $C_\varphi$ is a trace class operator.
\item[(iv)] If $\|\varphi_1\|_\infty^2+\cdots +\|\varphi_n\|_\infty^2<1$, then $C_\varphi$
 is a Hilbert-Schmidt  operator.
\end{enumerate}
\end{proposition}

\begin{proof}
To prove item (i),  assume first that $\varphi$ is an inner  free
holomorphic self-map of the  noncommutative ball $[B(\cH)^n]_1$ with
$\varphi(0)=0$. As in the proof of Theorem \ref{comp3},
$\{\varphi_\alpha\}_{\alpha\in \FF_n^+}$ is an orthonormal set in
$H^2_{\bf ball}$.   Consequently, if $\{a_\alpha\}_{|\alpha|\geq
k}\subset \CC$ is such that $\sum_{|\alpha|\geq k} |a_\alpha|^2=1$,
then $g:=\sum_{|\beta|\geq k}a_\beta \varphi_\beta$ is in $F^2(H_n)$
and $\|g\|_2=1$. Note also that
$$
\sum_{|\alpha|\geq k}|\left< g,\varphi_\alpha \right>|^2= \sum_{|\alpha|\geq k} |a_\alpha|^2=1.
$$
Since $\{\varphi_\alpha\}_{\alpha\in \FF_n^+}$ is an orthonormal set
in $H^2_{\bf ball}$, we have $\sum_{|\alpha|\geq k } |\left<
f,\varphi_\alpha\right>|^2\leq \|f\|_2$ for any $f\in H^2_{\bf
ball}$. Now, one can deduce that
$$
\sup_{f\in H^2_{\bf ball}, \|f\|\leq 1}
\left( \sum_{|\alpha|\geq k } |\left< f,\varphi_\alpha\right>|^2\right)^{1/2}=1.
$$
Due to Theorem \ref{essential}, we deduce that $\|C_\varphi\|_e=1$.
Now, we consider the case when $\xi:=\varphi(0)\neq 0$. Since the
involutive free holomorphic automorphism $\Phi_\xi$ is inner and the
composition of inner free holomorphic functions is inner (see
\cite{Po-holomorphic2}), we deduce that $\Psi:=\Phi_\xi\circ
\varphi$ is an inner  free holomorphic self-map of $[B(\cH)^n]_1$.
Since $\Psi(0)=0$,   the first part of the proof shows that $C_\Psi$
is not compact. Taking into account that $C_\Psi=C_\varphi
C_{\Phi_\xi}$, we deduce that $C_\varphi$ is not compact.

 To prove item (ii),
 let $\widetilde \varphi:=(\widetilde\varphi_1,\ldots, \widetilde\varphi_n)$ be the boundary function with respect to the left creation operators $S_1,\ldots, S_n$, and set $\|\widetilde \varphi\|=s<1$. It is easy to see that
$ \|[\widetilde\varphi_\alpha: \ |\alpha|=k]\|\leq
\|[\widetilde\varphi_1,\ldots, \widetilde\varphi_n]\|^k= s^k$, $k\in
\NN.$ For any $g\in F^2(H_n)$ and $m\in \NN$, we have
\begin{equation*}
\begin{split}
\left\| C_{\widetilde\varphi} g-\sum_{k=0}^m \sum_{|\alpha|=k}
\left<g, e_\alpha\right> \widetilde\varphi_\alpha(1)\right\| &=
\left\|\sum_{k=m+1}\sum_{|\alpha|=k}\left<g, e_\alpha\right> \widetilde\varphi_\alpha (1)\right\|\\
&\leq
\sum_{k=m+1}\left\|[\widetilde\varphi_\alpha: \ |\alpha|=k] \left[\begin{matrix}
\left<g, e_\alpha\right>\\
\vdots\\
|\alpha|=k\end{matrix}\right]\right\|\\
&\leq  \sum_{k=m+1} s^k  \left(\sum_{|\alpha|=k}|\left<g, e_\alpha\right>|^2\right)^{1/2}\\
&\leq \left(\sum_{k=m+1} s^{2k}\right)^{1/2}\left(\sum_{k=m+1}\sum_{|\alpha|=k}
|\left<g, e_\alpha\right>|^2 \right)^{1/2}\\
&\leq \|g\|_2 \frac{s^m}{\sqrt{1-s^2}}.
\end{split}
\end{equation*}
Consequently,  the operator $G_m:F^2(H_n)\to F^2(H_n)$  defined by
$G_m(g):=\sum_{k=0}^m \sum_{|\alpha|=k} \left<g, e_\alpha\right>
\widetilde\varphi_\alpha (1)$ has finite rank and  converges to the
composition operator $C_{\widetilde\varphi}$ in the operator norm
topology. Therefore, $C_\varphi$ is a compact operator.

To prove item (iii), note that
\begin{equation*}
\begin{split}
\sum_{\alpha\in \FF_n^+} \|C_{\widetilde \varphi}
e_\alpha\|&=\sum_{\alpha\in \FF_n^+} \|\widetilde \varphi_\alpha
(1)\| \leq \sum_{k=0}^\infty \sum_{|\alpha|=k}\|\widetilde
\varphi_\alpha\|\\
&\leq \sum_{k=0}^\infty (\|\widetilde \varphi_1\|+\cdots
+\|\widetilde \varphi_n\|)^k<\infty.
\end{split}
\end{equation*}
Consequently, $C_\varphi$ is a trace class operator. Finally, we
prove  item (iv).  First, note that  $C_\varphi$  is a
Hilbert-Schmidt operator   if and only if  \ $ \sum_{\alpha\in
\FF_n^+} \|\varphi_\alpha\|_2^2<\infty. $ On the other hand, as
above,  one ca show that
$$
\sum_{\alpha\in \FF_n^+} \|C_{\widetilde \varphi} e_\alpha\|^2 \leq
\sum_{k=0}^\infty (\|\widetilde \varphi_1\|^2+\cdots +\|\widetilde
\varphi_n\|^2)^k<\infty,
$$
which shows that $C_\varphi$ is a Hilbert-Schmidt  operator. The
proof is complete.
\end{proof}

\begin{corollary} If $\varphi$ is an inner    free holomorphic self-map of  the
noncommutative ball $[B(\cH)^n]_1$ such that $\varphi(0)=0$, then
the essential norm of the composition operator $C_\varphi$ on
$H^2_{\bf ball}$ is $1$.
\end{corollary}

\begin{theorem}\label{angular} Let $\varphi$ be a   free holomorphic self-map of  the
noncommutative ball $[B(\cH)^n]_1$ and let $C_\varphi$ be the
composition operator on $H^2_{\bf ball}$. Then the following
statements hold.
\begin{enumerate}
\item[(i)]  The essential norm of $C_\varphi$ on $H^2_{\bf ball}$ satisfies the
inequality
$$
\|C_\varphi\|_e\geq \limsup_{\|\lambda\|\to 1}
\left(\frac{1-\|\lambda\|^2}{1-\|\varphi(\lambda)\|^2}\right)^{1/2}.
$$
\item[(ii)] If $C_\varphi$ is  a compact  operator on $H^2_{\bf ball}$, then the scalar
representation of $\varphi$  cannot have
 finite angular derivative at any point of $\partial \BB_n$.
 \end{enumerate}
\end{theorem}

\begin{proof} For each $\mu:=(\mu_1,\ldots, \mu_n)\in \BB_n$, we define  the vector
$z_\mu:=\sum_{k=0}^\infty\sum_{|\alpha|=k} \overline{\mu}_\alpha
e_\alpha$, where $\mu_\alpha:=\mu_{i_1}\cdots \mu_{i_p}$ if
$\alpha=g_{i_1}\cdots g_{i_p}\in \FF_n^+$ and $i_1,\ldots, i_p\in
\{1,\ldots, n\}$, and $\mu_{g_0}=1$. Since  $z_\mu \in F^2(H_n)$ and
$S_i^*z_\mu=\overline{\mu}_i z_\mu$, one can see that $q(S_1,\ldots,
S_n)^* z_\mu=\overline{q(\mu)} z_\mu$ for any noncommutative
polynomial $q$.  Let $\lambda^{(j)}:=(\lambda^{(j)}_1,\ldots,
\lambda^{(j)}_n)\in \BB_n$ be such that $\|\lambda^{(j)}\|\to 1$ as
$j\to\infty$.  Since $\|z_\mu\|=\frac{1}{\sqrt{1-\|\mu\|^2}}$, we
deduce that
$$
\lim_{j\to\infty}\left<q,
\frac{z_{\lambda^{(j)}}}{\|z_{\lambda^{(j)}}\|}\right>=\lim_{j\to\infty}
\frac{q(\lambda^{(j)})}{\|z_{\lambda^{(j)}}\|}=0,
$$
where $q$ is seen as a noncommutative polynomial in $F^2(H_n)$.
Consequently,  since the unit ball of $F^2(H_n)$ is weakly compact
and the polynomials are dense in $F^2(H_n)$, there is a subsequence
$\frac{z_{\lambda^{(j_k)}}}{\|z_{\lambda^{(j_k)}}\|}$ which
converges weakly to $0$ as $j_k\to\infty$.  Since this is true for
any subsequence, we deduce that
\begin{equation} \label{weakly}
\frac{z_{\lambda^{(j)}}}{\|z_{\lambda^{(j)}}\|}\to 0 \ \text{ weakly
as } \ \|\lambda^{(j)}\|_2\to 1.
\end{equation}
If $K\in B(F^2(H_n))$ is an arbitrary compact operator, then
$\lim_{\|\lambda^{(j)}\|\to
1}\|K^*\left(\frac{z_{\lambda^{(j)}}}{\|z_{\lambda^{(j)}}\|}\right)\|=0$.
On the other hand,  due to relation \eqref{CZ}, we have
$$
\|C_\varphi^*z_
{\lambda^{(j)}}\|=\left(\frac{1}{1-\|\varphi(\lambda^{(j)})\|^2}\right)^{1/2}.
$$
Using all these facts, we deduce that
\begin{equation*}
\begin{split}
\|C_\varphi\|_e&=\inf \{ \|T-K\|:\ K\in B(\cH)
  \text{ is compact}\}\\
  &\geq \limsup_{\|\lambda^{(j)}\|\to
1}
\left\|(C_\varphi-K)^*\left(\frac{z_{\lambda^{(j)}}}{\|z_{\lambda^{(j)}}\|}\right)\right\|\\
&=\limsup_{\|\lambda^{(j)}\|\to 1}
\left\|C_\varphi^*\left(\frac{z_{\lambda^{(j)}}}{\|z_{\lambda^{(j)}}\|}\right)\right\|\\
&=\limsup_{\|\lambda^{(j)}\|\to 1}
\left(\frac{1-\|\lambda^{(j)}\|^2}{1-\|\varphi(\lambda^{(j)})\|^2}\right)^{1/2},
\end{split}
\end{equation*}
which proves  item (i).

To prove part (ii),  we recall that the Julia-Carath\' eodory
theorem in $\BB_n$ asserts that if $\psi:\BB_n\to \BB_n$ is analytic
and $\xi\in \partial \BB_n$, then $\psi$ has finite angular
derivative  at $\xi$ if and only if
$$
\liminf_{\lambda\to
\xi}\frac{1-\|\psi(\lambda)\|}{1-\|\lambda\|}<\infty,
$$
where the limit is taking as  $\lambda\to \xi$ unrestrictedly in
$\BB_n$. If $C_\varphi$ is  a compact  operator on $H^2_{\bf ball}$,
then according to part (i), we have
$$
\limsup_{{\lambda\to \xi}} \left(\frac{1-\|\lambda
\|^2}{1-\|\varphi(\lambda )\|^2}\right)^{1/2}=0.
$$
Now,  combining these results when  $\psi:\BB_n\to \BB_n$  is
defined by $\psi(\lambda):=\varphi(\lambda)$, $\lambda\in \BB_n$,
the result in part (ii) follows.
  The
proof is complete.
\end{proof}

 We need the following lemma which can be extracted from \cite{MC2}.
We include a proof for completeness.
\begin{lemma}\label{slice}
Let $\psi=(\psi_1,\ldots, \psi_n)$ be  a holomorphic self-map of the open unit ball $\BB_n$
  with the property  that
 $\psi(E(L,\zeta_1))\subseteq E(L,\zeta_1)$ for
each ellipsoid
$$
E(L,\zeta_1):=\left\{ \lambda\in \BB_n:\ |1-\left<\lambda,
\zeta_1\right>|^2\leq L(1-\|\lambda\|^2)\right\},\qquad L>0,
$$
  where
$\zeta_1:=(1,0,\ldots,n)\in \BB_n$. Then the slice function
$\phi_{\zeta_1}:\DD\to \DD$ defined by
$\phi_{\zeta_1}(z):=\psi_1(z,0\ldots,0)$, $z\in \DD$,  has the
property that
$$
\liminf_{z\to 1}\frac{1-|\phi_{\zeta_1}(z)|}{1-|z|}\leq 1.
$$
\end{lemma}
\begin{proof}   Note that  when  $w=(r,0,\ldots, 0)\in \BB_n$ with $r\in (0,1)$ and
 $L:=\frac{1-r}{1+r}$, the inclusion  $\psi(E(L,\zeta_1))\subseteq E(L,\zeta_1)$  implies
 $$
 \frac{|1-\psi_1(w)|^2}{1-\|\psi(w)\|^2}\leq L.
 $$
Hence, and using the inequality $1-|\psi_1(w)|\leq |1-\psi_1(w)|$,
we obtain
$$
\frac{1-|\psi_1(w)|}{1+|\psi_1(w)|}\leq \frac{1-r}{1+r},
$$
which implies $|\psi_1(w)|\geq r=\|w\|$ and, therefore,
$$
\frac{1-|\psi_1(w)|}{1-\|w\|}\leq 1
$$
for $w=(r,0,\ldots, 0)\in \BB_n$. The latter inequality  can  be used to complete the proof.
\end{proof}

In what follows we  also need the following lemma. Since the proof
is straightforward, we shall omit it. We denote by $H^2([B(\cH)]_1)$
the Hilbert space of all free holomorphic functions on $[B(\cH)]_1$
of the form $f(X)=\sum_{k=0}^\infty c_k X^k$ with $\sum_{k=0}^\infty
|a_k|^2<\infty$. It is easy to see that $H^2([B(\cH)]_1)$ can be
identified with the classical Hardy space $H^2(\DD)$.

\begin{lemma} Let $F:[B(\cH)^n]_1\to B(\cH)$ be a free holomorphic function and let
$\zeta_1:=(1,0,\ldots,0)\in \partial \BB_n$.  The slice function
$F_{\zeta_1}:[B(\cH)]_1\to B(\cH)$ defined by
$$
F_{\zeta_1}(Y):=F(\zeta_1Y),\qquad Y\in [B(\cH)]_1,
$$
has the  following properties.
\begin{enumerate}
\item[(i)]  $F_{\zeta_1}$ is a free holomorphic function on $[B(\cH)]_1$.
\item[(ii)] If $F\in H^2_{\bf ball} $ then $F_{\zeta_1}\in H^2([B(\cH)]_1)$ and
$\|F_{\zeta_1}\|_2\leq \|F\|_2$.
\item[(iii)] The inclusion $H^2([B(\cH)]_1)\subset H^2_{\bf ball}$ is an isometry.
\item[(iv)] Under the identification of $H^2_{\bf ball}$ with the full  Fock space  $F^2(H_n)$,
$$F_{\zeta_1}=P_{F^2(H_1)} F,
$$
where $P_{F^2(H_1)}$ is the orthogonal projection of $F^2(H_n)$ onto
$F^2(H_1)\subset F^2(H_n)$.
\item[(v)] If $F$ is bounded on $[B(\cH)^n]_1$, then $F_{\zeta_1}$ is bounded on $[B(\cH)]_1$ and
 $\|F_{\zeta_1}\|_\infty\leq \|F\|_\infty$.
\end{enumerate}
\end{lemma}

Now, we have all the ingredients to prove the following result.

\begin{theorem}\label{compact} Let $\varphi=(\varphi_1,\ldots, \varphi_n) $
be a free holomorphic  self-map of  the noncommutative ball
$[B(\cH)^n]_1$.  If $C_\varphi$ is a compact composition operator on
$H^2_{\bf ball}$, then the scalar representation of $\varphi$ is a holomorphic self-map of $\BB_n$ which  has
exactly one fixed point in the open ball $\BB_n$.
\end{theorem}
\begin{proof}

Let $\psi=(\psi_1,\ldots, \psi_n)$ be the scalar representation of $\varphi$,
  i.e.
the map  $\psi:\BB_n\to \BB_n$ defined by
$\psi(\lambda):=\phi(\lambda)$, $\lambda\in \BB_n$.
 It is clear that $\psi$ is a holomorphic self-map of the open unit ball $\BB_n$.
  Assume that $\psi$ has no fixed points in $\BB_n$. According to \cite{MC1} (see also Theorem \ref{Wolff}),
   there exists a unique  Denjoy-Wolff point $\zeta\in \partial \BB_n$ such that
    $\psi(E(L,\zeta))\subseteq E(L,\zeta)$ for
each ellipsoid $E(L,\zeta)$,  $L>0$. Without loss of generality we
can assume that $\zeta=\zeta_1:=(1,0,\ldots,0)\in \BB_n$. Then, due
to Lemma \ref{slice},
 the slice function $\phi_{\zeta_1}:\DD\to \DD$ defined by $\phi_{\zeta_1}(z):=\psi_1(z,0\ldots,0)$ has the property that
$$
\liminf_{z\to 1}\frac{1-|\phi_{\zeta_1}(z)|}{1-|z|}\leq 1.
$$
According to Julia-Carath\' eodory  theorem (see \cite{Ru2}),
$\phi_{\zeta_1}$ has finite angular derivative at $1$  which is less
than or equal to $1$. On the other hand, it is well-known  (see also
Theorem \ref{angular} when $n=1$) that if  a   composition operator
is compact on $H^2(\DD)$, then its symbol cannot have a finite
angular derivative at any point. Consequently, $C_{\phi_{\zeta_1}}$
is not a compact operator on $H^2(\DD)$.

Under the identification of $H^2_{\bf ball}$ with the full  Fock
space  $F^2(H_n)$, set
\begin{equation}
\label{GA}
\Gamma=P_{F^2(H_1)} \varphi_1,
\end{equation}
where $P_{F^2(H_1)}$ is the orthogonal projection of $F^2(H_n)$ onto $F^2(H_1)\subset F^2(H_n)$.
 According to Lemma \ref{slice}, $\Gamma:[B(\cH)]_1\to [B(\cH)]_1$ is a bounded
  free holomorphic function. Now we show that $C_{\Gamma}$ is a compact composition
  operator on $F^2(H_1)$. Let $\{f^{(m)}\}_{m=1}^\infty$ be a bounded sequence in
$F^2(H_1)$  such that $f^{(m)}\to 0$ weakly in $F^2(H_1)$. Since
$F^2(H_1)\subset F^2(H_n)$  and $F^2(H_n)=F^2(H_1)\oplus F^2(H_1)^\perp$, it is easy
to see that $f^{(m)}\to 0$ weakly in $F^2(H_n)$. Due to the compactness of $C_\varphi$
 on $F^2(H_n)$, we must have
\begin{equation}
\label{norm-lim}
\|C_\varphi f^{(m)}\|_{F^2(H_n)}\to 0\quad \text{ as }\ m\to \infty.
\end{equation}
Since $f^{(m)}\in F^2(H_1)$, it has the representation
$f^{(m)}=\sum_{k=0}^\infty a_k^{(m)}e_1^k$ for some coefficients
$a_k^{(m)}\in \CC$ with $\sum_{k=0}^\infty |a_k|^2<\infty$. Hence
$C_\varphi f^{(m)}=\sum_{k=0}^\infty a_k^{(m)} \varphi_1^k$, where
$\varphi_1$ is seen  in $F^2(H_n)$, i.e.,
$\varphi_1^k:=\widetilde\varphi_1^k (1)$, and  the convergence of
the series  is in $F^2(H_n)$. Note also that,  due to \eqref{GA},
for each $k\in \NN$, $\varphi_1^k =\Gamma^k+\chi_k$  for some
$\chi_k\in F^2(H_n)\ominus F^2(H_1)$. Consequently, we have
\begin{equation*}
\begin{split}
C_\varphi f^{(m)}&=\sum_{k=0}^\infty a_k^{(m)} \varphi_1^k=\sum_{k=0}^\infty a_k^{(m)} \Gamma^k +g\\
&=f^{(m)}\circ \Gamma +g
\end{split}
\end{equation*}
for some $g\in F^2(H_n)\ominus F^2(H_1)$. Hence, we deduce that $
\|C_\Gamma f^{(m)}\|_{F^2(H_1)}\leq \|C_\varphi
f^{(m)}\|_{F^2(H_n)}. $ Using relation \eqref{norm-lim}, we have
$\|C_\Gamma f^{(m)}\|_{F^2(H_1)}\to 0$ as $m\to\infty$. This proves
that the composition operator $C_\Gamma$ is compact on $F^2(H_1)$.
Note also that, under the natural identification of $F^2(H_1)$ with
$H^2(\DD)$, i.e., $f=\sum_{k=0}^\infty c_k e_1^k\mapsto g(z)=
\sum_{k=0}^\infty c_k z^k$, the composition operator $C_\Gamma$ on
$F^2(H_1)$ is unitarily equivalent to the composition operator
$C_{\phi_\zeta}$ on $H^2(\DD)$. Consequently, $C_{\phi_\zeta}$ is
compact, which is a contradiction. Therefore the map  $\psi$ has
fixed points in $\BB_n$.

Now we prove that $\psi$ has  only one fixed point in $\BB_n$. Assume that
 there are two distinct points $\xi^{(1)}, \xi^{(2)}\in \BB_n$ such that
  $\psi(\xi^{(1)})=\xi^{(1)}$ and $\psi(\xi^{(2)})=\xi^{(2)}$.  It is well-known (\cite{Ru2})
   that the fixed point set of the map  $\psi$ is affine.
As in the proof of Theorem \ref{comp2}, we have
\begin{equation*}
C_\varphi^*z_\mu=\sum_{k=0}\sum_{|\alpha|=k}
 \overline{\varphi_\alpha(\mu)}e_\alpha=z_{\varphi(\mu)},\qquad  \mu:=(\mu_1,\ldots, \mu_n)\in \BB_n,
\end{equation*}
where  the vector $z_\mu\in F^2(H_n)$ is defined by
$z_\mu:=\sum_{k=0}^\infty\sum_{|\alpha|=k} \overline{\mu}_\alpha e_\alpha$.
 As a consequence, we deduce that $C_\varphi^*z_\xi=z_\xi$ for any $\xi$ in
 the fixed point set $\Lambda$  of $\psi$. Since $\Lambda$  is infinite and
  according to the proof of   Lemma \ref{kernel} the vectors
$\{z_\xi\}_{\xi\in \Lambda}$ are linearly independent, we deduce that $\ker (I-C_\varphi^*)$
 is infinite dimensional. This contradicts the fact that $C_\varphi$
  is a compact operator on $ H^2_{\bf ball}$.
   In conclusion, $\psi$ has exactly on fixed point in $\BB_n$.  This  completes the proof.
\end{proof}

Combining now  Theorem \ref{compact} and Theorem \ref{similarity},
we can deduce the following similarity result.

\begin{corollary}  Every  compact composition operator on
$H^2_{\bf ball}$ is similar to a contraction.
\end{corollary}

\begin{theorem}\label{connected}
The  set of compact composition operators on $H^2_{\bf ball}$ is
arcwise connected, with respect to the operator norm topology, in
the set    of all composition operators.
\end{theorem}
\begin{proof}
Let $\varphi=(\varphi_1,\ldots, \varphi_n) $ be a  nonconstant free
holomorphic  self-map of  the noncommutative ball $[B(\cH)^n]_1$
such that  $C_\varphi$ is a compact composition operator on
$H^2_{\bf ball}$.   For each $r\in [0,1]$, consider the free
holomorphic  map $\varphi_r:[B(\cH)^n]_1\to [B(\cH)^n]_1$ defined by
$\varphi_r(X)=\varphi(rX)$, $X\in [B(\cH)^n]_1$. If
$\|\varphi\|_\infty<1$, then $\|\varphi_r\|_\infty<1$ and due to
Proposition \ref{comp-trace}, the operator $C_{\varphi_r}$ is
compact on $H^2_{\bf ball}$. Now assume that $\|\varphi\|_\infty=1$.
Since $\varphi$ is nonconstant, Theorem \ref{strict} implies
 $\|\varphi(0)\|<1$ and
the map $[0,1)\ni r\mapsto \|\varphi_r\|_\infty$ is strictly
increasing. Therefore $\|\varphi_r\|_\infty<1$ for all $r\in [0,1)$.
Using again  Proposition \ref{comp-trace},  we deduce that the
operator $C_{\varphi_r}$ is compact on $H^2_{\bf ball}$ for any
$r\in [0,1)$. Let $\cK(H^2_{\bf ball})$  denote the algebra of all
compact operators on $H^2_{\bf ball}$ and define the function
$\gamma:[0,1]\to \cK(H^2_{\bf ball})$ by setting
$\gamma(r):=C_{\varphi_r}$. Now we show that $\gamma$ is a
continuous map in the operator norm topology. Fix $r_0\in [0,1]$.
For any $g(X):=\sum_{\alpha\in \FF_n^+} a_\alpha X_\alpha \in
H^2_{\bf ball}$ set $g_r(X):=\sum_{\alpha\in \FF_n^+} a_\alpha
r^{|\alpha|}X_\alpha \in H^2_{\bf ball}$ and note that
\begin{equation} \label{gr}
\|g_r-g_{r_0}\|_2\to 0\quad \text{ as } r\to r_0.
\end{equation}
In particular, taking $g=C_\varphi f$ where $f\in H^2_{\bf ball}$
and  $\|f\|_2\leq 1$, we have $
\|(f\circ\varphi)_r-(f\circ\varphi)_{r_0}\|_2\to 0$ as  $ r\to r_0.
$ We need to show that the latter convergence is uniform with
respect to $f\in H^2_{\bf ball}$ with $\|f\|_2\leq 1$. Indeed, if we
assume the contrary, then there is $\epsilon_0>0$ such that for any
$n\in \NN$ there is $r_n\in [0,1]$ with $|r_n-r_0|<\frac{1}{n}$ and
there exists $f_n \in H^2_{\bf ball}$ with $\|f_n\|_2\leq 1$ such
that
\begin{equation}
\label{rn}
\|(f_n\circ\varphi)_{r_n}-(f_n\circ\varphi)_{r_0}\|_2>\epsilon_0.
\end{equation}
Since $C_\varphi$ is a compact operator  the image of the unit ball
of $H^2_{\bf ball}$ under $C_\varphi$ is relatively compact.
Therefore there is a subsequence $\{f_{n_k}\}$ such that
 \begin{equation}\label{fnk}
f_{n_k}\circ \varphi\to \psi\in H^2_{\bf ball}.
\end{equation}
Now, note that
\begin{equation*}
\begin{split}
\|(f_{n_k}\circ\varphi)_{r_{n_k}}-(f_{n_k}\circ\varphi)_{r_0}\|_2
&\leq \|(f_{n_k}\circ\varphi)_{r_{n_k}}-\psi_{r_{n_k}}\|_2
+\|\psi_{r_{n_k}}-\psi_{r_0}\|_2+
\|\psi_{r_0}-(f_{n_k}\circ\varphi)_{r_0}\|_2\\
&\leq  2\|f_{n_k}\circ\varphi -\psi \|_2
+\|\psi_{r_{n_k}}-\psi_{r_0}\|_2.
\end{split}
\end{equation*}
Due to  relations \eqref{gr} and \eqref{fnk}, we deduce  that
$$
\|(f_{n_k}\circ\varphi)_{r_{n_k}}-(f_{n_k}\circ\varphi)_{r_0}\|_2\to
0\quad \text{ as } r\to r_0,
$$
which contradicts  relation \eqref{rn}. Therefore
$\|C_{\varphi_r}-C_{\varphi_{r_0}}\|\to 0$ as $r\to r_0$, which
proves the continuity of the map $\gamma$. Let $\chi=(\chi_1,\ldots,
\chi_n) $ be another  nonconstant free holomorphic  self-map of  the
noncommutative ball $[B(\cH)^n]_1$ such that  $C_\chi$ is a compact
composition operator on $H^2_{\bf ball}$. As above, the  function
$\ell:[0,1]\to \cK(H^2_{\bf ball})$ given by  $\ell(r):=C_{\chi_r}$
is continuous in the operator norm topology. It remains to show that
there is a continuous mapping $\omega:[0,1]\to \cK(H^2_{\bf ball})$
such that $\omega(0)=C_{\varphi_0}$ and $\omega(1)=C_{\chi_0}$. To
this end, since $\|\varphi(0)\|<1$ and $\|\chi(0)\|<1$, we can
 define   the map $\sigma:
[0,1]\to \BB_n$ by setting $\sigma(t):=(1-t)\varphi(0) +t\chi(0)$
for  $t\in [0,1]$. Using again Proposition \ref{comp-trace}, we
deduce that $C_{\sigma(t)I}$ is a compact composition operator on
$H^2_{\bf ball}$ for any $t\in [0,1]$. Now we define
$\omega:[0,1]\to \cK(H^2_{\bf ball})$ by setting
$\omega(t):=C_{\sigma(t)I}$. To prove continuity of this map in the
operator norm topology, note that
\begin{equation}\label{C-C}
\begin{split}
\|C_{\sigma(t)I}f-C_{\sigma(t')I} f\|&=|\left<f,
z_{\sigma(t)}-z_{\sigma(t')}\right>| \leq
\|f\|_2\|z_{\sigma(t)}-z_{\sigma(t')}\|_2,
\end{split}
\end{equation}
where $z_\lambda=\sum_{\alpha\in \FF_n^+} \overline{\lambda}_\alpha
e_\alpha$ for $\lambda\in \BB_n$. On the other hand,  consider the
noncommutative Cauchy  kernel ${\bf
C}_\lambda:=(I-\overline{\lambda}_1S_1-\cdots-\overline{\lambda}_nS_n)^{-1}$,
$\lambda:=(\lambda_1,\ldots, \lambda_n)\in \BB_n$. Note that
$\|\overline{\lambda}_1S_1+\cdots+\overline{\lambda}_nS_n\|=\|\lambda\|_2<1$
and  ${\bf C}_\lambda\in F_n^\infty$ for any  $\lambda\in \BB_n$. We
have
\begin{equation*}
\begin{split}
\|z_{\sigma(t)}-z_{\sigma(t')}\|_2 &=\|({\bf C}_{\sigma(t)}-{\bf
C}_{\sigma(t')})1\|\leq \|{\bf C}_{\sigma(t)}-{\bf
C}_{\sigma(t')}\|\\
&\leq \|{\bf C}_{\sigma(t)}\|\|{\bf
C}_{\sigma(t')}\|\|\sigma(t)-\sigma(t')\|_2.
\end{split}
\end{equation*}
Consequently, since $\BB_n\ni \lambda\mapsto {\bf C}_\lambda\in
F_n^\infty$ is continuous, we deduce that $[0,1]\ni t\mapsto
z_{\sigma(t)}\in F^2(H_n)$ is continuous as well. Combining this
result  with  relation \eqref{C-C}, we deduce  the continuity of
$\omega$, which completes the proof.
\end{proof}

\bigskip

\section{  Schr\" oder equation  for noncommutative  power series and spectra  of composition operators }

In this section,   we consider a noncommutative multivariable
  Schr\" oder   type equation   and  use it to obtain results
concerning the spectrum of composition operators on $H^2_{\bf
ball}$. As a consequence,  using the results from the previous
section, we determine the spectra of compact composition operators
on $H^2_{\bf ball}$.

 First,  we provide  the following  noncommutative
 Schr\" oder   (\cite{Schr}) type result.

\begin{theorem}
\label{Schroder} Let $A\in M_{n\times n}$ be a scalar  matrix
and let $\Lambda=(\Lambda_1,\ldots\Lambda_n)$ be an $n$-tuple of
power series in noncommuting indeterminates $Z_1,\ldots, Z_n$,  of
the form
$$
\Lambda=[Z_1,\ldots, Z_n]A+[\Gamma_1,\ldots, \Gamma_n],
$$
where $\Gamma_1, \ldots, \Gamma_n$ are  noncommutative power series
containing only monomials  of degree greater than or equal to $2$.
If there is a noncommutative power series $F$ which is not
identically zero and satisfies the Schr\" oder type equation
$$
F\circ \Lambda=cF
$$
for some $c\in \CC$, then either $c=1$ or $c$ is a product of
eigenvalues of the matrix $A$.

\end{theorem}
\begin{proof}  Since $A\in M_{n\times n}$ there is a unitary matrix $U\in M_{n\times n}$ such that $U^{-1} A U$ is an upper triangular matrix. Setting
$\Phi_U=[Z_1,\ldots, Z_n] U$, the equation $F\circ \Lambda=cF$ is
equivalent to $F'\circ \Lambda'=cF'$, where $F':=\Phi_U\circ F \circ
\Phi_{U^{-1}}$ and
$$
\Lambda':=\Phi_U\circ \Lambda \circ \Phi_{U^{-1}}
=[Z_1,\ldots, Z_n]U^{-1}AU+U^{-1}[\Gamma_1,\ldots, \Gamma_n] U.
$$
 Therefore,  we can assume that $A=[a_{ij}]\in M_{n\times n}$ is  an upper triangular matrix.
We introduce a total order $\leq $ on the free semigroup $\FF_n^+$
as follows. If $\alpha, \beta\in \FF_n^+$ with $|\alpha|\leq
|\beta|$ we say that $\alpha<\beta$. If $\alpha,\beta\in \FF_n^+$
are such that $|\alpha|=|\beta|$, then $\alpha=g_{i_1}\cdots
g_{i_k}$ and $\beta=g_{j_1}\cdots g_{j_k}$ for some $i_1,\ldots,
i_k, j_1,\ldots,j_k\in \{1,\ldots,k\}$. We say that $\alpha<\beta$
if either $i_1<j_1$ or there exists $p\in \{2,\ldots,k\}$ such that
$ i_1=j_1,\ldots, i_{p-1}=j_{p-1}$ and $i_p<j_p$. It is easy to see
that relation $\leq$ is a total order on $\FF_n^+$.

According to the hypothesis and due to the fact that  $A$ is an
upper triangular matrix, we have
\begin{equation} \label{La}
\Lambda_j=\sum_{i=1}^j a_{ij} X_i+ \Gamma_j,\qquad j=1,\ldots,n.
\end{equation}
Consequently, if $\alpha=g_{i_1}\cdots g_{i_k}\in \FF_n^+$,
$i_1,\ldots i_k\in \{1,\ldots,n\}$, then
\begin{equation} \label{La-al}
\Lambda_\alpha:=\Lambda_{i_1}\cdots
\Lambda_{i_k}=\Psi^{<\alpha}+a_{i_1 i_1}\cdots a_{i_k i_k} X_\alpha
+\chi^{(\alpha)},
\end{equation}
where $\Psi^{<\alpha}$ is a power series containing only monomials
$X_\beta$ such that $|\beta|=|\alpha|$ and  $\beta<\alpha$, and
$\chi^{(\alpha)}$ is a power series containing only monomials
$X_\gamma$ with $|\gamma|\geq |\alpha|+1$.

Let  $F=\sum_{p=0}^\infty \sum_{|\alpha|=p} c_\alpha Z_\alpha$,
$c_\alpha\in \CC$,  be  a noncommutative power series   and assume
that it  satisfies the Schr\" oder type equation $ F\circ
\Lambda=\lambda F $ for some $\lambda\in \CC$ such that
$\lambda\neq 1$ and  $\lambda$ is  not a product of eigenvalues of
the matrix $A$.   We will show by induction over $p$, that
 $\sum_{|\alpha|=p} c_\alpha Z_\alpha=0$ for any $p=0,1,\ldots$. Note that the above-mentioned equation is equivalent to
 \begin{equation}\label{Sch2}
 \sum_{p=0}^\infty \sum_{|\alpha|=p} c_\alpha \Lambda_\alpha=\lambda \sum_{p=0}^\infty \sum_{|\alpha|=p} c_\alpha Z_\alpha.
 \end{equation}
Due to relation \eqref{La}, we have $c_0=\lambda c_0$. Since $\lambda\neq 1$,
 we deduce that $c_0=0$. Assume that  $c_\alpha=0$ for any $\alpha\in \FF_n^+$ with $|\alpha|< k$.
  According to  equations \eqref{La-al} and \eqref{Sch2}, we have

\begin{equation*}
\sum_{|\alpha|=k} c_\alpha \left(\Psi^{<\alpha}+ d_A(\alpha) X_\alpha
+\chi^{(\alpha)}\right)+ \sum_{p=k+1}^\infty \sum_{|\alpha|=p} c_\alpha \Lambda_\alpha=\lambda \sum_{|\alpha|=k} c_\alpha Z_\alpha +
\lambda \sum_{p=k+1}^\infty \sum_{|\alpha|=p} c_\alpha Z_\alpha,
 \end{equation*}
where $d_A(\alpha):=a_{i_1 i_1}\cdots a_{i_k i_k}$ if $\alpha=g_{i_1}\cdots g_{i_k}\in \FF_n^+$ and
$i_1,\ldots i_k\in \{1,\ldots,n\}$. Since $\chi^{(\alpha)}$ is a power series containing only monomials
$X_\gamma$ with $|\gamma|\geq |\alpha|+1$, and  the power series $\Lambda_\alpha$, $|\alpha|\geq k+1$, contains only monomials $X_\sigma$ with $|\sigma|\geq k+1$, we deduce that

\begin{equation} \label{sum-red}
\sum_{|\alpha|=k} c_\alpha \left(\Psi^{<\alpha}+ d_A(\alpha) X_\alpha
 \right) =\lambda \sum_{|\alpha|=k} c_\alpha Z_\alpha.
 \end{equation}
We arrange  the elements of  the set
$\{\alpha\in \FF_n^+: |\alpha|=k\}$ increasingly with respect to the total  order, i.e., $\beta_1<\beta_2<\cdots <\beta_{n^k}$.  Note that $\beta_1= g_1^k$ and $\beta_{n^k}=g_n^k$. The relation \eqref{sum-red} becomes

\begin{equation} \label{sum-red2}
\sum_{j=1}^{n^k}\left(c_{\beta_j} \Psi^{< \beta_j}+ c_{\beta_j} d(\beta_j) X_{\beta_{\beta_j}}\right)
=\lambda \sum_{j=1}^{n^k} c_{\beta_j} X_{\beta_j}.
\end{equation}
Taking into account that $\Psi^{<\alpha}$ is a power series containing only monomials
$X_\beta$ such that $|\beta|=|\alpha|$ and  $\beta<\alpha$, one can see that the
 monomial $X_{\beta_{n^k}}$ occurs just once in the left-hand side of  relation \eqref{sum-red2}.
Identifying the coefficients of the monomial $X_{\beta_{n^k}}$ in the equality
\eqref{sum-red2}, we deduce that
$$
c_{\beta_{n^k}} d(\beta_{n^k})=\lambda c_{\beta_{n^k}}.
 $$
 Since $\lambda\neq a_{nn}^k=d(\beta_{n^k})$, we must have $c_{\beta_{n^k}}=0$.
 Consequently, equation \eqref{sum-red2} becomes
 \begin{equation*}
\sum_{j=1}^{n^k-1}\left(c_{\beta_j} \Psi^{< \beta_j}+ c_{\beta_j} d(\beta_j) X_{\beta_{\beta_j}}\right)
=\lambda \sum_{j=1}^{n^k-1} c_{\beta_j} X_{\beta_j}.
\end{equation*}
Continuing the process, we deduce that $c_{\beta_j}=0$ for $j=1,\ldots, n^k$. Therefore $c_\alpha=0$ for any $\alpha\in \FF_n^+$ with $|\alpha|=k$, which completes our induction.
The proof is complete.
\end{proof}

\begin{corollary}\label{Sch-holo} Let $\varphi=(\varphi_1,\ldots, \varphi_n)
$ be a free holomorphic  self-map of  the noncommutative ball
$[B(\cH)^n]_1$  such that
$\varphi(\xi)=\xi$ for some $\xi\in \BB_n$.
If there is a free holomorphic function $f:[B(\cH)^n]_1\to B(\cH)$ such that
$$
f\circ \varphi =cf
$$
for some $c\in \CC$, then either $c=1$ or $c$ is a product of
eigenvalues of the matrix
$$
\left[\left<\psi_i,e_j\right>\right]_{n\times n},
$$
 where $\psi=(\psi_1,\ldots, \psi_n):=\Phi_\xi\circ \varphi \circ
\Phi_\xi$ and  $\Phi_\xi$ is the involutive free holomorphic
automorphism of $[B(\cH)^n]_1$ associated with $\xi\in \BB_n$, and
$\psi_1,\ldots, \psi_n$ are seen as elements in the Fock space
$F^2(H_n)$.
 \end{corollary}
\begin{proof} Note that $\psi(0)=0$ and the equation  $f\circ \varphi =cf$ is equivalent to the equation $f'\circ \psi =cf'$, where
$f':=\Phi_\xi\circ f \circ\Phi_\xi$. Applying Theorem \ref{Schroder}
to the power series associated with $\psi$ and $f'$ the result follows.
\end{proof}

\begin{theorem}
\label{point-spectrum}   Let $\varphi=(\varphi_1,\ldots, \varphi_n)
$ be a free holomorphic  self-map of  the noncommutative ball
$[B(\cH)^n]_1$  such that $\varphi(0)=0$, and let $C_\varphi$ be the
associated composition operator on $H^2_{\bf ball}$.   Then the
point spectrum of $C_\varphi^*$ contains   the conjugates of all
possible products of the eigenvalues of the matrix
$$
\left[\left<\varphi_i,e_j\right>\right]_{n\times n},
$$
where $\psi_1,\ldots, \psi_n$ are seen as elements in the Fock space
$F^2(H_n)$.
\end{theorem}
\begin{proof}
For each $m=0,1,\ldots$, consider the subspace  $\cK_m:=\text{\rm
span} \{ e_\alpha:\ \alpha\in \FF_n^+, |\alpha|\leq m\}$. Since
$\varphi(0)=0$, we have $\left< C_\varphi^* e_\alpha,
e_\beta\right>=\left< e_ \alpha, \varphi_\beta\right> =0$ for any
$\alpha,\beta\in \FF_n^+$ with $|\alpha|\leq m$ and $|\beta|\geq
m+1$. This implies $C_\varphi^*(\cK_m)\subseteq \cK_m$ and
$C_\varphi^*$ has the matrix representation
$$
C_\varphi^*=\left[\begin{matrix} C_\varphi^*|_{\cK_m}&*\\
0&P_{F^2(H_n)\ominus \cK_m}C_\varphi^*|_{F^2(H_n)\ominus
\cK_m}\end{matrix}\right]
$$
with respect to the orthogonal decomposition  $F^2(H_n)=\cK_m\oplus
(F^2(H_n)\ominus \cK_m)$, and $\sigma_p(C_\varphi^*|_{\cK_m})\subset
\sigma_p(C_\varphi^*)$, where $\sigma_p(T)$ denotes the point
spectrum of $T$. Moreover, since $\cK_m$ is finite dimensional, we
have
\begin{equation*}
  \sigma(C_\varphi^*)=\sigma(C_\varphi^*|_{\cK_m})\cup
\sigma(P_{F^2(H_n)\ominus \cK_m}C_\varphi^*|_{F^2(H_n)\ominus
\cK_m}).
\end{equation*}
Since $C_\varphi^*(\cK_{m-1})\subseteq \cK_{m-1}$ we have the matrix
decomposition
$$
C_\varphi^*|_{\cK_m}=\left[\begin{matrix} C_\varphi^*|_{\cK_m}&*\\
0&P_{\cK_m\ominus \cK_{m-1}}C_\varphi^*|_{\cK_m\ominus
\cK_{m-1}}\end{matrix}\right]
$$
with respect to the orthogonal decomposition  $F^2(H_n)=\cK_m\oplus
(\cK_m\ominus \cK_{m-1})$. Consequently, we have
$$
\sigma_p(C_\varphi^*|_{\cK_m})=\sigma_p(C_\varphi^*|_{\cK_{m-1}})\cup
\sigma_p(P_{\cK_m\ominus \cK_{m-1}}C_\varphi^*|_{\cK_m\ominus
\cK_{m-1}})
$$
for any $m=1,2\ldots$. Iterating  this formula, we get
\begin{equation}
\label{sig} \sigma_p(C_\varphi^*|_{\cK_m})=\{1\}\cup
\bigcup_{j=1}^m\sigma_p(P_{\cK_j\ominus
\cK_{j-1}}C_\varphi^*|_{\cK_j\ominus \cK_{j-1}}).
\end{equation}
Now, we determine $\sigma_p(P_{\cK_k\ominus
\cK_{k-1}}C_\varphi^*|_{\cK_k\ominus \cK_{k-1}})$ for
$k=1,2,\ldots$. As in the proof of Theorem \ref{Schroder}, we can
assume that
$$\varphi(X)=[X_1,\ldots, X_n]A+ (\Gamma_1(X),\ldots, \Gamma_n(X)), \qquad X=(X_1,\ldots,
X_n)\in [B(\cH)^n]_1,
$$
where $A=[a_{ij}]\in M_{n\times n}$ is an upper triangular scalar
matrix and $\Gamma_1,\ldots, \Gamma_n$ are free holomorphic
functions on $[B(\cH)^n]_1$ containing only monomials of degree
greater  than or equal to $2$. Consequently, using the Fock space
representation of $\varphi_1,\ldots, \varphi_n$ and
$\Gamma_1,\ldots, \Gamma_n$, we
  have
\begin{equation} \label{La1}
\varphi_j=\sum_{i=1}^j a_{ij} e_i+ \Gamma_j,\qquad j=1,\ldots,n,
\end{equation}
where $\Gamma_j\in F^2(H_n)\ominus \text{\rm span}\{e_\alpha:\
|\alpha|\leq 1\}$. Note that  the matrix $
\left[\left<\varphi_i,e_j\right>\right]_{n\times n} $ is upper
triangular and its eigenvalues are $a_{11},\ldots, a_{nn}$. Using
relation \eqref{La1}, one can see that  if $\alpha=g_{i_1}\cdots
g_{i_k}\in \FF_n^+$, $i_1,\ldots i_k\in \{1,\ldots,n\}$, then
\begin{equation} \label{La-al1}
\varphi_\alpha:=\varphi_{i_1}\cdots
\varphi_{i_k}=\psi^{<\alpha}+a_{i_1 i_1}\cdots a_{i_k i_k} e_\alpha
+\chi^{(\alpha)},
\end{equation}
where $\psi^{<\alpha}\in  \text{\rm span}\{e_\beta:\
|\beta|=|\alpha| \text{  and  } \beta<\alpha\}$ and
$\chi^{(\alpha)}\in  \text{\rm span}\{e_\gamma:\
|\gamma|\geq|\alpha|+1 \}$.

We arrange  the elements of  the set $\{\alpha\in \FF_n^+:
|\alpha|=k\}$ increasingly with respect to the total  order
introduced in the proof of Theorem \ref{Schroder}, i.e.,
$\beta_1<\beta_2<\cdots <\beta_{n^k}$. We denote
$d_A(\alpha):=a_{i_1 i_1}\cdots a_{i_k i_k}$ if
$\alpha=g_{i_1}\cdots g_{i_k}\in \FF_n^+$ and $i_1,\ldots i_k\in
\{1,\ldots,n\}$.
 Note that $\varphi_{\beta_1}= d(\beta_1)e_{\beta_1} +\chi^{\beta_1}$ and
 $$
 \varphi_{\beta_i}=\left(\sum_{j=1}^i b_{\beta_{j-1}} e_{\beta_{j-1}}\right)+d(\beta_i)e_{\beta_i} +\chi^{\beta_i}\qquad \text{ if  } \ 2\leq i\leq n^k,
 $$
for some $b_{\beta_{j-1}}\in \CC$, $j=1,\ldots,i$. Using these
relations, we deduce that
\begin{equation*}
\begin{split}
\left<P_{\cK_k\ominus \cK_{k-1}}C_\varphi^*|_{\cK_k\ominus
\cK_{k-1}} e_{\beta_j},
e_{\beta_i}\right>&=\overline{\left<\varphi_{\beta_i},e_{\beta_j}\right>}
=\begin{cases} \overline{d(\beta_i)}\quad  &\text{ if } i=j\\
0\quad &\text{ if } i< j.
\end{cases}
\end{split}
\end{equation*}
This shows that  the matrix of $P_{\cK_k\ominus
\cK_{k-1}}C_\varphi^*|_{\cK_k\ominus \cK_{k-1}}$  with respect to
the orthonormal  basis $\{e_{\beta_i}\}_{i=1}^{n^k}$ is lower
triangular with the diagonal entries $\overline{d(\beta_1)},\ldots,
\overline{d(\beta_{n^k})}$. Therefore $\sigma_p(P_{\cK_k\ominus
\cK_{k-1}}C_\varphi^*|_{\cK_k\ominus \cK_{k-1}})$ consists of these
diagonal entries. On the other hand, due to  relation \eqref{sig},
we have
$$
\{1\}\cup \bigcup_{j=1}^\infty\sigma_p(P_{\cK_j\ominus
\cK_{j-1}}C_\varphi^*|_{\cK_j\ominus \cK_{j-1}})\subset
\sigma_p(C_\varphi^*).
$$
The proof is complete.
\end{proof}

Theorem \ref{point-spectrum} and Corollary  \ref{Sch-holo} imply the
following result concerning the spectrum of composition operators on the noncommutative Hardy space
$H^2_{\bf ball}$.

\begin{theorem} \label{inclusions} Let $\varphi
$ be a free holomorphic  self-map of  the noncommutative ball
$[B(\cH)^n]_1$  such that  its scalar representation has a fixed
point $\xi\in \BB_n$, and let $C_\varphi$ be the associated
composition operator on $H^2_{\bf ball}$. Then
$$
\sigma_p(C_\varphi)\subseteq \{1\}\cup \cP_{eig}\subseteq
\sigma(C_\varphi),
$$
where $\cP_{eig}$ is the set of all possible products of eigenvalues
of the matrix
$
\left[\left<\psi_i,e_j\right>\right]_{n\times n},
$
 where $\psi=(\psi_1,\ldots, \psi_n):=\Phi_\xi\circ \varphi \circ
\Phi_\xi$ and  $\Phi_\xi$ is the involutive free holomorphic
automorphism of $[B(\cH)^n]_1$ associated with $\xi\in \BB_n$.
\end{theorem}
\begin{proof} The first inclusion follows from Corollary  \ref{Sch-holo}. To prove the second inclusion note that $C_\varphi 1=1$ and $C_\psi=C_{\Phi_\xi} C_\varphi C_{\Phi_\xi}^{-1}$. Consequently,
$1\in \sigma(C_\varphi)=\sigma(C_\psi)$.  Since $\psi(0)=0$, we can apply Theorem \ref{point-spectrum} to the composition operator $C_\psi$ and complete the proof.
\end{proof}

Now we can determine the spectra of compact composition operators on $H^2_{\bf ball}$.
\begin{theorem}   Let $\varphi$
be a free holomorphic  self-map of  the noncommutative ball
$[B(\cH)^n]_1$.  If $C_\varphi$ is a compact composition operator on
$H^2_{\bf ball}$, then   the scalar representation of $\varphi$ has a unique fix point $\xi\in \BB_n$  and   the spectrum $
\sigma(C_\varphi)$ consists of \ $0$, $1$, and
 all
possible products of the eigenvalues of the matrix
$$
\left[\left<\psi_i,e_j\right>\right]_{n\times n},
$$
 where $\psi=(\psi_1,\ldots, \psi_n):=\Phi_\xi\circ \varphi \circ
\Phi_\xi$ and  $\Phi_\xi$ is the involutive free holomorphic
automorphism of $[B(\cH)^n]_1$ associated with $\xi\in \BB_n$, and
$\psi_1,\ldots, \psi_n$ are seen as elements in the Fock space
$F^2(H_n)$.
\end{theorem}
\begin{proof} If $C_\varphi$ is a compact composition operator on
$H^2_{\bf ball}$, then, according to  Theorem \ref{compact}, the scalar representation of $\varphi$ has a unique fix point $\xi\in \BB_n$. On the other hand, it is well-known that any nonzero  point in the spectrum  of a compact operator is an eigenvalue. Using Theorem \ref{inclusions}, we deduce that
$$
\sigma_p(C_\varphi)\subseteq \{1\}\cup \cP_{eig}\subseteq
\{0\}\cup \sigma_p(C_\varphi),
$$
where $\cP_{eig}$ is the set of all possible products of eigenvalues
of the matrix
$
\left[\left<\psi_i,e_j\right>\right]_{n\times n}.
$
Hence the result follows and the proof is complete.
\end{proof}

In \cite{MC2},  MacCluer  determined the spectrum of
 composition operators on $H^2(\BB_n)$ when the
symbols are  automorphisms of $\BB_n$ which fix at least one point
in $\BB_n$. The following theorem is an extension of this result to
 compositions operators on $H^2_{\bf ball}$ induced by free holomorphic automorphisms of $[B(\cH)^n]_1$.

\begin{theorem} \label{one-auto} Let $\varphi\in Aut(B(\cH)^n_1)$ be such that
$\varphi(\xi)=\xi$ for some $\xi\in \BB_n$.
Then the spectrum  of the composition operator $ C_\varphi $  on $H^2_{\bf ball}$  is the closure of all
possible products of the eigenvalues of the matrix
$$
\left[\left<\psi_i,e_j\right>\right]_{n\times n},
$$
 where $\psi=(\psi_1,\ldots, \psi_n):=\Phi_\xi\circ \varphi \circ
\Phi_\xi$  and  $\Phi_\xi$ is the involutive free holomorphic
automorphism of $[B(\cH)^m]_1$ associated with $\xi\in \BB_n$.
Moreover, $ \sigma(C_\varphi)$ is either the unit circle $\TT$, or a
finite subgroup of  $\TT$.

\end{theorem}
\begin{proof} Note that $\psi\in Aut(B(\cH)^n_1)$ and $\psi(0)=0$. According to \cite{Po-automorphism},
the free holomorphic automorphism $\psi$ has the form
$\psi(X)=[X_1,\ldots, X_n] U$ for some unitary matrix $U\in
M_{n\times n}$. It is easy to see that
$U=\left[\left<\psi_i,e_j\right>\right]_{n\times n}$. Since $U$ is
unitary there is another unitary matrix $W\in M_{n\times n}$ such
that
$$W^{-1} UW=\left[\begin{matrix} w_1&0&\cdots &0\\
0&w_2&\cdots &0\\
0&0&\cdots & w_n
\end{matrix}
\right],$$ where $w_1,\ldots, w_n$ are the eigenvalues of $U$. Set
$\chi:= \psi_W\circ \psi \circ \psi_W^{-1}$, where
$\psi_W(X):=[X_1,\ldots, X_n] W$ for $X:=[X_1,\ldots, X_n]\in
[B(\cH)^n]_1$. Note that $\chi(X)=[X_1,\ldots, X_n] W^{-1} UW$ and
$C_\chi=C_{\psi_W}^{-1}C_{\phi_\xi}^{-1}C_\varphi C_{\phi_\xi}
C_{\psi_W}$. Hence, $\sigma(C_\chi)=\sigma(\psi)=\sigma(\varphi)$.
Now we determine the spectrum of $C_\chi$. Since $C_\psi$ is
invertible and $\psi(0)=0$,  Theorem \ref{comp4} implies
$\|C_\chi\|=\|C_\psi^{-1}\|=1$. Therefore, $\sigma(C_\chi)\subseteq
\TT$. Using  now  Theorem \ref{inclusions}, we deduce that
 $\overline{\cP}_{eig}\subseteq
\sigma(C_\psi)\subseteq \TT$,
where $\cP_{eig}$ is the set of all possible products of eigenvalues
of the matrix $U$. It is obvious that if $\overline{\cP}_{eig}=\TT$, then $\sigma(C_\psi)= \TT$. When $\overline{\cP}_{eig}\neq \TT$, then $\overline{\cP}_{eig}$ is a finite subgroup  of $\TT$. Consequently, there is $m\in\NN$ such that
$\overline{\cP}_{eig}=\{z\in \TT:\ z^m=1\}$. This implies $w_j^m=1$ for $j=1,\ldots,n$ and $C_\chi^{m}=I$. Consequently, if $\lambda\in \sigma(C_\chi)$ then $\lambda^m\in
\sigma(C_\chi^m)=\{1\}$. This shows that $\lambda\in \overline{\cP}_{eig}$ and completes the proof.
\end{proof}

Comparing our Theorem \ref{one-auto} with MacCluer result (see
Theorem 3.1 from \cite{MC2}),  we are led  to the conclusion that if
$\varphi\in Aut(B(\cH)^n_1)$ has at least one fixed point in
$\BB_n$, then the spectrum  of the composition operator $ C_\varphi
$  on $H^2_{\bf ball}$ coincides  with the spectrum  of the
composition operator $ C_{\varphi^\CC} $  on $H^2(\BB_n)$, where
$\varphi^\CC$ is the scalar representation of $\varphi$.

\begin{theorem} \label{two-auto} If  $\varphi\in Aut(B(\cH)^n_1)$  and there is only one point $\zeta\in \overline{\BB}_n$ such that $\varphi(\zeta)=\zeta$ and $\zeta\in \partial \BB_n$, then  the spectral radius of the composition operator $C_\varphi$ on $H^2_{\bf ball}$ is equal to $1$ and $\sigma(C_\varphi)\subseteq \TT$.
\end{theorem}
\begin{proof} The proof  that the spectral radius is $1$ is similar to that of Theorem \ref{rad=1},
 in the parabolic case. The inclusion  $\sigma(C_\varphi)\subseteq \TT$ is due to the fact that
$\varphi^{-1}(\zeta)=\zeta$ and, according  to the first part of the
theorem we have $r(C_\varphi^{-1})=r(C_\varphi)=1$.
\end{proof}

\bigskip
\section{Composition operators on Fock spaces associated to noncommutative varieties}

In this section, we consider composition operators on Fock spaces
associated to noncommutative varieties in unit ball $[B(\cH)^n]_1$
and obtain results concerning boundedness, norm estimates, and
spectral radius. In particular, we show that
 many of our  results have
commutative counterparts  for composition operators on the symmetric
Fock space   and  on spaces of analytic functions in the unit ball
of $\CC^n$. In particular, we obtain new proofs for some of Jury's
(\cite{Ju}) recent results concerning compositions operators on
$\BB_n$.

Let $\cP_0$ be a set on noncommutative polynomials in $n$ indeterminates  such that $p(0)=0$
for all $p\in \cP_0$.
Consider the noncomutative variety $\cV_{\cP_0}(\cH)\subseteq [B(\cH)^n]_1$ defined by
$$
\cV_{\cP_0}(\cH):=\{(X_1,\ldots, X_n)\in [B(\cH)^n]_1:\ p(X_1,\ldots, X_n)=0 \text{ for all } p\in \cP_0\}.
$$
Let
$$\cM_{\cP_0}:=\overline{\text{\rm span}}\,\{S_\alpha p(S_1,\ldots, S_n)S_\beta 1:\ p\in \cP_0, \alpha,\beta\in \FF_n^+\}
$$
and $\cN_{\cP_0}:=F^2(H_n)\ominus \cM_{\cP_0}$. We remark that $1\in
\cN_{\cP_0}$ and the subspace $\cN_{\cP_0}$ is  invariant under
$S_1^*,\ldots, S_n^*$ and  $R_1^*,\ldots, R_n^*$. Define the {\it
constrained  left} (resp.~{\it right}) {\it creation operators} by
setting
$$B_i:=P_{\cN_{\cP_0}} S_i|_{\cN_{\cP_0}}\quad \text{and}\quad W_i:=P_{\cN_{\cP_0}} R_i|_{\cN_{\cP_0}},\quad i=1,\ldots, n.
$$
We proved in \cite{Po-varieties} that the $n$-tuple $(B_1,\ldots,
B_n)\in \cV_{\cP_0}(\cN_{\cP_0})$ is the universal model associated
with  the noncommutative variety $\cV_{\cP_0}(\cH)$. Let $F_n^\infty
(\cV_{\cP_0})$ be the $w^*$-closed algebra generated by $B_1,\ldots,
B_n$ and the identity. The $w^*$ and WOT topologies coincide on this
algebra  and
$$
F_n^\infty (\cV_{\cP_0})=P_{\cN_{\cP_0}}F_n^\infty
|_{\cN_{\cP_0}}=\{f(B_1,\ldots, B_n):\ f \in F_n^\infty\},
$$
where if $f$ has the Fourier representation $\sum_{\alpha\in
\FF_n^+} a_\alpha S_\alpha$ then
$$f(B_1,\ldots, B_n)=\text{\rm SOT-}\lim\limits_{r\to 1} \sum_{k=0}^\infty \sum_{|\alpha|=k}
r^{|\alpha|} a_\alpha B_\alpha.
$$
The latter limit exists due   to the $F_n^\infty$-functional
calculus for row contractions \cite{Po-funct}. Similar results hold
for $R_n^\infty (\cV_{\cP_0})$, the $w^*$-closed algebra generated
by $W_1,\ldots, W_n$ and the identity. Moreover,
$$
F_n^\infty (\cV_{\cP_0})^\prime=R_n^\infty (\cV_{\cP_0})\ \text{ and } \ R_n^\infty (\cV_{\cP_0})^\prime=F_n^\infty (\cV_{\cP_0}),
$$
where $^\prime$ stands for the commutant. According to
\cite{Po-varieties}, each  $\widetilde\chi\in F_n^\infty
(\cV_{\cP_0})$ generates a mapping $\chi:\cV_{\cP_0}(\cH)\to B(\cH)$
given by
$$
\chi(X_1,\ldots, X_n):={\bf P}_X[\widetilde\chi], \qquad
X:=(X_1,\ldots, X_n)\in \cV_{\cP_0}(\cH),
$$
where ${\bf P}_X$ is the noncommutative  Poisson transform
associated with $\cV_{\cP_0}(\cH)$. On the other hand, since
$\widetilde \chi=P_{\cN_{\cP_0}} \widetilde\phi |_{\cN_{\cP_0}}$ for
some $\widetilde\phi=\sum_{\alpha\in \FF_n^+} a_\alpha S_\alpha$ in
$F_n^\infty$, we have
$$
\chi(X_1,\ldots, X_n)=\sum_{k=0}^\infty \sum_{|\alpha|=k}
  a_\alpha X_\alpha,\qquad (X_1,\ldots, X_n)\in \cV_{\cP_0}(\cH),
  $$
  where the convergence is in the operator norm topology. This shows that $\chi$ is the restriction to
  $\cV_{\cP_0}(\cH)$ of a  bounded free holomorphic function on $[B(\cH)^n]_1$, namely
  $X\mapsto \phi (X)={\bf P}_X[\widetilde \psi]$.  We remark that the map
  $\chi$ does not depend on the choice of $\widetilde\phi\in F_n^\infty$  with the
  property
  that  $\widetilde
\chi=P_{\cN_{\cP_0}} \widetilde\phi |_{\cN_{\cP_0}}$. Note also
that $\chi(0)=\left< \widetilde\chi 1,1\right>$.

We remark that when $f\in F^2(H_n)$ and $f=\sum_{k=0}^\infty \sum_{|\alpha|=k} a_\alpha e_\alpha$,
then $f\in \cN_{\cP_0}$ if and only if $$\sum_{k=0}^\infty \sum_{|\alpha|=k} a_\alpha e_\alpha
=\sum_{k=0}^\infty \sum_{|\alpha|=k} a_\alpha B_\alpha 1.$$
We say that $\widetilde\psi\in F_n^\infty (\cV_{\cP_0})\otimes \CC^n$ is non-scalar operator if
 it does not have the form $(a_1I_{\cN_{\cP_0}},\ldots, a_nI_{\cN_{\cP_0}})$ for some
$a_i\in \CC$. The main result of this section is the following.

\bigskip
\bigskip

\begin{theorem}\label{comp-variety}
Let $\widetilde\psi=(\widetilde\psi_1,\ldots, \widetilde\psi_n)\in F_n^\infty (\cV_{\cP_0})\otimes \CC^n$
 be a non-scalar  operator  with $\|\widetilde\psi\|\leq 1$.
Then the following statements hold.
\begin{enumerate}
\item[(i)] If $g\in \cN_{\cP_0}$ has the representation\  $\sum_{k=0}^\infty \sum_{|\alpha|=k}
c_\alpha e_\alpha $   then
    $$
    g\circ \widetilde\psi:=\sum_{k=0}^\infty \sum_{|\alpha|=k} c_\alpha \widetilde\psi_\alpha
    1 \in \cN_{\cP_0},
    $$
      where the convergence of the series is in
    $F^2(H_n)$.
\item[(ii)] The composition operator $C_{\widetilde\psi} : \cN_{\cP_0}\to \cN_{\cP_0}$  defined by
$$C_{\widetilde\psi}g:=g\circ
\widetilde\psi, \qquad   g\in \cN_{\cP_0}, $$
 is bounded. Moreover,
$$
\|P_{\cN_{\cP_0}} z_{\psi(0)}\|
\leq
\sup_{\lambda\in
\cV_{\cP_0}(\CC)}  \frac{\|P_{\cN_{\cP_0}} z_{\psi(\mu)}\|}{\|z_\mu\|}
\leq
\|C_{\widetilde\psi}\|\leq
\left(\frac{1+\|\psi(0)\|}{1-\|\psi(0)\|}\right)^{1/2}.
$$
\item[(iii)] The adjoint of the composition operator $C_{\widetilde\psi} : \cN_{\cP_0}\to \cN_{\cP_0}$ satisfies the formula
    $$
    C_{\widetilde\psi}^* g=\sum_{k=0}^\infty \sum_{|\alpha|=k}
     \left<g, \widetilde \psi_\alpha (1)\right> P_{\cN_{\cP_0}}e_\alpha , \qquad g\in \cN_{\cP_0}.
    $$

    \end{enumerate}
\end{theorem}
\begin{proof}
Since $R_n^\infty (\cV_{\cP_0})^\prime=F_n^\infty (\cV_{\cP_0})$, the operator
 $\widetilde\psi:\cN_{\cP_0}\otimes \CC^n\to \cN_{\cP_0}$ satisfies the commutation
 relations
  $$
  \widetilde \psi (W_i\otimes I_{\CC^n})=W_i\widetilde\psi,\qquad i=1,\ldots, n.
  $$
Since $W_i:=P_{\cN_{\cP_0}} R_i|_{\cN_{\cP_0}}$, $ i=1,\ldots, n$,
it is clear that $[R_1\otimes I_{\CC_n},\ldots, R_1\otimes
I_{\CC_n}]$ is an isometric dilation of the row contraction
$[W_1\otimes I_{\CC_n},\ldots, W_1\otimes I_{\CC_n}]$. According to
the noncommutative commutant theorem \cite{Po-isometric}, there
exists $\widetilde \varphi=[\widetilde \varphi_1,\ldots, \widetilde
\varphi_n] : F^2(H_n)\otimes \CC^n\to F^2(H_n)$ with the properties
$\|\widetilde \varphi\|\leq 1$, $\widetilde
\varphi^*|_{\cN_{\cP_0}}=\widetilde \psi^*$, and $\widetilde
\varphi(R_i\otimes I_{\CC^n})=R_i \widetilde \varphi$ for
$i=1,\ldots,n$. Hence, we deduce that $\widetilde
\varphi_j^*|_{\cN_{\cP_0}}=\widetilde \psi_j^*$ and $\widetilde
\varphi_jR_i=R_i \widetilde \varphi_j$ for  $i,j=1,\ldots,n$. Since,
due to \cite{Po-analytic},  the commutant of the right creation
operators $R_1,\ldots, R_n$ coincides with the noncommutative
analytic Toeplitz algebra $F_n^\infty$, we deduce that $\widetilde
\varphi_j\in F_n^\infty$, $j=1,\ldots,n$. Since $\widetilde
\varphi^*|_{\cN_{\cP_0}}=\widetilde \psi^*$ and $\widetilde\psi$ is
a non-scalar  operator,  so is $\widetilde \varphi$.
 According to Theorem \ref{comp4} and Corollary \ref{comp5}, the composition operator
$C_{\widetilde\varphi}:F^2(H_n)\to F^2(H_n)$ satisfies the equation
\begin{equation} \label{C}
 C_{  \widetilde\varphi}\left(\sum_{k=0}^\infty \sum_{|\alpha|=k}a_\alpha e_\alpha\right)=
 \sum_{k=0}^\infty \sum_{|\alpha|=k} a_\alpha (\widetilde\varphi_\alpha
 1)
 \end{equation}
 for any $f=\sum_{k=0}^\infty \sum_{|\alpha|=k}a_\alpha e_\alpha$ in  $F^2(H_n)$.
Since $\widetilde \varphi_j^*|_{\cN_{\cP_0}}=\widetilde \psi_j^*$,
$j=1,\ldots,n$, we have $P_{\cN_{\cP_0}} \widetilde \varphi_\alpha
|_{\cN_{\cP_0}}=\widetilde \psi_\alpha$ for all $\alpha\in \FF_n^+$.
Since $1\in \cN_{\cP_0}$, we  assume that $f\in \cN_{\cP_0}$ in
 relation  \eqref{C} and, taking the projection on $\cN_{\cP_0}$,  we complete the proof of part (i).

Now,  to prove item (ii), note that part (i) implies  $C_{\widetilde\psi}=P_{\cN_{\cP_0}}
 C_{\widetilde\varphi}|_{\cN_{\cP_0}}$. Using  this relation and Theorem \ref{comp4}, we deduce that $\|C_{\widetilde\psi}\|\leq
\left(\frac{1+\|\psi(0)\|}{1-\|\psi(0)\|}\right)^{1/2}$. Recall that
$z_\lambda:=\sum_{\alpha\in \FF_n^+} \overline{\lambda}_\alpha
e_\alpha$, $\lambda\in \BB_n$. Note that if
$\lambda=(\lambda_1,\ldots, \lambda_n)$ is in the scalar
representation of the noncommutative variety $\cV_{\cP_0}$, i.e.,
$$
\cV_{\cP_0}(\CC):=\{(\lambda_1,\ldots, \lambda_n)\in \BB_n:\ p(\lambda_1,\ldots, \lambda_n)=0, p\in \cP_0\},
$$
then  we have
$$\left< [S_\alpha p(S_1,\ldots, S_n) S_\beta](1), z_\lambda \right>= {\lambda}_\alpha
 {p(\lambda)}  {\lambda}_\beta=0,
$$
for any $p\in \cP_0$  and $\alpha,\beta\in \FF_n^+$. Hence
$z_\lambda\in \cN_{\cP_0}$ for any $\lambda\in\cV_{\cP_0}(\CC)$. As
in the proof of Theorem \ref{comp2}, we have
\begin{equation*}
C_{\widetilde\varphi}^*z_\mu=\sum_{k=0}\sum_{|\alpha|=k}
\overline{\varphi_\alpha(\mu)}e_\alpha=z_{\varphi(\mu)},\qquad  \mu:=(\mu_1,\ldots, \mu_n)\in \BB_n,
\end{equation*}
Now, note that
 \begin{equation*}
\begin{split}
\|C_{\widetilde\psi}\| &=\|C_{\widetilde\psi}^*\|\geq \frac{\|C_{\widetilde\psi}^* z_\mu\|}{\|z_\mu\|}\\
&=\frac{\|P_{\cN_{\cP_0}} C_{\widetilde\varphi}^* z_\mu\|}{\|z_\mu\|}=
 \frac{\|P_{\cN_{\cP_0}} z_{\psi(\mu)}\|}{\|z_\mu\|}
\end{split}
\end{equation*}
for any $\lambda\in\cV_{\cP_0}(\CC)$. Since $0\in \cV_{\cP_0}(\CC)$ the first two inequalities in part (ii) follow.

Now, it remains to prove part (iii). According to Proposition \ref{adjoint}, we have
$$C_{\widetilde\psi}^* g=
P_{\cN_{\cP_0}}C_{\widetilde\varphi}^*g =\sum_{\alpha\in \FF_n^+} \left<
g,\widetilde\varphi_\alpha 1\right> P_{\cN_{\cP_0}}e_\alpha,\qquad g\in F^2 (H_n).
$$
Since $P_{\cN_{\cP_0}} \widetilde \varphi_\alpha |_{\cN_{\cP_0}}=\widetilde \psi_\alpha$ for all $\alpha\in \FF_n^+$ and  $1\in \cN_{\cP_0}$, we deduce part (iii).
The proof is complete.
\end{proof}

We remark that under the conditions of Theorem \ref{comp-variety},
we can use Theorem \ref{strict} to show that
$$
\|\psi(X_1,\ldots, X_n)\|<1,\qquad (X_1,\ldots, X_n)\in
\cV_{\cP_0}(\cH).
$$
Consequently, $ g\circ \widetilde\psi$ induces the map
$$
 (g\circ \psi)(X):= \sum_{k=0}^\infty \sum_{|\alpha|=k} c_\alpha
 \psi_\alpha (X),\qquad X\in \cV_{\cP_0}(\cH),
 $$
 where the convergence is in the operator norm topology. Using
 Corollary \ref{comp5}, we deduce that
 $$\lim_{r\to 1}(g\circ \psi)(rB_1,\ldots, rB_n) 1=g\circ
 \widetilde\psi.
 $$
Moreover, the map $g\circ \psi $ is
 the restriction  to $\cV_{\cP_0}(\cH)$ of the free holomorphic
 function $g\circ \varphi$ on $[B(\cH)^n]_1$, where $\varphi$ was introduced
 in the proof of Theorem \ref{comp-variety}.

\begin{corollary}  \label{spec-for}Let
 $\widetilde\psi=(\widetilde\psi_1,\ldots, \widetilde\psi_n)\in F_n^\infty (\cV_{\cP_0})\otimes \CC^n$
  be a non-scalar operator  with $\|\widetilde\psi\|\leq 1$ and $p(\psi(0))=0$
   for all $p\in \cP_0$. Then the  norm of composition operator
   $C_{\widetilde\psi} : \cN_{\cP_0}\to \cN_{\cP_0}$ satisfies the inequalities
$$
\frac{1}{(1-\|\psi(0)\|^2)^{1/2}}
\leq
\|C_{\widetilde\psi}\|\leq
\left(\frac{1+\|\psi(0)\|}{1-\|\psi(0)\|}\right)^{1/2}.
$$
Moreover, the spectral radius of $C_{\widetilde\psi}$ satisfies the relation
$$r(C_{\widetilde\psi})=\lim_{k\to\infty}(1-\|\varphi^{[k]}(0)\|)^{-1/2k}. $$
\end{corollary}
\begin{proof} Since $p(\psi(0))=0$
   for all $p\in \cP_0$, we have $\psi(0)\in \cV_{\cP_o}(\CC)$ and,
   as in the proof of Theorem \ref{comp-variety}, we deduce that
   $z_{\psi(0)}\in \cN_{\cP_0}$. Consequently,
   $$
   \|P_{\cN_{\cP_0}}z_{\psi(0)}\|=\|z_{\psi(0)}\|=\frac{1}{(1-\|\psi(0)\|^2)^{1/2}}.
   $$
   Combining this  relation with part (ii) of Theorem \ref{comp-variety},  we deduce the
   inequalities above. The proof of the last part of this corollary is similar to the
   proof of Theorem \ref{spectral radius}.
\end{proof}

Now we consider an important particular case.
If
$\cP_c:=\{X_iX_j-X_jX_i:\ i,j=1,\ldots, n\}$, then $\cN_{\cP_c}=\overline{\text{\rm span}}\{z_\lambda: \ \lambda\in
\BB_n\}=F^2_s$, the symmetric Fock space.
For each $\lambda=(\lambda_1,\ldots,
\lambda_n)$ and each $n$-tuple ${\bf k}:=(k_1,\ldots, k_n)\in
\NN_0^n$, where $\NN_0:=\{0,1,\ldots \}$, let $\lambda^{\bf
k}:=\lambda_1^{k_1}\cdots \lambda_n^{k_n}$. For each ${\bf k}\in
\NN_0^n$, we denote
$$
\Lambda_{\bf k}:=\{\alpha\in \FF_n^+: \ \lambda_\alpha =\lambda^{\bf
k} \text{ for all } \lambda\in \CC^n\}
$$
and define the vector
$$
w^{\bf k}:=\frac{1}{\gamma_{\bf k}} \sum_{\alpha \in \Lambda_{\bf
k}}  e_\alpha\in F^2(H_n), \quad  \text{ where } \
\gamma_{\bf k}:=\text{\rm card}\, \Lambda_{\bf k}.
$$
 The set  $\{w^{\bf k}:\ {\bf
k}\in \NN_0^n\}$ consists  of orthogonal vectors in $F^2(H_n)$ which span the symmetric Fock space $F^2_s$ and
$\|w^{\bf k}\|=\frac{1}{\sqrt{\gamma_{\bf k}}}$. The symmetric  Fock space $F_s^2$ can be identified with
the Drury-Arveson
 space ${\bf H}_n^2$ of all functions
$\varphi:\BB_n\to \CC$ which admit a power series representation
$\varphi(\lambda)=\sum_{{\bf k}\in \NN_0} c_{\bf k} \lambda^{\bf k}$
with
$$
\|\varphi\|_2=\sum_{{\bf k}\in \NN_0}|c_{\bf
k}|^2\frac{1}{\gamma_{\bf k}}<\infty.
$$
More precisely, every  element  $\varphi=\sum_{{\bf k}\in \NN_0}
c_{\bf k} w^{\bf k}$ in $F_s^2$  has a functional
representation on $\BB_n$ given by
\begin{equation} \label{fila}
\varphi(\lambda):=\left<\varphi, z_\lambda\right>=\sum_{{\bf k}\in
\NN_0} c_{\bf k} \lambda^{\bf k}, \qquad \lambda=(\lambda_1,\ldots,
\lambda_n)\in \BB_n,
\end{equation}
and
$$
|\varphi(\lambda)|\leq \frac{\|\varphi\|_2}{\sqrt{1-
\|\lambda\|^2}},\qquad \lambda=(\lambda_1,\ldots, \lambda_n)\in
\BB_n.
$$
 Arveson showed  that the algebra
$F_n^\infty (\cV_{\cP_c})$
      can be identified with  the algebra of all  multipliers  of ${\bf H}_n^2$.
       Under this identification the creation operators $L_i:=P_{F_s^2} S_i|_{F_s^2}$, $i=1,\dots, n$,  on the symmetric Fock space become the multiplication operators $M_{z_1},\ldots, M_{z_n}$ by the coordinate functions $z_1,\ldots, z_n$ of $\CC^n$.

\begin{theorem}\label{symmetric} Let
 $\widetilde\psi=(\widetilde\psi_1,\ldots, \widetilde\psi_n)\in F_n^\infty (\cV_{\cP_c})\otimes \CC^n$
 be a non-scalar operator  with $\|\widetilde\psi\|\leq 1$. Under
 the identification of the symmetric Fock space $F_s^2$ with the  Drury-Arveson space
   ${\bf H}_n^2$, the composition operator $C_{\widetilde\psi} :F_s^2\to F_s^2$
has the functional representation
$$
(C_{\widetilde\psi} f)(\lambda)=f(\psi(\lambda)), \qquad \lambda\in \BB_n.
$$
Moreover, if $f\in F^2_s$, then
$$
(C_{\widetilde\psi}^* f)(\lambda)=\left<f, z_\lambda\circ\widetilde\psi\right>, \qquad \lambda\in \BB_n,
$$
where $z_\lambda:=\sum_{\alpha\in \FF_n^+} \overline{\lambda}_\alpha e_\alpha$.
\end{theorem}

\begin{proof} As in the proof of Theorem \ref{comp-variety},
due  to the noncommutative commutant lifting theorem, there is
$\widetilde\varphi=(\widetilde\varphi_1,\ldots,
\widetilde\varphi_n)\in F_n^\infty \otimes \CC^n$  a non-scalar
 operator  with $\|\widetilde\varphi\|\leq 1$, such
that $\widetilde\varphi_i^*|_{F^2_s}=\widetilde\psi_i^*$,
$i=1,\ldots,n$. In particular,  due to \eqref{fila}, we have
$\varphi(\lambda)=\psi(\lambda)$, $\lambda\in \BB_n$. Fix
$f=\sum_{\alpha\in \FF_n^+} a_\alpha e_\alpha \in F^2_s$ and
$\lambda\in \BB_n$. Since $z_\lambda\in F^2_s$ and $P_{F^2_s}
\widetilde \varphi_\alpha|_{F^2_s}=\widetilde \psi_\alpha$ for all
$\alpha\in \FF_n^+$, we can  use  relations \eqref{fila},
\eqref{CZ}, as well as Corollary \ref{comp5} and Theorem
\ref{comp-variety}, to obtain
\begin{equation*}
\begin{split}
f(\psi(\lambda))&=\left< f, z_{\psi(\lambda)}\right>=\left< f, z_{\varphi(\lambda)}\right>=\left< f, C_{\widetilde \varphi}^*z_\lambda
\right>\\
&=\left< C_{\widetilde \varphi}f, z_\lambda
\right>=\left< \sum_{\alpha\in \FF_n^+} a_\alpha \widetilde \varphi_\alpha 1, z_\lambda
\right>\\
&=\left< \sum_{\alpha\in \FF_n^+} a_\alpha  P_{F^2_s}\widetilde\varphi_\alpha 1, z_\lambda
\right>=\left< \sum_{\alpha\in \FF_n^+} a_\alpha \widetilde \psi_\alpha 1, z_\lambda
\right>\\
&=\left< C_{\widetilde \psi}f, z_\lambda\right>=(C_{\widetilde\psi}
f)(\lambda).
\end{split}
\end{equation*}
Therefore, the first part of the theorem holds. To prove the second
part, note that
 according to item
(iii) of Theorem \ref{comp-variety}, we have
\begin{equation}\label{C*f}
    C_{\widetilde\psi}^* f=\sum_{k=0}^\infty \sum_{|\alpha|=k}
     \left<f, \widetilde \psi_\alpha (1)\right> P_{F_s^2} e_\alpha, \qquad f\in F^2_s.
    \end{equation}
On the other hand, since  $z_\lambda\in F^2_s$,  part (i) of Theorem
\ref{comp-variety} implies $z_\lambda\circ \widetilde \psi\in F^2_s$
and
$$z_\lambda\circ \widetilde \psi=\sum_{k=0}^\infty \sum_{|\alpha|=k} \overline{\lambda}_\alpha
\widetilde \psi_\alpha 1,$$ where the convergence is in $F^2(H_n)$.
Consequently,  using relations \eqref{C*f} and \eqref{fila}, we
deduce that
\begin{equation*}
\begin{split}
\left<f, z_\lambda\circ\widetilde\psi\right> &= \left<f,
\sum_{k=0}^\infty \sum_{|\alpha|=k} \overline{\lambda}_\alpha
\widetilde \psi_\alpha 1\right> =
\sum_{k=0}^\infty \sum_{|\alpha|=k}\left< f, \widetilde \psi_\alpha 1\right> \lambda_\alpha\\
&=\left< \sum_{k=0}^\infty \sum_{|\alpha|=k} \left<f, \widetilde \psi_\alpha (1)\right>
e_\alpha, z_\lambda\right>\\
&=(C_{\widetilde\psi}^* f)(\lambda)
\end{split}
\end{equation*}
for any $\lambda\in \BB_n$.
The proof is complete.
\end{proof}

Since $\psi(\lambda)\in \cV_{\cP_c}$ for all $\lambda\in \BB_n$ part
(ii) of Theorem \ref{comp-variety} implies the following result
concerning the composition operators on the symmetric Fock space
$F_s^2$ and, consequently,  on the Drury-Arveson space ${\bf
H}_n^2$. The next result was obtained by Jury (\cite{Ju}) using
different methods.

\begin{corollary}\label{comp6}
Let
 $\widetilde\psi=(\widetilde\psi_1,\ldots, \widetilde\psi_n)\in F_n^\infty (\cV_{\cP_c})\otimes \CC^n$
 be a non-scalar operator  with $\|\widetilde\psi\|\leq 1$. Then
 the composition operator $C_{\widetilde\psi} :F_s^2\to F_s^2$ is
  bounded  and
$$
\frac{1}{(1-\|\psi(0)\|^2)^{1/2}}\leq \sup_{\lambda\in
\BB_n}\left(\frac{1-\|\lambda\|^2}{1-\|\psi(\lambda)\|^2}\right)^{1/2}\leq
\|C_{\widetilde\psi}\|\leq
\left(\frac{1+\|\psi(0)\|}{1-\|\psi(0)\|}\right)^{1/2}.
$$
\end{corollary}
It is obvious now that the  formula for the spectral radius of
$C_{\widetilde\psi}$   (see Corollary \ref{spec-for}) holds. We also
remark that  one can deduce commutative versions of Corollary
\ref{contr}, Theorem \ref{similarity}, and Corollary \ref{sp1}. We
leave this task to the reader.

\bigskip

       %

      \end{document}